\nonstopmode
\documentclass[10pt, reqno]{amsart}
\usepackage{graphicx}
\usepackage{latexsym}
\usepackage{fancyhdr}
\usepackage{amsmath, amssymb, amsthm}
\usepackage[all]{xy}
\usepackage{pdflscape}
\usepackage{longtable}
\usepackage{rotating}
\usepackage{verbatim}
\usepackage{hyperref}
\hypersetup{
    linkcolor = blue,
    citecolor = blue,
    urlcolor  = blue,
    colorlinks = true,
}

\usepackage{cleveref}
\usepackage{subfigure}
\usepackage{mathrsfs}
\usepackage{mdwlist}
\usepackage{dsfont}
\usepackage{mathtools}
\usepackage{float}
\usepackage{color}
\usepackage{stmaryrd}

\usepackage{tkz-base}
\usepackage{pgfplots}
\pgfplotsset{compat=newest}

\usepackage{etoolbox}

\definecolor{teal}{rgb}{0.0, 0.5, 0.5}

\newcounter{mnotecount}[section]

\newcommand{\rmnote}[1]{}

\allowdisplaybreaks

\DeclareFontFamily{U}{mathb}{\hyphenchar\font45}
\DeclareFontShape{U}{mathb}{m}{n}{
      <5> <6> <7> <8> <9> <10> gen * mathb
      <10.95> mathb10 <12> <14.4> <17.28> <20.74> <24.88> mathb12
      }{}
\DeclareSymbolFont{mathb}{U}{mathb}{m}{n}
\DeclareFontSubstitution{U}{mathb}{m}{n}

\theoremstyle{plain}
\newtheorem*{theorem*}{Theorem}
\newtheorem{theorem}{Theorem}[section]
\newtheorem*{lemma*}{Lemma}
\newtheorem{lemma}[theorem]{Lemma}
\newtheorem*{assumption*}{Assumption}

\newtheorem*{proposition*}{Proposition}
\newtheorem{proposition}[theorem]{Proposition}
\newtheorem*{corollary*}{Corollary}
\newtheorem{corollary}[theorem]{Corollary}
\newtheorem*{claim*}{Claim}

\newtheorem*{conjecture*}{Conjecture}

\newtheorem*{result*}{Result}

\theoremstyle{definition}
\newtheorem*{definition*}{Definition}
\newtheorem{definition}[theorem]{Definition}
\newtheorem*{example*}{Example}
\newtheorem{example}[theorem]{Example}

\newtheorem*{algorithm*}{Algorithm}
\newtheorem*{remark*}{Remark}
\newtheorem*{remarks*}{Remarks}
\newtheorem{remark}[theorem]{Remark}

\newtheorem*{question*}{Question}
\newtheorem{convention}[theorem]{Convention}
\newtheorem*{convention*}{Convention}

\numberwithin{equation}{section}

\newcommand{\Iq}{{\mathcal{A}}_d (\C)}
\def\a#1{\left\llbracket{#1}\right\rrbracket}
\def\bth#1{\left[{#1}\right]_\th}
\newcommand{\abs}[1]{\left|#1\right|}

\newcommand{\ra}{\rightarrow}

\def\Et{H}

\def\al{\alpha}
\def\be{\beta}
\def\ga{\gamma}
\def\de{\delta}
\def\ep{\epsilon}

\def\et{\eta}
\def\th{\theta}

\def\ka{\kappa}
\def\la{\lambda}

\def\rh{\rho}

\def\si{\sigma}

\def\ta{\tau}

\def\vh{\varphi}
\def\ch{\chi}
\def\ps{\psi}

\def\Ga{\Gamma}
\def\De{\Delta}

\def\La{\Lambda}

\def\Om{\Omega}

\def\C{\mathbb{C}}

\def\N{\mathbb{N}}

\def\R{\mathbb{R}}

\def\cA{\mathcal{A}}
\def\cB{\mathcal{B}}

\def\cE{\mathcal{E}}
\def\cF{\mathcal{F}}

\def\cI{\mathcal{I}}
\def\cJ{\mathcal{J}}

\def\cL{\mathcal{L}}

\def\sP{\mathscr{P}}

\def\p{\partial}
\def\id{\on{id}}

\renewcommand{\Re}{\mathrm{Re}}
\renewcommand{\Im}{\mathrm{Im}}
\def\<{\langle}
\def\>{\rangle}
\renewcommand{\o}{\circ}

\def\ol{\overline}
\def\ul{\underline}

\def\dd{\mathbf{d}}
\def\ddd{\widehat \dd}
\def\bs{\mathbf{s}}

\let\on=\operatorname

\newcommand{\sr}[1]%
{\ifmmode{}^\dagger\else${}^\dagger$\fi\ifvmode
\vbox to 0pt{\vss
 \hbox to 0pt{\hskip\hsize\hskip1em
 \vbox{\hsize3cm\raggedright\pretolerance10000
 \noindent #1\hfill}\hss}\vss}\else
 \vadjust{\vbox to0pt{\vss%
 \hbox to 0pt{\hskip\hsize\hskip1em%
 \vbox{\hsize3cm\raggedright\pretolerance10000%
 \noindent #1\hfill}\hss}\vss}}\fi%
}

\providecommand{\mapsfrom}{\kern.2em%
\setbox0=\hbox{$\leftarrow$\kern-.10em\rule[0.26mm]{0.1mm}{1.3mm}}\box0%
\kern.3em}

\title{On the continuity of the solution map for polynomials}

\author[Adam Parusi\'nski and  Armin Rainer]
{Adam Parusi\'nski and Armin Rainer}

\address {Adam Parusi\'nski: Universit\'e C\^ote d'Azur,  CNRS,  LJAD, UMR 7351, 06108 Nice, France}
\email{adam.parusinski@univ-cotedazur.fr}

\address{Armin Rainer: Fakult\"at f\"ur Mathematik, Universit\"at Wien,
Oskar-Morgenstern-Platz~1, A-1090 Wien, Austria}
\email{armin.rainer@univie.ac.at}

\begin{document}

\begin{abstract}
    In previous work, we proved that the continuous roots of a monic polynomial of degree $d$ 
    whose coefficients depend in a $C^{d-1,1}$ way on real parameters 
    belong to the Sobolev space $W^{1,q}$
    for all $1\le q<d/(d-1)$. 
    This is optimal. 
    We obtained uniform bounds that show that the solution map ``coefficients-to-roots''
    is bounded with respect to the $C^{d-1,1}$ and the Sobolev $W^{1,q}$ structures on source and target space, respectively. 
    In this paper, we prove that the solution map is continuous, provided that
    we consider the $C^d$ structure
    on the space of coefficients. 
    Since there is no canonical choice of an ordered $d$-tuple of the roots,
    we work in the space of $d$-valued Sobolev functions
    equipped with a strong notion of convergence.
    We also interpret the results in the Wasserstein space on the complex plane.
\end{abstract}

\thanks{This research was funded in whole or in part by the Austrian Science Fund (FWF) DOI 10.55776/P32905.
For open access purposes, the authors have applied a CC BY public copyright license to any author-accepted manuscript version arising from this submission.}
\date{\today}

\maketitle

\setcounter{tocdepth}{1}
\tableofcontents


\section{Introduction}

Consider a monic polynomial of degree $d$,
\[
    P_a(Z) = Z^d + \sum_{j=1}^d a_j Z^{d-j},
\]
where the coefficients $a_j$, for $1\le j \le d$, are complex valued functions 
defined on a bounded open interval $I \subseteq \R$.
Given that the coefficients are smooth, it is natural to ask how regular the roots of $P_a$ 
can be.

This question was answered in \cite{ParusinskiRainerAC}, \cite{ParusinskiRainer15}, and \cite{Parusinski:2020aa} (see also \Cref{thm:optimal}): 
let $\la : I \to \C$ be a continuous root of $P_a$, i.e., $P_{a(x)}(\la(x)) = 0$ for all $x \in I$.
If $a = (a_1,\ldots,a_d) \in C^{d-1,1}(\ol I,\C^d)$, then $\la$ is absolutely continuous and bounded on $I$ and $\la' \in L^q(I)$, in particular, 
$\la  \in W^{1,q}(I)$, 
for all $1 \le q <d/(d-1)$. Moreover, there is a uniform bound for the $L^q$ norm of $\la'$ in terms of the $C^{d-1,1}$ norm of $a$ (see \eqref{eq:optimal}).
This result is optimal in the sense that there are examples of 
$a \in C^\infty(I,\C^d) \cap \bigcap_{0<\ga < 1} C^{d-1,\ga}(\ol I,\C^d)$ such that  
no root of $P_a$ has bounded variation on $I$
and polynomial curves $a: \R \to \C^d$ such that no root of $P_a$ has derivative in $L^{d/(d-1)}(I)$.

Even though there always exists a continuous parameterization $\la = (\la_1,\ldots,\la_d) : I \to \C^d$ 
of the roots of $P_a$, i.e., $P_{a(x)}(Z) = \prod_{j=1}^d (Z-\la_j(x))$ for all $x \in I$,\footnote{The roots of a monic polynomial $P_a$ with $a \in C^0(I,\C^d)$ admit a continuous parameterization $\la : I \to \C^d$; see \cite[II.5.2]{Kato76}.\label{fn:Kato}} 
there is in general no canonical choice of a continuous ordered $d$-tuple of the roots. 
(The situation is different for hyperbolic polynomials, see \Cref{ssec:hyperbolic}.)
But we may consider the unordered $d$-tuple $\La = [\la_1,\ldots,\la_d]$ of roots and thus obtain a continuous curve $\La : I \to \cA_d(\C)$ 
in the complete metric space $(\cA_d(\C),\dd)$ of unordered tuples of $d$ complex numbers (see \Cref{lem:homeo}). 
We refer to \Cref{ssec:AdC} for the definition of $\dd$, but it is worth mentioning that $\cA_d(\C)$ 
can naturally be identified with a subset of the set $\sP(\C)$ of probability measures on $\C$ (sending $[z_1,\ldots,z_d]$ to the 
formal sum $\sum_{j=1}^d \a{z_j}$ of Dirac delta measures at $z_j$) 
and then $\dd$ is induced by the $2$-Wasserstein distance on $\sP(\C)$; see \Cref{ssec:probability}.

With this terminology the above result can be interpreted as follows: 
if $a \in C^{d-1,1}(\ol I,\C^d)$, then $\La \in W^{1,q}(I, \cA_d(\C))$, for all $1 \le q < d/(d-1)$, 
and the map $a \mapsto \La$ takes bounded sets to bounded sets.
Here $W^{1,q}(I,\Iq)$ denotes the space
of $d$-valued Sobolev functions 
(see \eqref{eq:dSobintro} and \Cref{sec:dSob}).  
The boundedness of the map $a \mapsto \La$ follows from \Cref{cor:aLabd}.

In the present paper, we address the natural question if the map $a \mapsto \La$ is continuous.
We prove that this is true
with respect to the $C^d$ topology on the space of coefficients $a$
and various natural structures on the target space $W^{1,q}(I,\cA_d(\C))$ (for $1 \le q < d/(d-1)$).
These results will lead to multiparameter versions by a sectioning argument.

\subsection{The main results}

Due to Almgren \cite{Almgren00},
there exists a bi-Lipschitz embedding $\De : \Iq \to \R^N$, where $N =N(d)$.
Almgren used $\De$ to define Sobolev spaces of $\cA_d(\C)$-valued functions: 
for open $U \subseteq \R^m$ and $1 \le q \le \infty$ set 
\begin{equation} \label{eq:dSobintro}
    W^{1,q}(U,\cA_d(\C)) := \{f : U \to \cA_d(\C) : \De \o f \in W^{1,q}(U,\R^N)\}.  
\end{equation}
An equivalent intrinsic definition of $W^{1,q}(U,\cA_d(\C))$ 
is due to De Lellis and Spadaro \cite{De-LellisSpadaro11}.\footnote{Actually, in \cite{Almgren00} and \cite{De-LellisSpadaro11} the theory is developed for $\cA_d(\R^n)$-valued functions, 
where $\cA_d(\R^n)$ is the space of unordered $d$-tuples of vectors in $\R^n$. In this paper, we stick to the case $n=2$ and identify $\C =\R^2$.}
Then $W^{1,q}(U,\cA_d(\C))$ carries the metric
\begin{equation} \label{eq:Almgrenintro}
    (f,g) \mapsto \|\De \o f - \De \o g\|_{W^{1,q}(U,\R^N)}
\end{equation}
which makes it a complete metric space; see \Cref{lem:Acomplete}.

Let us first assume that $m=1$ and $U$ is a bounded open interval.
We will generally assume that the degree $d$ is at least $2$, since for $d=1$ all results are trivially true.

\begin{theorem} \label{thm:mainAlmgren}
    Let $d\ge 2$ be an integer. 
    Let $I \subseteq \R$ be a bounded open interval.
    Let $a_n \to a$ in $C^d(\ol I,\C^d)$, i.e.,  
    \begin{equation*}
        \|a-a_n\|_{C^d(\ol I,\C^d)} \to 0 \quad \text{ as } n \to \infty.
    \end{equation*}
    Let $\La,\La_n : I \to \cA_d(\C)$ be the curves of unordered roots of $P_a$, $P_{a_n}$, respectively.
    Then
    \begin{equation*}
        \|\De \o \La - \De \o \La_n\|_{W^{1,q}(I,\R^N)} \to 0 \quad \text{ as } n \to \infty,
    \end{equation*}
    for all $1 \le q < d/(d-1)$.
\end{theorem}

\Cref{thm:mainAlmgren} is true for all Almgren embeddings $\De$; see \Cref{def:Amap}. 
We will not work directly with an Almgren embedding but (inspired by the intrinsic definition of $d$-valued Sobolev functions of \cite{De-LellisSpadaro11})
introduce a semimetric $\dd^{1,q}_{I}$ on $W^{1,q}(I,\cA_d(\C))$, 
without reference to any Almgren embedding, that generates the same topology as the metric \eqref{eq:Almgrenintro}.

\begin{theorem} \label{thm:main1}
    Let $d\ge 2$ be an integer.
    Let $I \subseteq \R$ be a bounded open interval.
    Let $a_n \to a$ in $C^d(\ol I,\C^d)$ as $n\to \infty$.  
    Let $\La,\La_n : I \to \cA_d(\C)$ be the curves of unordered roots of $P_a$, $P_{a_n}$, respectively.
    Then 
    \begin{equation*}
        \dd^{1,q}_I(\La,\La_n)   \to 0 \quad \text{ as } n \to \infty,
    \end{equation*}
    for all $1 \le q < d/(d-1)$.
\end{theorem}

The majority of the paper is dedicated to the proof of \Cref{thm:main1} which will be 
completed in \Cref{sec:proofs1}. 
We will see in \Cref{thm:Almgren and convergence} that the conclusions of \Cref{thm:mainAlmgren} and \Cref{thm:main1} are equivalent.

Let $|\dot \La|$ denote the metric speed and $\cE_q(\La)$ the $q$-energy of the curve $\La \in AC^q(I, \cA_d(\C))$; see \Cref{ssec:ACq} for definitions.
As a consequence of \Cref{thm:main1}, we obtain the following.

\begin{theorem} \label{thm:mainse}
    Let $d\ge 2$ be an integer.
    Let $I \subseteq \R$ be a bounded open interval.
    Let $a_n \to a$ in $C^d(\ol I,\C^d)$ as $n\to \infty$.  
    Let $\La,\La_n : I \to \cA_d(\C)$ be the curves of unordered roots of $P_a$, $P_{a_n}$, respectively.
    Then
    \begin{align*}
        \|\dd(\La,\La_n) \|_{L^\infty(I)}  &\to 0 \quad \text{ as } n \to \infty,
        \\
        \big\| |\dot\La|- |\dot \La_n| \big\|_{L^q(I)} &\to 0 \quad \text{ as } n \to \infty,
        \\
        \big| \cE_q(\La)- \cE_q(\La_n) \big| &\to 0 \quad \text{ as } n \to \infty,
    \end{align*}
    for all $1 \le q < d/(d-1)$.
\end{theorem}

Note that there always exist continuous parameterizations 
$\la,\la_n : I \to \C^d$ of the roots of $P_a$, $P_{a_n}$, respectively, 
i.e., $\La = [\la]$ and $\La_n = [\la_n]$ (see \Cref{fn:Kato}).
By \Cref{thm:optimal}, it follows that $\la,\la_n \in W^{1,q}(I, \C^d)$ for all $1 \le q < d/(d-1)$.

The next corollary is an easy consequence of \Cref{thm:main1}, as shown in \Cref{ssec:proofcormain1}.

\begin{corollary} \label{cor:main1}
    Let $d\ge 2$ be an integer.
    Let $I \subseteq \R$ be a bounded open interval.
    Let $a_n \to a$ in $C^d(\ol I,\C^d)$ as $n\to \infty$.  
    Let $\la,\la_n : I \to \C^d$ be  continuous parameterizations of the roots of $P_a$, $P_{a_n}$, respectively.
    Then
    \begin{align*} 
        \big\| \|\la'\|_2 - \|\la_n'\|_2 \big\|_{L^q(I)}\to 0 \quad \text{ as } n \to \infty,
        \\
        \|\la_n'\|_{L^q(I,\C^d)} \to \|\la'\|_{L^q(I,\C^d)} \quad \text{ as } n \to \infty,
    \end{align*}
    for all $1 \le q < d/(d-1)$.
\end{corollary}

We shall see in \Cref{sec:Wasserstein} 
that \Cref{cor:main1} implies \Cref{thm:mainse}. 
Note that \Cref{thm:contse} is an interpretation of \Cref{thm:mainse} in the Wasserstein space on $\C$.

Since the components of $\la_n$ and $\la$ are absolutely continuous, 
\Cref{cor:main1} for $q=1$ immediately gives the following consequence.

\begin{corollary} \label{cor:main12}
    In the setting of \Cref{cor:main1},
    \[
        \on{length}(\la_n) \to \on{length}(\la) \quad \text{ as } n \to \infty.
    \]
\end{corollary}

We have the following variant of \Cref{thm:main1}.

\begin{theorem} \label{thm:main1var}
    Let $d\ge 2$ be an integer.
    Let $I \subseteq \R$ be a bounded open interval.
    Let $a_n \to a$ in $C^d(\ol I,\C^d)$ as $n\to \infty$.  
    Assume that $\la_n : I \to \C^d$ is a continuous parameterization of the roots of $P_{a_n}$
    and that $\la_n$ converges in $C^0(\ol I,\C^d)$ to a continuous parameterization $\la$ of the roots of $P_a$.
    Then 
    \[
        \|\la' - \la_n'\|_{L^q(I,\C^d)} \to 0 \quad \text{ as } n \to \infty,
    \]
    for all $1 \le q <d/(d-1)$. 
\end{theorem}

It is clear that the conclusion of \Cref{cor:main1} holds in this setting.

Note that not every continuous parameterization $\la$ of the roots of $P_a$ 
is the limit of continuous parameterizations $\la_n$ of the roots of $P_{a_n}$; see \Cref{ex:lifts}.

\begin{remark}
    In all our results, we require convergence of the coefficient vectors in $C^d(\ol I,\C^d)$, not just in $C^{d-1,1}(\ol I,\C^d)$.
    This seems natural (in view of a continuity instead of a boundedness result) but we do not know if it is necessary. See \Cref{ssec:optimality}. 
\end{remark}

We expect that our results have generalizations to $d$-degree algebraic hypersurfaces in $\C\mathbb P^n$, 
in the spirit of \cite{Antonini:2022aa}.
We hope to address this in a future work.

\subsection{Multiparameter versions}

Let $\De : \Iq \to \R^N$ be an Almgren embedding.

\begin{theorem} \label{thm:mainAlmgrenmult}
    Let $d\ge 2$ be an integer.
    Let $U \subseteq \R^m$ be a bounded open box, $U=I_1 \times \cdots \times I_m$.
    Let $a_n \to a$ in $C^d(\ol U,\C^d)$, i.e.,  
    \begin{equation*}
        \|a-a_n\|_{C^d(\ol U,\C^d)} \to 0 \quad \text{ as } n \to \infty.
    \end{equation*}
    Let $\La,\La_n : U \to \cA_d(\C)$ be the maps of unordered roots of $P_a$, $P_{a_n}$, respectively.
    Then
    \begin{equation*}
        \|\De \o \La - \De \o \La_n\|_{W^{1,q}(U,\R^N)} \to 0 \quad \text{ as } n \to \infty,
    \end{equation*}
    for all $1 \le q < d/(d-1)$.
\end{theorem}

It should be added that the maps $\La,\La_n : U \to \cA_d(\C)$ are continuous (see \Cref{lem:homeo}) and that we even have uniform convergence $\De \o \La_n \to \De \o \La$ on $U$.
\Cref{thm:mainAlmgrenmult} will be proved in \Cref{sec:proofs2}.

As a consequence, we immediately get a solution of \cite[Open Problem 4.8]{Parusinski:2023ab}: 

\begin{corollary} \label{cor:OP4.8}
    Let $U \subseteq \R^m$ be open.
    For all $1 \le q <d/(d-1)$,
    the ``coefficients-to-roots'' map 
    \[
        C^d(U,\C^d) \to W^{1,q}_{\on{loc}}(U,\cA_d(\C)), \quad a \mapsto \La,
    \]
    is continuous with respect to the topology induced by \eqref{eq:Almgrenintro} for all relatively compact open 
    subsets in $U$.
\end{corollary}

We will see in \Cref{thm:topmult} that the conclusions of 
\Cref{thm:mainAlmgrenmult} and \Cref{cor:OP4.8} 
are independent of the choice of the Almgren embedding $\De$.

\begin{remark}
    It is possible to obtain multiparameter versions of \Cref{thm:main1}, \Cref{thm:mainse}, \Cref{cor:main1}, and 
    \Cref{thm:main1var}, by working with suitable multivariate definitions and adjusting the sectioning argument 
    in the proof of \Cref{thm:mainAlmgrenmult}.
    This will be demonstrated in \Cref{thm:multex}.

    However, note that  continuous parameterizations of the roots might not always exist (even locally) if the parameter 
    space is at least $2$-dimensional because of monodromy. 
    Nevertheless, due to \cite{Parusinski:2020aa}, 
    there always exist parameterizations of the roots by functions of bounded variation.
\end{remark}

\subsection{Hyperbolic polynomials} \label{ssec:hyperbolic}

Let us briefly comment on the case of hyperbolic polynomials, in which 
canonical choices of continuous parameterizations of the roots exist and stronger results hold true.
We refer to \cite{Parusinski:2024aa}. A monic polynomial $P_a$ of degree $d$ is called \emph{hyperbolic} if
all its $d$ roots (counted with multiplicities) are real. The space $\on{Hyp}(d)$ of monic hyperbolic polynomials of degree $d$
can be identified with a semialgebraic subset of $\R^d$ (via the coefficient vector $a$). 
Ordering the roots of $P_a \in \on{Hyp}(d)$ increasingly, 
induces a continuous solution map 
\[
    \la^\uparrow =(\la^\uparrow_1,\ldots,\la^\uparrow_d) : \on{Hyp}(d) \to \R^d,
\]
where $\la^\uparrow_1 \le \la^\uparrow_2 \le \cdots \le \la^\uparrow_d$.
Bronshtein's theorem \cite{Bronshtein79} (see also \cite{ParusinskiRainerHyp}) states that 
\begin{align*}
    (\la^\uparrow)_* : C^{d-1,1}(U, \on{Hyp}(d)) &\to C^{0,1}(U,\R^d), 
    \\
    (x \mapsto P_{a(x)}) &\mapsto (x \mapsto \la^\uparrow (P_{a(x)})),
\end{align*}
is well-defined and bounded, where $U \subseteq \R^m$ is open.
Hereby the space $C^{d-1,1}(U, \on{Hyp}(d)) = \{f \in C^{d-1,1}(U,\R^d) : f(U) \subseteq \on{Hyp}(d)\}$ carries the trace topology of the Fr\'echet topology of $C^{d-1,1}(U, \R^d)$.

\begin{theorem}[\cite{Parusinski:2024aa}] \label{thm:hyp}
   The map  
   $(\la^\uparrow)_* : C^{d}(U, \on{Hyp}(d)) \to W^{1,q}_{\on{loc}}(U,\R^d)$ is continuous, for all $1 \le q<\infty$.
\end{theorem}

But $(\la^\uparrow)_* : C^{d}(U, \on{Hyp}(d)) \to C^{0,1}(U,\R^d)$ is not continuous as shown by an example in \cite{Parusinski:2024aa}.

In the real case, the map $(\cdot)^\uparrow : \cA_d(\R) \to \R^d$ (that orders the coordinates increasingly) is a Lipschitz right-inverse of $[\cdot] : \R^d \to \cA_d(\R)$, 
thus a canonical version of an Almgren embedding.

The proof of \Cref{thm:hyp} follows the same general strategy as the one of \Cref{thm:main1} but it is much simpler.

\subsection{Outline of the proof of \Cref{thm:main1}}

We first give a full proof of the special case of radicals and then tackle the general case, using the result for radicals.

\subsubsection*{Radical case}

Here the polynomials take the simple form
\[
    Z^d = g \quad \text{ and } \quad Z^d = g_n, \quad n\ge1,
\]
where we assume that $g_n \to g$ in $C^d(\ol I,\C)$ as $n \to \infty$.
Let $\la,\la_n : I \to \C$ be continuous functions satisfying $\la^d =g$ and $\la_n^d=g_n$.

The proof essentially consists of two parts. 
First, on the complement of the zero set $Z_g$ of $g$, we show that 
the distance of $\la'(x)$ and $\th^{r(x)}\la_n'(x)$, where $\th$ is a $d$-th root of unity and the power $r(x) \in \{1,\ldots,d\}$ is chosen such that 
the distance of $\la(x)$ and $\th^{r(x)} \la_n(x)$ is minimal,
tends to zero as $n \to \infty$. Then we use the dominated convergence theorem; the domination is guaranteed by a result of Ghisi and Gobbino \cite{GhisiGobbino13}
which we recall, in slightly adapted form, in \Cref{prop:m}.

Secondly, on the accumulation points $\on{acc}(Z_g)$ of $Z_g$,
the derivative $\la'$ vanishes (where it exists). 
Using the uniform bounds for the $L^q$-norm (for $1 \le q <d/(d-1)$) of $\la_n'$ given in \Cref{prop:m},
we prove that $\|\la_n'\|_{L^q(\on{acc}(Z_g))} \to 0$ as $n \to \infty$.

This is enough to conclude the proof since $Z_g \setminus \on{acc}(Z_g)$ has measure zero.

\subsubsection*{General case} 

The proof of the general case proceeds by induction on the degree of the polynomials.
It follows the overall strategy of our proof of the optimal Sobolev regularity of the roots in \cite{ParusinskiRainer15}; see also \Cref{thm:optimal}.

Without loss of generality we may assume that the polynomials $P_{\tilde a_n}$, for $n \ge 1$, and $P_{\tilde a}$ are in Tschirnhausen form, i.e., the coefficients of $Z^{d-1}$ vanish identically.
(For notational clarity, we consistently equip the coefficients of polynomials in Tschirnhausen form with a ``tilde''.)
Let $\la,\la_n : I \to \C^d$ be continuous parameterizations of the roots of $P_{\tilde a}$, $P_{\tilde a_n}$, respectively.

On the zero set $Z_{\tilde a}$ of the coefficient vector $\tilde a$, all roots of $P_{\tilde a}$ are equal to zero, 
hence, $\la'(x)=0$ for all $x \in \on{acc}(Z_{\tilde a})$, where $\la'(x)$ exists.
In analogy to the radical case, we show that $\|\la_n'\|_{L^q(\on{acc}(Z_g), \C^d)} \to 0$ as $n \to \infty$, for all $1 \le q<d/(d-1)$.
To this end, we modify in \Cref{thm:optimalmod} the uniform bounds found in \cite{ParusinskiRainer15}.

For each $x_0$ in the complement of $Z_{\tilde a}$, we find an interval $I' \subseteq I$ containing $x_0$ 
on which the polynomial $P_{\tilde a}$ splits and, for large enough $n$, also $P_{\tilde a_n}$ splits. 
More precisely, on $I'$ and for large $n$, we have simultaneous splittings into polynomial factors 
\[
    P_{\tilde a} = P_bP_{b^*} \quad \text{ and } P_{\tilde a_n} = P_{b_n}P_{b_n^*},
\]
where
\begin{itemize}
    \item $d_b:=\deg P_b = \deg P_{b_n} < d$, and
    \item there exist bounded analytic functions $\ps_i$ with bounded partial derivatives of all orders 
        such that the coefficients of $P_b$ and $P_{b_n}$ are given by
        \begin{align*}
            b_{i} &= \tilde a_k^{i/k}\, \ps_i(\tilde a_k^{-2/k} \tilde a_2, \ldots, \tilde a_k^{-d/k} \tilde a_d),  
            \\
            b_{n,i} &= \tilde a_{n,k}^{i/k}\, \ps_i(\tilde a_{n,k}^{-2/k} \tilde a_{n,2}, \ldots, \tilde a_{n,k}^{-d/k} \tilde a_{n,d}).  
        \end{align*}
\end{itemize}
The same is true for the second factors in the splitting
and similar formulas hold for the coefficients of the factors after putting them in Tschirnhausen form.
Here $k \in \{2,\ldots, d\}$ is chosen such that $|\tilde a_k(x_0)|^{1/k} \ge |\tilde a_j(x_0)|^{1/j}$ for all $2 \le j \le d$, 
which entails $|\tilde a_{n,k}(x_0)|^{1/k} \ge\frac{2}{3} |\tilde a_{n,j}(x_0)|^{1/j}$ for all $2 \le j \le d$ and large enough $n$. 
Note that $\tilde a_k$ and $\tilde a_{n,k}$ are bounded away from zero on $I'$.

Using that the composition from the left with an analytic function is continuous on the space of $C^d$ maps (see \Cref{prop:lefttr}),
we conclude that 
\begin{align*}
    \|b - b_n\|_{C^d(\ol {I'},\C^{d_b})} \to 0 \quad 
    \text{ as } n \to \infty; 
\end{align*}
and similarly for the second factors.
This allows us to argue by induction on the degree; for the precise induction argument see \Cref{prop:induction}.

Finally, the proof of \Cref{thm:main1} will be completed in \Cref{lem:subsequence} 
with an application of Vitali's convergence theorem. The uniform integrability follows 
from the uniform bounds proved in \cite{ParusinskiRainer15}.

\subsection{Organization of the paper}

After recalling general facts on the function spaces and fixing notation in \Cref{sec:spaces},
we introduce in \Cref{sec:dSob} the metric space $\cA_d(\C)$ of unordered $d$-tuples of complex numbers 
and the Sobolev space $W^{1,q}(I,\cA_d(\C))$.
In \Cref{sec:dSob}, we also define and discuss the semimetric $\dd^{1,q}_I$ on this space and the 
corresponding notion of convergence.

\Cref{sec:Almgrenproof} is dedicated to the proof of \Cref{thm:Almgren and convergence} which implies 
that the conclusions of \Cref{thm:mainAlmgren} and \Cref{thm:main1} are equivalent.

In \Cref{sec:rad}, we give a complete proof of the radical case. While it contains some of the main ideas,
it is much simpler than the general case since the splitting principle is not needed. 

In \Cref{sec:poly}, we collect facts on polynomials and prepare the tools for the general case.
We recall our result on the optimal Sobolev regularity of the roots in \Cref{sec:opt} 
proving a new uniform bound for the $L^q$ norm of the derivatives of the roots.
This new bound is a crucial ingredient, besides the splitting principle and the result in the radical case, 
for the proof of \Cref{thm:main1} which is carried out in \Cref{sec:proofs1}.

In \Cref{sec:main1var}, we prove \Cref{thm:main1var}. 
In \Cref{sec:proofs2}, 
multiparameter versions are deduced by sectioning arguments, in particular, \Cref{thm:mainAlmgrenmult} is proved.
In \Cref{sec:Wasserstein}, we interpret the main results in the Wasserstein space on $\C$ which finally leads to 
the proof of \Cref{thm:mainse}.

In the \Cref{sec:appendix}, we recall Vitali's convergence theorem and give a short proof of \Cref{prop:lefttr}.

\subsection*{Notation}

The $m$-dimensional Lebesgue measure in $\R^m$ is denoted by $\cL^m$.  
If not stated otherwise, ``measurable'' means ``Lebesgue measurable'' and ``almost everywhere'' means 
``almost everywhere with respect to Lebesgue measure''. For measurable $E \subseteq \R^m$, we usually write 
$|E|=\cL^m(E)$.

For $1 \le p \le \infty$, $\|z\|_p$ denotes the $p$-norm of $z \in \C^d$. 
If $f : E \to \C^d$, for measurable $E \subseteq \R^m$, is a measurable map, then we set 
\[
    \|f\|_{L^p(E,\C^d)} := \big\|\|f\|_2 \big\|_{L^p(E)}. 
\]

For us a set is countable if it is either finite or has the cardinality of $\N$.

A selection of a set-valued map $F: X \to 2^Y$ between sets $X$ and $Y$ 
is a map $f : X \to Y$ such that $f(x) \in F(x)$ for all $x \in X$.
A parameterization of $F$ is a pair $(f,Z)$, where $f: X \times Z \to Y$ 
is such that $F(x) = \{f(x,z): z \in Z\}$ for all $x \in X$.
For instance, the roots of a monic polynomial $P_a$ of degree $d$ form a set-valued map $\C^d \ni a \mapsto \La(a) \in 2^{\C}$ 
and a parameterization of the roots is a map $\la : \C^d \times \{1,\ldots,d\} \to \C$ with $\La(a) = \{\la(a,1),\ldots,\la(a,d)\}$ for all $a \in \C^d$.

\section{Function spaces} \label{sec:spaces}

Let us fix notation and recall background on the function spaces used in this paper.

\subsection{H\"older--Lipschitz spaces}

Let $U \subseteq \R^m$ be open and $k \in \N$.
Then $C^k(U)$ is the space of $k$-times continuously differentiable complex valued functions
with its natural Fr\'echet topology.
If $U$ is bounded, then $C^k(\ol U)$ denotes the space of all $f \in C^k(U)$
such that each $\p^\al f$, $0\le |\al|\le k$, 
has a continuous extension to the closure $\ol U$. Endowed with the norm
\[
    \|f\|_{C^k(\ol U)} := \max_{|\al|\le k} \sup_{x \in U} |\p^\al f(x)|
\]
it is a Banach space.
For $0 < \ga \le 1$, we consider the H\"older--Lipschitz seminorm
\[
    |f|_{C^{0,\ga}(\ol U)} := \sup_{x,y \in U, \, x \ne y}\frac{|f(x)-f(y)|}{\|x-y\|_2^\ga}.
\]
For $k \in \N$ and $0 < \ga \le 1$, we have the Banach space  
\[
    C^{k,\ga}(\ol U) := \{f \in C^k(\ol U) : \|f\|_{C^{k,\ga}(\ol U)}  < \infty\},
\]
where 
\[
    \|f\|_{C^{k,\ga}(\ol U)} := \|f\|_{C^k(\ol U)} + \max_{|\al|=k} |\p^\al f|_{C^{0,\ga}(\ol U)}.
\]
We write $C^{k,\ga}(U)$ for the space of $C^k$ functions on $U$ that 
belong to $C^{k,\ga}(\ol V)$ for each relatively compact open $V \Subset U$,
with its natural Fr\'echet topology.

\subsection{Lebesgue spaces}

Let $U \subseteq \R^m$ be open and  $1 \le p \le \infty$.
We denote by $L^p(U)$ the Lebesgue space with respect to the $m$-dimensional Lebesgue measure $\cL^m$, 
and $\| \cdot \|_{L^p(U)}$ is the corresponding $L^p$-norm.
For Lebesgue measurable sets $E \subseteq \R^n$ we also write $|E| = \cL^m(E)$. 

Assume that $U$ is bounded. 
A measurable function $f : U \to \C$ belongs to the weak $L^p$-space $L_w^p(U)$ if 
\[
    \|f\|_{p,w,U} := \sup_{r\ge 0} \Big(  r\, |\{x \in U : |f(x)| > r\}|^{1/p} \Big) < \infty.
\]  
For $1 \le q < p < \infty$ we have (cf.\ \cite[Ex.\ 1.1.11]{Grafakos08})
\begin{equation} \label{eq:qp}
    \|f\|_{q,w,U} \le \|f\|_{L^q(U)} \le \Big(\frac{p}{p-q}\Big)^{1/q} 
    |U|^{1/q-1/p} \|f\|_{p,w,U}
\end{equation}
and hence
$L^p(U) \subseteq L_w^p(U) \subseteq L^q(U) \subseteq L_w^q(U)$
with strict inclusions. 
We remark that $\|\cdot\|_{p,w,U}$ is only a quasinorm.
Its $p$-th power is $\si$-subadditive but not $\si$-additive (see \cite[Section 2.2]{ParusinskiRainer15}).

We remark that for continuous functions $f : U \to \C$ we have (and use interchangeably)  
$\|f\|_{L^\infty(U)} = \|f\|_{C^0(\ol U)}$.

\subsection{Sobolev spaces}

For $k \in \N$ and $1 \le q \le \infty$, 
we consider the Sobolev space 
\[
    W^{k,q}(U) := \{f \in L^q(U) : \p^\al f \in L^q(U) \text{ for } |\al|\le k\},
\]
where $\p^\al f$ are distributional derivatives. Endowed with the norm 
\[
    \|f\|_{W^{k,q}(U)} := \sum_{|\al|\le k} \|\p^\al f\|_{L^q(U)}
\]
it is a Banach space. 

\subsection{A result on composition}

In the following proposition we use the norm  
\begin{equation} \label{eq:Cknorm}
    \|f\|_{C^k(\ol U,\R^\ell)} := \max_{0\le j \le k} \sup_{x \in U} \|d^j f(x)\|_{L_j(\R^m,\R^\ell)} 
\end{equation}
on the space $C^{k}(\ol U,\R^\ell) := (C^k(\ol U,\R))^\ell$, where $U \subseteq \R^m$ 
and $L_j(\R^m,\R^\ell)$ is the space of $j$-linear maps with $j$ arguments in $\R^m$ and values in $\R^\ell$.

\begin{proposition} \label{prop:lefttr}
    Let $U \subseteq \R^m$ and $V \subseteq \R^\ell$ be open, bounded, and convex.
    Let $\ps \in C^{k+1}(\ol V,\R^p)$. Then 
    $\ps_* : C^k(\ol U,V) \to C^k(\ol U,\R^p)$, $\ps_* (\vh) := \ps \o \vh$, is well-defined and continuous.
    More precisely, for $\vh_1,\vh_2$ in a bounded subset $B$ of $C^k(\ol U,V)$,
    \[
        \| \ps_*(\vh_1) - \ps_*(\vh_2)\|_{C^k(\ol U,\R^p)} \le C \, \|\ps\|_{C^{k+1}(\ol V,\R^p)} \|\vh_1-\vh_2\|_{C^k(\ol U,\R^\ell)},
    \]
    where $C=C(k,B)$.
\end{proposition}

This result must be well-known; 
we give a short proof in \Cref{ssec:A2}.

\subsection{Absolutely continuous curves in a metric space}

\label{ssec:ACq}

Let $I \subseteq \R$ be a bounded open interval. 
Let $1 \le q \le \infty$. 
A curve $\ga : I \to X$ in a complete metric space $(X,\mathsf d)$ belongs to $AC^q(I,X)$ if there exists $m \in L^q(I)$ such that 
\begin{equation} \label{eq:metricspeed1}
    \mathsf d(\ga(x),\ga(y)) \le \int_x^y m(t)\, dt, \quad \text{ for all } x, y \in I,~ x \le y.
\end{equation}
In that case, the limit
\[
    \lim_{h \to 0} \frac{\mathsf d(\ga(x+h),\ga(x))}{|h|} =: |\dot \ga|(x)
\]
exists for almost every $x \in I$ and is called the \emph{metric speed} of $\ga$ at $x$.
Furthermore, $|\dot \ga| \in L^q(I)$ and \eqref{eq:metricspeed1} holds with $m$ replaced by $|\dot \ga|$; 
one has $|\dot \ga| \le m$ almost everywhere in $I$ for any $m$ that satisfies \eqref{eq:metricspeed1}.
See \cite[Definition 1.1.1]{Ambrosio:2008aa}.

The \emph{$q$-energy} $\cE_q : C^0(I,X) \to [0,\infty]$ is defined by 
\begin{equation*}
    \cE_q(\ga) := 
    \begin{cases}
        \int_I (|\dot \ga|(t))^q\, dt & \text{ if } \ga \in AC^q(I,X),
        \\
        \infty & \text{ otherwise. }
    \end{cases}
\end{equation*}

\subsection{Absolutely continuous curves in $\C^d$}

Let $I \subseteq \R$ be a bounded open interval. 
Let $1 \le q \le \infty$.
A continuous curve $\ga : I \to \C^d$ belongs to $AC^q(I,\C^d)$ with respect to the metric 
induced by $\| \cdot \|_2$ if and only if 
$\ga$ is differentiable at almost every $x \in I$, 
the derivative $\ga'$ belongs to $L^q(I,\C^d)$, and 
\[
    \ga(y)- \ga(x) = \int_x^y \ga'(t)\, dt, \quad \text{ for all }  x,y \in I,~ x \le y.  
\]
In that case,
\[
    |\ga'|(x) = \|\ga'(x)\|_2 \quad \text{ for almost every } x \in I.
\]
See \cite[Remark 1.1.3]{Ambrosio:2008aa}.

\section{$d$-valued Sobolev functions} \label{sec:dSob}

\subsection{Unordered $d$-tuples of complex numbers} \label{ssec:AdC}

The symmetric group $\on{S}_d$ acts on $\C^d$ by permuting the coordinates,
\[
    \si z = \si (z_1,\ldots,z_d) := (z_{\si(1)},\ldots, z_{\si(d)}), \quad \si \in \on{S}_d,~ z \in \C^d,
\]
and thus induces an equivalence relation.
The equivalence class of $z=(z_1,\ldots,z_d)$ is the \emph{unordered tuple} $[z]=[z_1,\ldots,z_d]$.  
Let us consider the set 
\[
    \cA_d(\C) := \{[z] : z \in \C^d\} 
\]
of unordered complex $d$-tuples.
It is a complete metric space if equipped with the metric
\begin{equation*}
    \mathbf d([z],[w]) := \min_{\si \in \on{S}_d} \de(z,\si w), 
\end{equation*}
where 
\[
    \de(z,\si w) := \frac{1}{\sqrt d} \|z-\si w\|_2=  \frac{1}{\sqrt d}\Big( \sum_{i=1}^d |z_i - w_{\si(i)}|^2\Big)^{1/2}.
\]
It follows that the induced map $[\cdot] : \C^d \to \cA_d(\C)$ is Lipschitz.

We will also represent the element $[z_1,\ldots,z_d]$ of $\cA_d(\C)$ by the sum $\sum_{i=1}^d \a{z_i}$, where 
$\a{z_i}$ denotes the Dirac mass at $z_i \in \C$.
If normalized, i.e., $\frac{1}{d}\sum_{i=1}^d \a{z_i}$, then, in this picture, $\mathbf d$ is induced by 
the $L^2$ based Wasserstein metric on the space of probability measures on $\C$; see \Cref{ssec:probability}.\footnote{This is the reason for the factor $1/\sqrt d$ in the definition of $\dd$.}

\subsection{$d$-valued Sobolev functions}

Due to Almgren \cite{Almgren00}, see also \cite{De-LellisSpadaro11}, 
there exist an integer $N= N(d)$, positive constants $C_i= C_i(d)$, $i=1,2$, and an  
injective Lipschitz mapping $\De : \cA_d(\C) \to \R^N$ with Lipschitz constant $\le C_1$ and   
Lipschitz constant of $\De|^{-1}_{\De(\cA_d(\C))}$ bounded by $C_2$. Moreover, there is a
Lipschitz retraction of $\R^N$ onto $\De(\cA_d(\C))$.
Almgren used this bi-Lipschitz embedding to define Sobolev spaces of $\cA_d(\C)$-valued functions: 
for open $U \subseteq \R^m$ and $1 \le q \le \infty$ set 
\[
    W^{1,q}(U,\cA_d(\C)) := \{f : U \to \cA_d(\C) : \De \o f \in W^{1,q}(U,\R^N)\}.  
\]
For an equivalent intrinsic definition of $W^{1,q}(U,\cA_d(\C))$, see \cite[Definition~0.5 and Theorem~2.4]{De-LellisSpadaro11}.
Then $W^{1,q}(U,\cA_d(\C))$ carries the metric
\begin{equation} \label{eq:Almgren}
    (f,g) \mapsto \|\De \o f - \De \o g\|_{W^{1,q}(U,\R^N)}
\end{equation}
which makes it a complete metric space (where functions that coincide almost everywhere are identified).

\begin{lemma} \label{lem:Acomplete}
    The space $W^{1,q}(U,\cA_d(\C))$ with the metric given in \eqref{eq:Almgren} is complete. 
\end{lemma}

\begin{proof}
    A Cauchy sequence $f_n$ in $W^{1,q}(U,\cA_d(\C))$ is by definition  
    a Cauchy sequence $\De \o f_n$ in $W^{1,q}(U,\R^N)$. 
    The completeness of $W^{1,q}(U,\R^N)$ implies that $\De\o f_n$ converges to a function $h \in W^{1,q}(U,\R^N)$.
    It remains to show that there exists $f : U \to \cA_d(\C)$ such that $h = \De \o f$ almost everywhere in $U$.
    There exist a subsequence $\De \o f_{n_k}$ and a nonnegative function $g \in L^q(U)$ 
    such that $(\De \o f_{n_k})(x) \to h(x)$ and $\|(\De \o f_{n_k})(x)\|_2 \le g(x)$ for almost every $x \in U$
    (cf.\ \cite[Theorem 2.7]{LiebLoss01}). 
    For each such $x$, it follows that $f_{n_k}(x)$ is a Cauchy sequence in $\cA_d(\C)$ and 
    hence it converges in $\cA_d(\C)$. So there is a function $f : U \to \cA_d(\C)$ such that $f_{n_k} \to f$ almost everywhere in $U$ 
    and hence $\De \o f_{n_k} \to \De \o f$ almost everywhere in $U$. 
    By the dominated convergence theorem, $\|\De \o f_{n_k} -\De \o f\|_{L^q(U,\R^N)} \to 0$ and thus $\|\De \o f_{n} -\De \o f\|_{L^q(U,\R^N)} \to 0$,
    since $\De \o f_n$ is a Cauchy sequence. 
    Since the limit is unique, we have $h=\De \o f$ almost everywhere.
\end{proof}

\subsection{Almgren's embedding}

Let us recall Almgren's construction of $\De$.

\begin{definition}\label{def:Amap}
    We say that 
    $$
    \Et : \cA_d(\C) \to \R^d
    $$
    is an \emph{Almgren map} if there is a unit complex number 
    $\th \in \C$ such that $\Et ([z])$ is an array of $d$ real numbers $\et (z_i):=\Re (\th z_i) $ arranged in increasing order, i.e.,
    $$
    \Et ([z]) = \Et ([z_1,\ldots,z_d]) = (\et (z_{\sigma(1)}) , \ldots , \et (z_{\sigma(d)}), 
    $$
    where $\si \in \on{S}_d$ is chosen so that $\et (z_{\sigma(1)})\le \et (z_{\sigma(2)}) \le 
    \cdots \le \et (z_{\sigma(d)})$.
    We also say that $\Et$ is the Almgren map associated to the real linear form $\et$. 
\end{definition}

By Almgren's combinatorial lemma (see e.g.\ \cite[Lemma 2.3]{De-LellisSpadaro11}) there exists $\al=\al(d)>0$ and a finite set of linear forms 
$\Lambda=\{\et_1, \ldots \et_h\}$, where $\et_l (z):=\Re (\th_l z) $ 
for unit complex numbers $\th_l$, with the following property: given any set of $d^2$ complex numbers, $\left\{z_1,\ldots,z_{d^2}\right\}\subseteq\C$, 
there exists $\et_l\in\Lambda$ such that
\begin{equation}\label{eq:combineq}
    |\et_l (z_k)| \geq \alpha |z_k| \quad\textrm{ for all }k\in\big\{1,\ldots,d^2\big\}.
\end{equation}
For instance, we may take $h=2d^2+1$ and as $\{\th_1, \ldots \th_h\}$ the set of all $h$-th roots of unity. 
Let $\Et_l$ denote the Almgren map associated to $\et_l$.
Almgren's embedding $\De : \cA_d(\C) \to \R^N$, $N=dh$, is then defined by
\begin{align}\label{eq:AlmgrenDE}
    \De ([z]) = h^{-1/2} (\Et_1 ([z]), \ldots, \Et_h ([z])).
\end{align}  

\subsection{Curves of class $W^{1,q}$ in $\cA_d(\C)$}  

We recall some basic constructions and results from \cite{De-LellisSpadaro11}.  
We focus our attention on the one parameter case, so let $I \subseteq \R$ be an open interval.\footnote{Following our notation, the number $Q$ of \cite{De-LellisSpadaro11} is replaced by $d$.} 

First we recall another (equivalent) definition of $W^{1,q} (I,\cA_d(\C))$ (see \cite[Definition 0.5]{De-LellisSpadaro11}) which is independent of Almgren's embedding.  

\begin{definition}[Intrinsic definition]
    \label{d:W1p}
    A measurable function $f:I \ra\Iq$ is in the Sobolev class $W^{1,q}$ ($1\leq q\leq\infty$) if there exists a function
    $\varphi \in L^q(I,\R_{\ge 0})$
    such that 
    \begin{itemize}
        \item[(i)] $x\mapsto \mathbf d (f(x),T)\in W^{1,q}(I)$ for all $T\in \Iq$;
        \item[(ii)] $\abs{(\mathbf d (f, T))'}\leq\varphi$ almost everywhere in $I$
            for all $T\in \Iq$.
    \end{itemize}
\end{definition}

The minimal function $\tilde\varphi$ fulfilling (ii), that is, 
\begin{equation*}
    \tilde\varphi\leq\varphi \quad \text{almost everywhere for any other $\varphi$ satisfying (ii),}
\end{equation*}
is measurable and is denoted by $|D f|$.
It can be characterized by the following
property: for every countable
dense subset $\{T_i\}_{i\in\N}$ of $\Iq$, 
\begin{equation*} 
    \abs{D f}=\sup_{i\in\N}\abs{(\mathbf d(f,T_i))' }
    \quad\textrm{ almost everywhere in } I.
\end{equation*}

\begin{proposition}[{\cite[Proposition 1.2]{De-LellisSpadaro11}}]\label{p:Wselection-1} 
    Let $f\in W^{1,q} (I, \Iq)$.
    Then,
    \begin{itemize}
        \item[$(a)$] $f\in AC(I,\Iq)$ and, moreover,
            $f\in C^{0,1-\frac{1}{q}}(I,\Iq)$ for $q>1$;\footnote{Here we mean that the statements hold after possibly redefining $f$ on a set of measure $0$.} 
        \item[$(b)$] 
            there exists a parameterization\footnote{In \cite{De-LellisSpadaro11}, it is called a selection of $f$.} $f_1,\ldots,f_d\in W^{1,q}(I,\C)$ of $f$, i.e.,
            \[
                f = \a{f_1} + \cdots + \a{f_d},
            \]
            such that $\abs{Df_i} = \abs{f_i'}\leq\abs{Df}$ almost everywhere.
    \end{itemize}
\end{proposition}

Actually, the proof of \Cref{p:Wselection-1} in \cite{De-LellisSpadaro11} implies that $f \in W^{1,q}(I,\Iq)$ belongs to $AC^q(I,\Iq)$ in 
the sense of \Cref{ssec:ACq}. 
In the situation of \Cref{p:Wselection-1}, we will always mean without further mention that $f$ and $f_1,\ldots,f_d$ are the continuous representatives.

Since there exists an absolutely continuous parameterization of $f\in W^{1,q} (I, \Iq)$ (by \Cref{p:Wselection-1}(ii)), we can define 
not only the absolute value of its derivative $|D f|$ but also its derivative $D f$.

\begin{definition}[{\cite[Definition 1.9]{De-LellisSpadaro11}}]\label{d:diff}
    Let $f = \sum_i \a{f_i} : I\ra\Iq$ and $x_0\in I$. We say that $f$ is \emph{differentiable}
    at $x_0$ if there exist $d$ complex numbers $L_i$ satisfying:
    \begin{itemize}
        \item[(i)]$\mathbf d(f(x),T_{x_0} f(x))=o(\abs{x-x_0})$, where
            \begin{equation}\label{e:taylor1st}
                T_{x_0} f(x):=\sum_i\a{f_i(x_0) + L_i\cdot(x-x_0)};
            \end{equation}
        \item[(ii)] $L_i=L_j$ if $f_i(x_0)=f_j(x_0)$.
    \end{itemize}
    The $d$-valued map $T_{x_0} f$
    is called the {\em first-order approximation} of $f$ at $x_0$.
    We denote $L_i$ by $Df_i(x_0)$ and the point $\sum_i \a{Df_i(x_0)} \in \Iq$
    will be called the \emph{differential} of $f$ at $x_0$ and will be denoted by $Df (x_0)$.  
\end{definition}

What we call here ``differentiable'', following \cite{De-LellisSpadaro11}, is called
``strongly affine approximable'' by Almgren \cite{Almgren00}. 

Note that, by (ii) in the definition, the notation is consistent (see \cite[Remark 1.11]{De-LellisSpadaro11}):
if $g_1,\ldots, g_d$ is another parameterization of $f$, $f$ is differentiable at $x_0$, and $\si \in \on{S}_d$ is such that 
$g_i(x_0) = f_{\si(i)}(x_0)$ for all $1 \le i \le d$, then $Dg_i(x_0) = Df_{\si(i)}(x_0)$.

As follows from 
\Cref{p:Wselection-1}, every $f\in W^{1,q} (I, \Iq)$ is differentiable almost everywhere.  Moreover
if $f$ is represented as in (b) of \Cref{p:Wselection-1} then $L_i = Df_i(x_0)$ almost everywhere.
Indeed, let $f_1,\ldots,f_d \in W^{1,q}(I,\C)$ be a parameterization of $f$ and assume that all $f_i$ are differentiable at $x_0$.
Then
\[
    f_i(x) = f_i(x_0) + f'(x_0)(x-x_0) + o(|x-x_0|)
\]
and 
\begin{align*}
    \MoveEqLeft \dd\Big(f(x), \sum_i \a{f_i(x_0) + f_i'(x_0)(x-x_0)}\Big) 
    =     
    \dd\Big(f(x), \sum_i \a{f_i(x) + o(|x-x_0|)}\Big) 
    \\
    &= \min_{\si \in \on{S}_d} \frac{1}{\sqrt d} \Big(\sum_i |f_i(x) - f_{\si(i)}(x) + o(|x-x_0|)|^2  \Big)^{1/2} = o(|x-x_0|).
\end{align*}
On each accumulation point $x_0$ of $\{x \in I : f_i(x) = f_j(x)\}$, where the derivatives $f_i'(x_0)$ and $f_j'(x_0)$ exist, 
we have $f_i'(x_0) = f_j'(x_0)$. Now it is easy to conclude the claim.

\subsection{A distance notion on $W^{1,q}(I,\Iq)$}

\begin{definition} \label{def:dd}
    Let $f,g \in W^{1,q}(I,\Iq)$ and let 
    \[
        f=\a{f_1} + \cdots + \a{f_d}, \quad g=\a{g_1} + \cdots + \a{g_d}
    \]
    be parameterizations of $f$, $g$ with $f_i,g_i \in W^{1,q}(I,\C)$ as in \Cref{p:Wselection-1}.
    Fix any ordering of the elements of $\on{S}_d$.
    For $x \in I$, let
    \begin{equation*}
        \ta(x) := \min \Big\{\ta\in \on{S}_d : \frac{1}{\sqrt d} \Big(\sum_i | f_i(x) - g_{\ta(i)}(x)|^2\Big)^{1/2} =  \dd(f(x),g(x)) \Big\}
    \end{equation*}
    and set 
    \begin{equation*}
        \bs_0(f,g)(x) :=  \dd(f(x),g(x)).
    \end{equation*}
    For $x \in I$ such that $Df(x) = \sum_i \a{Df_i(x)}$ and $Dg(x)  = \sum_i \a{Dg_i(x)}$ exist in the sense of \Cref{d:diff},
    set 
    \begin{equation}\label{eq:maxorderings}
        \bs_1(f, g)(x) := \max  \frac{1}{\sqrt d} \Big(\sum_i | Df_i(x) - Dg_{\ta(x)(i)}(x)|^2\Big)^{1/2},
    \end{equation}
    where the maximum is taken over all orderings of $\on{S}_d$.
    By the remarks above, $\bs_1(f,g)(x)$ is defined for almost every $x \in I$. 
    It is independent of the choices of parameterizations $f_1,\ldots,f_d$ and $g_1,\ldots,g_d$ of $f$ and $g$.

    For any measurable subset $E \subseteq I$, we set 
    \begin{equation*}
        \mathbf{d}^{1,q}_{E}( f,g )
        := \|\bs_0(f,g)\|_{L^\infty(E)}
        + \| \bs_1(f,g) \|_{L^q(E)}
    \end{equation*}
    which is justified by \Cref{lem:Borel}.
\end{definition}

\begin{lemma} \label{lem:Borel}
    The functions $\mathbf{s}_i(f,g) : I \to \R$, for $i=0,1$, are Borel measurable.
    Here we extend $\bs_1(f,g)$ by $0$ to those points in $I$, where it is not defined.
\end{lemma}

\begin{proof}
    First of all, $\mathbf{s}_0(f,g)$ is continuous. To see that $\mathbf{s}_1(f,g)$ is Borel measurable,
    it suffices to check that $\ta : I \to \on{S}_d$ is Borel measurable 
    (with respect to the power set of $\on{S}_d$ as $\si$-algebra).
    Fix $\si \in \on{S}_d$. Then 
    \begin{align*}
        \MoveEqLeft \{x \in I : \ta(x)\le \si\}
        \\
        &= \bigcup_{\ka \le \si} \Big\{ x \in I : \frac{1}{\sqrt d} \Big(\sum_i | f_i(x) - g_{\ka(i)}(x)|^2\Big)^{1/2} =\mathbf d([f(x)],[g(x)]) \Big\}
    \end{align*}
    is Borel measurable. Since the sets $\{\ta \in \on{S}_d : \ta \le \si\}$ generate the power set 
    of $\on{S}_d$ as $\si$-algebra, the assertion follows.
\end{proof}

\begin{lemma} \label{lem:dd1qsemidist}
    Let $I \subseteq \R$ be a bounded open interval and $E \subseteq I$ a measurable set. 
    Let $f, g \in W^{1,q}(I,\Iq)$. Then:
    \begin{enumerate}
        \item $\mathbf{d}^{1,q}_{E}( f,f )=0$.
        \item $\mathbf{d}^{1,q}_{E}( f,g )=0$ implies $f = g$ on $E$.  
        \item $\mathbf{d}^{1,q}_{E}( f,g )=\mathbf{d}^{1,q}_{E}( g,f )$.
    \end{enumerate}
    In particular, $\dd^{1,q}_I$ is a semimetric on $W^{1,q}(I,\Iq)$.
\end{lemma}

\begin{proof}
    (1) In this case, for any $x \in I$, 
    we have 
    \[
        \ta(x) = \min \{\ta \in \on{S}_d : f_i(x) = f_{\ta(i)}(x) \text{ for all } i\}
    \]
    so that $\bs_1(f,f)(x) = 0$ thanks to \Cref{d:diff}(ii) (if it is defined at $x$).

    (2) If $\dd^{1,q}_E(f,g) = 0$ then $\dd(f,g)=0$ on $E$ (since $\dd(f,g)$ is continuous).

    (3) It is immediate from the definition that $\bs_0(f,g)(x) = \bs_0(g,f)(x)$ and  $\bs_1(f,g)(x) = \bs_1(g,f)(x)$ 
    (where defined).
\end{proof}

\subsection{Convergence in $W^{1,q}(I,\Iq)$}

There is a notion of \emph{weak convergence} in $W^{1,q}(I,\Iq)$; see \cite[Definition 2.9]{De-LellisSpadaro11}.

\begin{definition}[Weak convergence]\label{d:weak convergence}
    Let $f_n, f\in W^{1,q}(I,\Iq)$. We say 
    that $f_n$ \emph{converges weakly 
    to $f$ in $W^{1,q}(I,\Iq)$} as $n \to \infty$  
    (and we write $f_n\rightharpoonup f$) if
    \begin{itemize}
        \item[(i)] $\int_I \mathbf d(f(x),f_n(x))^q\, dx\ra 0$ as  $n\ra\infty$;
        \item[(ii)] there exists a constant
            $C>0$ such that $\int_I |Df_n(x)|^q\, dx\leq C$ for every $n$.
    \end{itemize}
\end{definition}

This notion is too weak for our purpose: 
weak convergence $\La_n \rightharpoonup \La$ does not even imply the conclusion of \Cref{thm:mainse}.  
Let us introduce a stronger notion of convergence based on the semimetric $\dd^{1,q}_I$.

\begin{definition}[Strong convergence]\label{d:convergence}
    Let $f_n, f\in W^{1,q}(I,\Iq)$. We say 
    that $f_n$ \emph{converges to $f$ in $W^{1,q}(I,\Iq)$} as $n \to \infty$
    (and we write $f_n\rightarrow f$),
    if
    \[ 
        \mathbf{d}^{1,q}_{I}( f, f_n )\ra 0 \quad \text { as }  n\ra\infty .
    \] 
\end{definition}

\begin{theorem}\label{thm:Almgren and convergence}
    Let $\De : \Iq \to \R^N$ be an Almgren embedding.  
    Then $f_n \to f$ in $W^{1,q}(I,\Iq)$ as $n \to \infty$ if and only if  
    $f_n$ converges to $f$ with respect to the topology induced by the metric \eqref{eq:Almgren}.
\end{theorem}

In particular, the topology induced by the metric \eqref{eq:Almgren} on $W^{1,q}(I,\Iq)$ does not depend on the 
choice of the Almgren embedding.

We will prove \Cref{thm:Almgren and convergence} in \Cref{sec:Almgrenproof}.

\section{Proof of \Cref{thm:Almgren and convergence}} \label{sec:Almgrenproof}

Before we show \Cref{thm:Almgren and convergence} we need some preparatory results.
In the following, $I \subseteq \R$ is a bounded open interval and $q\ge 1$. 
Moreover, $\Et : \Iq \to \R^d$ is an Almgren map with associated real linear form $\et$ (see \Cref{def:Amap}).
Recall the $H$ is Lipschitz,
\begin{equation} \label{eq:HLip}
    \|H([z])-H([w])\|_2  \le C_1 \, \dd([z],[w]), \quad [z],[w] \in \Iq,
\end{equation}
where $C_1=C_1(d) = \sqrt h=\sqrt {2d^2+1}$; see e.g.\ \cite[Section 2.1.2]{De-LellisSpadaro11} and the discussion after \Cref{def:Amap}.

\begin{lemma}\label{lem:forHconvergence}
    Let $f, g\in W^{1,q}(I,\Iq)$  and let $f= [f_{1},\ldots,f_{d}]$, $g= [g_1,\ldots,g_d]$ with all $f_{i}, g_i$ continuous on $I$. 
    Fix $x_0\in I$.  For $x\in I$, let $\ta(x) \in \on{S}_d$ be a permutation  such that 
    \[
        \mathbf d(f(x_0),g(x))= \frac{1}{\sqrt d} \Big(\sum_i|f_i(x_0)-g_{\ta (x)(i)} (x)|^2 \Big)^{1/2}.
    \]
    Denote $H_f=H\circ f$ and $H_g= H\circ g$. 
    Assume that not all $\et(f_i(x_0))$ are equal and
    let $\rho$ denote the minimal distance between distinct 
    $\et(f_i(x_0))$.  
    If $x\in I$ satisfies $\mathbf d(f(x_0),g(x)) < \frac {\rho}{2C_1}$, where $C_1$ is the constant from \eqref{eq:HLip},  then 
    \begin{align*}
        (H_g)_k(x) = \et(g_{\ta (x) (j)}(x)) \quad \Longrightarrow \quad (H_f)_k(x_0) = \et(f_j(x_0)).
    \end{align*}
\end{lemma}

\begin{proof}
    Because $H$ is Lipschitz with Lipschitz constant $\le C_1$, 
    \[
    |(H_{g})_k (x) - (H_f)_k (x_0)|\le C_1\, \mathbf d(g (x),f(x_0)\big) < \frac \rho 2.
\]
Therefore, if $(H_g)_k(x) = \et(g_{\ta (x) (j)}(x))$, then 
\begin{align*}
    |(H_f)_k(x_0) - \et(f_j(x_0))|  \le |(H_{g})_k (x) - (H_f)_k (x_0)| 
    + |\et(g_{\ta (x) (j)}(x) - f_j(x_0))|<\rho,
\end{align*} 
using that $\sqrt d \le C_1$.
Thus, $(H_f)_k(x_0) = \et(f_j(x_0))$, by the definition of $\rho$.
\end{proof}

\begin{corollary}\label{cor:variouswithH}
    Let $f\in W^{1,q}(I,\Iq)$  and let $f= [f_{1},\ldots,f_{d}]$ with all $f_{i}$ continuous on $I$. 
    Fix $x_0\in I$.
    For $x\in I$, let $\ta(x) \in \on{S}_d$ be a permutation  such that 
    \[
        \mathbf d(f(x_0),f(x))= \frac{1}{\sqrt d} \Big(\sum_i|f_i(x_0)-f_{\ta(x)(i)} (x)|^2\Big)^{1/2}.
    \]
    Then, if $x$ is sufficiently close to $x_0$, we have 
    \begin{enumerate}
        \item
            $\et (f_{\ta(x)(i)} (x_0))= \et(f_i(x_0))$ for all $i$;
        \item
            if $(H_f)_i(x) = \et(f_j(x))$ then $(H_f)_i(x_0) = \et(f_j(x_0))$.
    \end{enumerate}
\end{corollary}

\begin{proof}
    If all $\et(f_i(x_0))$ are equal, then the conclusion is trivially true.
    Assume that not all $\et(f_i(x_0))$ are equal.
Let $\rho$ be the minimal distance between distinct $\et(f_i(x_0))$.  
Let  $x$ be such that  $\mathbf d(f(x),f(x_0)\big) < \frac {\rho}{2\sqrt d}$ and  $|f_i(x)-f_i(x_0)| < \frac \rho 2$ for all $i$.  Then 
\begin{align*}
    |\et(f_{\ta(x)(i)} (x_0)-f_i(x_0))| & \le |f_{\ta(x)(i)} (x_0)-f_i(x_0)| \\
                                        & \le |f_{\ta(x)(i)} (x)-f_i(x_0)| 
                                        + |f_{\ta(x)(i)} (x) - f_{\ta(x)(i)}(x_0)| < \rho.
\end{align*} 
This implies (1).  Now (2) follows from (1) and 
\Cref{lem:forHconvergence} for $g=f$.  
\end{proof}

There is a chain rule formula for compositions (from the right and from the left) of differentiable (in the sense of \Cref{d:diff}) maps $f:I\ra\Iq$  
with classically differentiable maps; see \cite[Proposition 1.12]{De-LellisSpadaro11}.  Let us show a version of (iii) 
of that proposition.  

\begin{proposition}\label{prop:chain}
    Let $f\in W^{1,q} (I, \Iq)$ and let $f_{1},\ldots,f_{d} \in W^{1,q}(I,\C)$ be a parameterization of $f$ (see \Cref{p:Wselection-1}).  
    Let $\et :\C\to \R$ be a real linear form and let 
    $F = F_\et :\Iq\to\R^d$ associate to  $[z_1, \ldots, z_d]$ an array of $d$ real numbers $\et (z_i)$ arranged in increasing order.
    Then $F \circ f$ is differentiable almost everywhere and at a point $x_0$ of differentiability,
    after renumbering the $f_i$ such that $F(f(x_0)) = (\et(f_1(x_0),\et(f_2(x_0)),\ldots,\et(f_d(x_0)))$,
    we have 
    \begin{equation*} 
        D (F\circ f)(x_0) =
        (\et (Df_1(x_0)) , \ldots, \et (Df_d(x_0)) ).
    \end{equation*}
\end{proposition}

\begin{proof}
    This would follow from (iii) of \cite[Proposition 1.12]{De-LellisSpadaro11} if $F$ were induced by a differentiable $\on{S}_d$-invariant map $\C^d \to \R^d$,
    but this is not the case, although there is a semialgebraic stratification  of $\Iq $ such that $F$ restricted to every stratum is real analytic. 
    Denote $F\circ f$ by $H:I\to \R^d$.  Suppose for simplicity that $\et (z):=\Re (\th z) $ as in \Cref{def:Amap}.

    Let $x_0\in I$ be such that $f$, all $f_1,\ldots,f_d$, and all $\et(f_1),\ldots,\et(f_d)$ are differentiable at $x_0$. 
    After renumbering the $f_i$, we may assume that $H_i(x_0)= \et(f_i(x_0))$ for all $i$. 
    We also assume that $\et(Df_i(x_0)) = \et(Df_j(x_0))$ whenever $\et(f_i(x_0)) = \et(f_j(x_0))$. 
    Indeed, if $f_i(x_0) = f_j(x_0)$ then the assertion follows from (ii) in \Cref{d:diff}.
    The assertion is also true at accumulations points of $\{x \in I : f_i(x) \ne f_j(x), \et(f_i(x)) = \et(f_j(x))\}$. 
    All the other points where $\et(f_i(x)) = \et(f_j(x))$  form a set of measure zero.

    We want to show that, for each component $H_i$ of $H$, 
    \begin{align}\label{eq:derH}
        |H_i(x) - H_i(x_0) - \et (Df_i(x_0)) (x-x_0)| = o(|x-x_0|).
    \end{align} 
    Let $H_i(x) = \et(f_j(x))$. Then, by (2) of \Cref{cor:variouswithH}, 
    $\et(f_i(x_0)) = H_i(x_0) = \et(f_j(x_0))$. Since $f_j$ is differentiable at $x_0$ 
    \begin{align*}
        |\et\big( f_{j}(x) - f_{j}(x_0) - Df_j(x_0) (x-x_0) \big)| = o(|x-x_0|).
    \end{align*} 
    That implies \eqref{eq:derH} because $\et(Df_i(x_0)) = \et(Df_j(x_0))$ (see the previous paragraph).   
    This ends the proof of \Cref{prop:chain}.
\end{proof}

Fix $x_0$ and $x$ satisfying the assumption of \Cref{lem:forHconvergence}.  
After changing the order of the $f_i $ (or the $g_i$) we may suppose $\ta(x) = \id$.  Then, changing the order of $f_i $ and  the $g_i$ simultaneously we have both 
\begin{align}\label{eq:permutationcleanup}
    \begin{split}
&  \mathbf d(f(x_0),g(x)) =  \frac{1}{\sqrt d} \Big(\sum_i|f_i(x_0)-g_{i} (x)|^2\Big)^{1/2},  \\
&  (H_g)_k(x) = \et(g_{k}(x)), \quad (H_f)_k(x_0) = \et(f_k(x_0))\quad \text{ for all } k.  
    \end{split}
\end{align}
Now we use the above formula for $x=x_0$.  

\begin{corollary}\label{cor:goodbound}
    Let $f, g\in W^{1,q}(I,\Iq)$. Let $x_0\in I$. If not all $\et(f_i(x_0))$ are equal, assume that  
    \[
        \mathbf d(f(x_0),g(x_0)) < \frac \rho{2C_1}, 
    \]
    where $\rho$ is the minimal distance between distinct $\et(f_i(x_0))$ and $C_1$ is the constant from \eqref{eq:HLip}.  
    Then, provided that all derivatives exist at $x_0$, we have 
    \begin{align*}
        \|(H_f-H_g)'(x_0)\|_2\le \sqrt d\cdot \mathbf s_1(f,g) (x_0),
    \end{align*}
    where $\mathbf s_1(f,g) (x_0)$ is defined in \Cref{def:dd}.  
\end{corollary}

\begin{proof}
    Let $f_1,\ldots, f_d \in W^{1,q}(I,\C)$ and $g_1,\ldots, g_d \in W^{1,q}(I,\C)$ be parameterizations of $f$ and $g$, 
    respectively (see \Cref{p:Wselection-1}).
    We may assume that \eqref{eq:permutationcleanup} holds with $x=x_0$ (irrespective if all $\et(f_i(x_0))$ are equal or not). 
    Moreover, we may assume that $\ta (x_0)$ ($=\id$) gives the maximum in 
    \eqref{eq:maxorderings} for $x=x_0$. Then 
    \begin{align*} 
        \Big(\sum_i |Df_i(x_0)-Dg_i(x_0)|^2\Big)^{1/2}   \notag
        = \sqrt d \cdot \bs_1(f,g)(x_0), 
    \end{align*}
    By  \Cref{prop:chain},   
    \begin{align*}
        (H_f)' (x_0)&= (H\circ f)'(x_0)=  (\et (Df_1(x_0)) , \ldots, \et (Df_d(x_0) )), 
        \\
        (H_g)' (x_0)&= (H\circ g)'(x_0)=  (\et (Dg_1(x_0)) , \ldots, \et (Dg_d(x_0) )), 
    \end{align*}
    and therefore
    \begin{align*} 
        \|(H_f-H_g)'(x_0)\|_2 &= \Big(\sum_i |\et (Df_i(x_0)-Dg_i(x_0))|^2\Big)^{1/2}  \notag 
        \\
                              &\le \Big(\sum_i |Df_i(x_0)-Dg_i(x_0)|^2\Big)^{1/2} = \sqrt d \cdot \bs_1(f,g)(x_0) 
    \end{align*}
    as claimed.
\end{proof}

\begin{corollary}\label{cor:oppositebound}
 Let $\De = h^{-1/2} (H_1,\ldots,H_h)  : \Iq \to \R^N$ be an Almgren embedding as in \eqref{eq:AlmgrenDE}
    and let $\et_l$ be the real linear form associated with $H_l$. 
    Let $f, g\in W^{1,q}(I,\Iq)$. Let $x_0\in I$.
    For all $1\le l \le h$ assume the following: if not all $\et_l(f_i(x_0))$ are equal, then 
    \[
        \mathbf d(f(x_0),g(x_0)) < \frac {\rho_l}{2C_1},
    \]
    where $\rho_l$ is the minimal distance between distinct $\et_l(f_i(x_0))$ and $C_1$ is the constant from \eqref{eq:HLip}.  
    Then, provided that all derivatives exist at $x_0$, we have 
    \begin{align*}
        \mathbf s_1(f,g) (x_0) \le C(d)\, \|(\De \o f - \De \o g)'(x_0)\|_2.
    \end{align*}
\end{corollary}

\begin{proof}
    Let $f_1,\ldots, f_d \in W^{1,q}(I,\C)$ and $g_1,\ldots, g_d \in W^{1,q}(I,\C)$ be parameterizations of $f$ and $g$, respectively (see \Cref{p:Wselection-1}).
    Now \eqref{eq:combineq} applied to $z_k = f_i' (x_0) - g_{j}' (x_0)$, for $1 \le i,j \le d$, gives the existence of some $l \in \{1,\ldots,h\}$ such that 
    \begin{equation*}
        |\et_l (f_i' (x_0) - g_{j}' (x_0))| \geq \alpha\, |f_i' (x_0) - g_{j}' (x_0)| \quad\textrm{ for all }
        1\le i,j \le d,
    \end{equation*} 
    assuming that the derivatives exist at $x_0 \in I$.
    We may assume that \eqref{eq:permutationcleanup} holds with $x=x_0$ for $H=H_\ell$. 
    As in the proof of \Cref{cor:goodbound},
    we find that
    \begin{align*}
        \|(H_l \o f-H_l \o g)'(x_0)\|_2 &= \Big(\sum_i |\et (Df_i(x_0)-Dg_i(x_0))|^2\Big)^{1/2}.  
    \end{align*}
    Thus,
    \begin{align*}
        \|(H_l \o f-H_l \o g)'(x_0)\|_2 &\ge \al\, \Big(\sum_i |Df_i(x_0)-Dg_i(x_0)|^2\Big)^{1/2}  
    \end{align*}
    which implies the assertion.
\end{proof}

\subsection{Proof of \Cref{thm:Almgren and convergence}}

Before we start the proof, let us recall an elementary lemma 
which will be used several more times. 

\begin{lemma} \label{lem:reals}
    Let $(r_n)$ be a sequence of real numbers.
    Then $r_n \to 0$ as $n \to \infty$ if and only if each subsequence of $(r_n)$ has a subsequence that converges to $0$.
\end{lemma}

\begin{proof}
    Suppose that $r_n \not\to 0$ as $n \to 0$. Then there exist $\ep>0$ and a  sequence $n_1 < n_2 < \cdots$ such that 
    $|r_{n_k}|\ge \ep$ for all $k \ge 1$. So no subsequence of $(r_{n_k})$ converges to $0$. The opposite direction is trivial. 
\end{proof}

Let $\De = h^{-1/2} (H_1,\ldots,H_h)  : \Iq \to \R^N$ be an Almgren embedding as in \eqref{eq:AlmgrenDE}
and let $\et_l$ be the real linear form associated with $H_l$.
Assume that $f_n \to f$ in $W^{1,q}(I,\Iq)$ as $n \to \infty$, in the sense of \Cref{d:convergence}. 
Then $\|\dd(f,f_n)\|_{L^\infty(I)} \to 0$ as $n \to \infty$. 
Fix $x_0 \in I$.
For all $1 \le l \le h$, let 
$\rho_l(x_0)$ be the minimal distance between distinct $\et_l(f_i(x_0))$, if not all $\et_l(f_i(x_0))$ are equal. 
Let $C_1$ be the constant from \eqref{eq:HLip}.  
Then there exists $n_0 \ge 1$ such that, for all $1 \le l \le h$, 
\begin{equation} \label{eq:ptwrho}
    \dd(f(x_0),f_n(x_0)) < \frac{\rh_l(x_0)}{2C_1}, \quad  n \ge n_0,
\end{equation}
provided not all $\et_l(f_i(x_0))$ are equal.
By \Cref{cor:goodbound},  provided that all the derivatives exist at $x_0$, we have 
\begin{align}\label{eq:Almgrenbound}
    \|(\Delta \circ f - \Delta \circ f_n)'(x_0)\|_2\le C(d)\, \mathbf s_1(f,f_n) (x_0), \quad  n \ge n_0.  
\end{align}  
Since $\|\bs_1(f,f_n)\|_{L^q(I)} \to 0$ as $n \to \infty$, by assumption,
there is a subsequence $(n_k)$ such that $\bs_1(f,f_{n_k}) \to 0$ almost everywhere in $I$ as $k \to \infty$.
By \eqref{eq:Almgrenbound}, for almost every $x_0 \in I$, 
\[
    \|(\Delta \circ f - \Delta \circ f_{n_k})'(x_0)\|_2 \to 0 \quad \text{ as } k \to \infty.
\]
By \cite[Theorem 2.4 and Proposition 2.7]{De-LellisSpadaro11}, almost everywhere in $I$,
\[
    \|(\Delta \circ f_{n})'\|_2 \le \|Df_n\|_2 \le  \bs_1(f,f_n) + \|Df\|_2, 
\]
using that $\|Df_n\|_2$ is independent of the parameterization of $f_n$.  
Since $\sup_{n\ge 1} \bs_1(f,f_n) + \|Df\|_2$ is in $L^q(I)$, by assumption,
the dominated convergence theorem implies that 
\[
    \|(\Delta \circ f - \Delta \circ f_{n_k})'\|_{L^q(I,\R^N)} \to 0 \quad \text{ as } k \to \infty.
\]
This implies that 
\[
    \|(\Delta \circ f - \Delta \circ f_{n})'\|_{L^q(I,\R^N)} \to 0 \quad \text{ as } n \to \infty,
\]
by \Cref{lem:reals}. 
Clearly, also
\[
    \|\Delta \circ f - \Delta \circ f_{n}\|_{L^q(I,\R^N)} \to 0 \quad \text{ as } n \to \infty,
\]
by the Lipschitz property of $\De$ and $\|\dd(f,f_n)\|_{L^\infty(I)} \to 0$.

Conversely, assume that 
\begin{equation} \label{eq:Almgrenmetricconv}
    \|\Delta \circ f - \Delta \circ f_{n}\|_{W^{1,q}(I,\R^N)} \to 0 \quad \text{ as } n \to \infty.
\end{equation}
By Morrey's theorem, we have 
$\|\dd(f,f_n)\|_{L^\infty(I)} \to 0$ as $n \to \infty$. 
So there exists $n_0 \ge 1$ such that \eqref{eq:ptwrho} holds.
By \Cref{cor:oppositebound}, we have the opposite of \eqref{eq:Almgrenbound}: 
provided that all the derivatives exist at $x_0$,
\begin{align*}
    \mathbf s_1(f,f_n) (x_0)\le C(d)\, \|(\Delta \circ f - \Delta \circ f_n)'(x_0)\|_2, \quad  n \ge n_0.  
\end{align*}  
As above, this and the assumption \eqref{eq:Almgrenmetricconv} imply that there is a subsequence $(n_k)$ such that,
for almost every $x_0 \in I$,
\[
    \bs_1(f,f_{n_k})(x_0) \to 0 \quad \text{ as } k \to \infty.
\]
By the dominated convergence theorem, we conclude 
\[
    \|\bs_1(f,f_{n_k})\|_{L^q(I)} \to 0 \quad \text{ as } k \to \infty,
\]
and, in turn, by \Cref{lem:reals},
\[
    \|\bs_1(f,f_{n})\|_{L^q(I)} \to 0 \quad \text{ as } n \to \infty.
\]
For the domination, observe that, by \cite[Theorem 2.4 and Proposition 2.7]{De-LellisSpadaro11}, almost everywhere in $I$,
\begin{align*}
    \sqrt d\cdot \bs_1(f,f_{n}) &\le \|D f\|_2 + \|D f_n\|_2,
    \\
    \|D f_n\|_2 &\le C(d)\,\|(\Delta \circ f_{n})'\|_2 \le C(d)\, \big( \|(\De \o f -\Delta \circ f_{n})'\|_2+ \|(\Delta \circ f)'\|_2\big)  
\end{align*}
and the supremum over all $n\ge 1$ of the right-hand side is in $L^q(I)$, by assumption.
The proof of \Cref{thm:Almgren and convergence} is complete.

\section{The continuity problem for radicals} \label{sec:rad}

This section is devoted to the radical case, i.e., 
solutions of the equation
\[
    Z^d=g,
\]
where $g$ is a suitable function.
The goal is to prove the following theorem.

\begin{theorem} \label{thm:radicals}
    Let $d\ge 2$ be an integer. Let $I \subseteq \R$ be a bounded open interval.
    Let $g_n \to g$ in $C^d(\ol I,\C)$, i.e.,
    \begin{equation}   \label{eq:ass} 
        \|g-g_n\|_{C^{d}(\ol I)} \to 0 \quad  \text{ as } n \to \infty.
    \end{equation}
    Let $\La,\La_n : I \to \cA_d(\C)$ be the curves of unordered solutions of 
    \begin{equation} \label{eq:eqns}
        Z^d = g \quad \text{ and } \quad Z^d =g_n, \quad \text{ respectively}.
    \end{equation}
    Then
    \begin{equation} \label{eq:concl}
        \dd^{1,q}_{\on{rad},I}(\La,\La_n) \to 0
        \quad  \text{ as } n \to \infty,
    \end{equation}
    for all $1 \le q <d/(d-1)$.
\end{theorem}

The distance $\dd^{1,q}_{\on{rad},I}$ is induced by $\dd^{1,q}_{I}$;
see \Cref{def:srad}.

\begin{corollary} \label{cor:radicals}
    Let $d\ge 2$ be an integer. Let $I \subseteq \R$ be a bounded open interval.
    Let $g_n \to g$ in $C^d(\ol I,\C)$ as $n \to \infty$.
    Let $\la,\la_n : I \to \C$ be continuous functions satisfying  
    \begin{equation*}
        \la^d = g \quad \text{ and } \quad \la_n^d =g_n, \quad \text{ respectively}.
    \end{equation*}
    Then
    \begin{align} \label{eq:rads}
        \big\| |\la'| - |\la_n'| \big\|_{L^q(I)} \to 0 \quad \text{ as } n \to \infty,
        \\
        \label{eq:rade}
        \|\la_n'\|_{L^q(I)} \to \|\la'\|_{L^q(I)} \quad \text{ as } n \to \infty,
    \end{align}
    for all $1 \le q < d/(d-1)$.
\end{corollary}

\Cref{cor:radicals} will be proved in \Cref{proof:corrad}. 
It will be used in the proof of \Cref{thm:main1}.

\subsection{Unordered $d$-tuples of radicals}

Let us consider the set
\[
    \cA_{\on{rad},d}(\C) := \{ [z_1,\ldots,z_d] \in \cA_d(\C) : z_1^d = z_2^d = \cdots = z_d^d\}.
\]

\begin{definition}
    Let $d$ be a positive integer and $\th$ a $d$-th root of unity.
    For any $\la \in \C$ we define the unordered $d$-tuple 
    \[
        \bth{\la} := [\la, \th \la, \th^2 \la, \ldots, \th^{d-1} \la].
    \]
    Note that, for all $a \in \C$,
    $\bth{a\la} = a \bth{\la}  = \la \bth{a}$.
\end{definition}

We have the equivalent representation
\[
    \cA_{\on{rad},d}(\C) = \{[\la]_\th : \la \in \C\}.
\]
The restriction of the metric $\dd$ to $\cA_{\on{rad},d}(\C)$ is very simple: 

\begin{lemma} \label{lem:drad}
    For $\la, \mu \in \C$,      
    \[
        \mathbf d(\bth{\la} ,\bth{\mu} ) = \min_{1 \le j \le d} |\la - \th^j \mu|. 
    \]
    In particular, the map $\C \ni \la  \mapsto \bth{\la} \in \cA_{\on{rad},d}(\C)$ is Lipschitz with Lipschitz constant $\le 1$. 
\end{lemma}

\begin{proof}
    Clearly, the minimum over $\on{S}_d$ in the definition of $\mathbf d(\bth{\la},\bth{\mu})$ 
    is attained on a permutation induced by a rotation $\th^i \mapsto \th^{i+j}$. Hence
    \begin{align*}
        \mathbf d(\bth{\la} ,\bth{\mu} ) 
       &= \min_{\si \in \on{S}_d}  \Big(\frac{1}{d} \sum_{i=1}^d |\th^i\la - \th^{\si(i)}\mu|^2\Big)^{1/2}
       = \min_{1 \le j \le d}\Big(\frac{1}{d} \sum_{i=1}^d |\th^i\la - \th^{i+j}\mu|^2\Big)^{1/2}
       \\
       &= \min_{1 \le j \le d}\Big(\frac{1}{d} \sum_{i=1}^d |\la - \th^{j}\mu|^2\Big)^{1/2}
       = \min_{1 \le j \le d}  |\la - \th^{j}\mu|
    \end{align*}
    as claimed.
\end{proof}

\subsection{The distance $\dd_{\on{rad},E}^{1,q}$}

The distance $\dd^{1,q}_I$
from \Cref{def:dd} induces a distance $\dd^{1,q}_{\on{rad},I}$ on $W^{1,q}(I,\cA_{\on{rad},d}(\C))$.

\begin{definition}\label{def:srad}
    Let $f,g \in W^{1,q}(I,\cA_{\on{rad},d}(\C))$ and fix a $d$-th root of unity $\th$.
    By \Cref{p:Wselection-1}, there exist $\la, \mu \in W^{1,q}(I,\C)$ such that 
    \[
        f=\a{\la} + \a{\th \la}+  \cdots + \a{\th^{d-1}\la}, 
        \quad
        g=\a{\mu} + \a{\th \mu}+  \cdots + \a{\th^{d-1}\mu}. 
    \]
    For $x \in I$, let
    \begin{equation*}
        r(x) := \min \Big\{r \in \{0,1,\ldots,d-1\}:  |\la(x) - \th^r \mu(x)| =  \dd(f(x),g(x)) \Big\}
    \end{equation*}
    and set 
    \begin{equation*}
        \bs_{\on{rad},0}(f,g)(x) :=  \dd(f(x),g(x)).
    \end{equation*}
    For $x \in I$ such that $Df(x) = \sum_{i=0}^{d-1} \a{\th^i D\la(x)}$ and $Dg(x)  = \sum_{i=0}^{r-1} \a{\th^i D\mu(x)}$ exist in the sense of \Cref{d:diff},
    set 
    \begin{equation*}
        \bs_{\on{rad},1}(f, g)(x) := \max_\th   | D\la(x) - \th^{r(x)} D\mu(x)|,
    \end{equation*}
    where the maximum is taken over all $d$-th roots of unity.
    Then $\bs_{\on{rad},1}(f,g)(x)$ is defined for almost every $x \in I$. 
    It is independent of the choices of $\la$, $\mu$, and $\th$.

    For any measurable subset $E \subseteq I$, we set 
    \begin{equation*}
        \mathbf{d}^{1,q}_{\on{rad},E}( f,g )
        := \|\bs_{\on{rad},0}(f,g)\|_{L^\infty(E)}
        + \| \bs_{\on{rad},1}(f,g) \|_{L^q(E)}.
    \end{equation*}
    That $\bs_{\on{rad},i}(f,g)$, for $i=0,1$, are Borel measurable can be seen as in \Cref{lem:Borel}.
\end{definition}

\begin{lemma} \label{lem:ddrad1qsemidist}
    Let $I \subseteq \R$ be a bounded open interval and $E \subseteq I$ a measurable set. 
    Let $f, g \in W^{1,q}(I,\cA_{\on{rad},d}(\C))$. Then:
    \begin{enumerate}
        \item $\mathbf{d}^{1,q}_{\on{rad},E}( f,f )=0$.
        \item $\mathbf{d}^{1,q}_{\on{rad},E}( f,g )=0$ implies $f = g$ on $E$.  
        \item $\mathbf{d}^{1,q}_{\on{rad},E}( f,g )=\mathbf{d}^{1,q}_{\on{rad},E}( g,f )$.
    \end{enumerate}
    In particular, $\dd^{1,q}_{\on{rad},I}$ is a semimetric on $W^{1,q}(I,\cA_{\on{rad},d}(\C))$.
\end{lemma}

\begin{proof}
    (1) In this case, for any $x \in I$, 
    we have $r(x) =0$ and the assertion is obvious. 

    (2) If $\dd^{1,q}_{\on{rad},E}(f,g) = 0$ then $\dd(f,g)=0$ on $E$ (since $\dd(f,g)$ is continuous).

    (3) This is immediate from the definition. 
\end{proof}

\subsection{Proof of \Cref{cor:radicals}} \label{proof:corrad}

Recall that $\la$, $\la_n$ are absolutely continuous and belong to $W^{1,q}(I,\C)$, and that $\La := [\la]_\th$, $\La_n := [\la_n]_\th$  are the curves of 
unordered solutions of \eqref{eq:eqns}.
For each $1 \le j \le d$ and $x \in I$, 
where $\la'(x)$ and $\la_n'(x)$ exist,
\[
    \big| |\la'(x)| - |\la_n'(x)| \big| = \big| |\la'(x)| - |\th^j \la_n'(x)| \big| \le |\la'(x) - \th^j \la_n'(x)|.
\]
Fix $1 \le q < d/(d-1)$.
For $x \in I$, let $r(x) \in \{0,1,\ldots,d-1\}$ be as defined in \Cref{def:srad}.
Then 
\begin{align*}
    \big\| |\la'| - |\la_n'| \big\|_{L^q(I)}  &\le \|\la' - \th^r \la_n'\|_{L^q(I)} 
    \\
                                              &\le \|\bs_{\on{rad},1}(\La,\La_n)\|_{L^q(I)} 
                                              \le \dd^{1,q}_{\on{rad},I}(\La,\La_n).
\end{align*}
Thus \eqref{eq:rads} follows from \Cref{thm:radicals}.
Since
\begin{align*}
    \big| \|\la'\|_{L^q(I)} - \|\la_n'\|_{L^q(I)} \big| 
        &\le \big\| |\la'| - |\la_n'| \big\|_{L^q(I)},
\end{align*}
\eqref{eq:rads} implies \eqref{eq:rade}.
This ends the proof of \Cref{cor:radicals}.

\subsection{Ghisi and Gobbino's higher order Glaeser inequalities}
The following proposition is a variant of the results obtained in \cite{GhisiGobbino13}.

\begin{proposition}[{\cite[Proposition 1]{ParusinskiRainer15}}] \label{prop:m}
    Let $k\ge 1$ be an integer, $\ga \in (0,1]$, and $I \subseteq \R$ a bounded open interval.
    Let $g \in C^{k,\ga}(\ol I)$ be a complex valued function.
    Then there exists a nonnegative function $m \in L^p_w(I)$, for $p= \frac{k+\ga}{k +\ga-1}$, with 
    \begin{equation*}
        \|m\|_{p,w,I} \le C(k)\, \max \Big\{ |g^{(k)}|_{C^{0,\ga}(\ol I)}^{1/(k+\ga)}\, |I|^{1/p}, \|g'\|_{L^\infty(I)}^{1/(k+\ga)}  \Big\}
    \end{equation*}
    such that
    \begin{equation*} \label{eq:g}
        |g'(x)| \le m(x)\, |g(x)|^{1-1/(k+\ga)} \quad \text{ for almost every } x \in I.
    \end{equation*}
\end{proposition}

\begin{corollary}[{\cite[Corollary 2]{ParusinskiRainer15}}] \label{cor:m}
    Let $d$ be a positive integer.
    Let $I \subseteq \R$  be a bounded open interval.
    For any continuous function $f : I \to \C$ such that $f^d = g \in C^{d-1,1}(\ol I)$, 
    we have 
    $f' \in L^p_w(I)$, where $p = d/(d-1)$, and 
    \[
        \|f'\|_{p,w,I} \le C(d)\, \max\Big\{|g^{(d-1)}|_{C^{0,1}(\ol I)}^{1/d} |I|^{1/p}, \|g'\|_{L^\infty(I)}^{1/d}  \Big\}.
    \]
\end{corollary}

\subsection{Proof of \Cref{thm:radicals}}

Let $d\ge 2$ be an integer and $I \subseteq \R$ a bounded open interval.
Let $g_n \to g$ in $C^d(\ol I,\C)$ as $n\to \infty$. 
Let $\La,\La_n : I \to \cA_d(\C)$ be the curves of unordered solutions of \eqref{eq:eqns}. 
We have to show that
\begin{equation*} 
    \dd^{1,q}_{\on{rad},I}(\La,\La_n) \to 0
    \quad  \text{ as } n \to \infty,
\end{equation*}
for all $1 \le q <d/(d-1)$.

Let $\la,\la_n : I \to \C$ be continuous functions satisfying
\begin{equation*}
    \la^d = g \quad \text{ and } \quad \la_n^d =g_n.
\end{equation*}
By \Cref{cor:m}, $\la,\la_n \in W^{1,q}(I)$,  
for all $1 \le q <d/(d-1)$,
and
\[
    \La = [\la]_\th \quad \text{ and } \quad \La_n = [\la_n]_\th. 
\]

\begin{remark}
    By assumption, all derivatives of order $\le d$ of $g$ and $g_n$ extend continuously to 
    the endpoints of the interval $I$. 
    In particular, also $\la$ and $\la_n$ extend continuously to $\ol I$.
    For technical reasons, we will work with the compact interval $\ol I$. 
\end{remark}

We consider the zero set of $g$ in $\ol I$, 
\[
    Z_g := \{x \in \ol I : g(x)=0\},
\]
and its complement in $I$,
\[
    \Om_g := I \setminus Z_g  = \{x \in I : g(x) \ne 0\}.
\]
The set of accumulation points of $Z_g$ is denoted by $\on{acc} (Z_g)$.
Let
\[
    p := \frac{d}{d-1}
\]
for the rest of the section.

\subsection*{Strategy of the proof of \Cref{thm:radicals}}

\begin{description}
    \item[Step 0] We prove that 
        \begin{equation*} 
            \|\mathbf{s}_{\on{rad},0}(\La,\La_n)\|_{L^\infty(I)}  
            = \| \dd(\La,\La_n) \|_{L^\infty(I)} \to 0 \quad \text{ as } n \to \infty. 
        \end{equation*}
        Thus, it suffices to show that, for all $1 \le q < p$,
        \begin{equation} \label{eq:radtoshow}
            \| \mathbf{s}_{\on{rad},1}(\La,\La_n) \|_{L^q(I)} \to 0 \quad \text{ as } n \to \infty.
        \end{equation}
    \item[Step 1] We show that, for each $x \in \Om_g$,
        \[
            \mathbf{s}_{\on{rad},1}(\La,\La_n)(x)   \to 0 \quad \text{ as } n \to \infty,        
        \]
        and thus conclude, 
        using the dominated convergence theorem,
        that,
        for all $1 \le q < p$, 
        \[
            \| \mathbf{s}_{\on{rad},1}(\La,\La_n) \|_{L^q(\Om_g)} \to 0 \quad \text{ as } n \to \infty.
        \]
    \item[Step 2] We prove that for each $\ep>0$ there exist a neighborhood $U$ of $\on{acc}(Z_g)$ in $\ol I$
        and $n_0 \ge 1$ such that, 
        for all $1 \le q < p$, 
        \[
            \|\la_n'\|_{L^q(U)} \le C(d,q,|I|)\, \ep, \quad n \ge n_0.
        \]
        Note that for $x \in \on{acc}(Z_g)$ we have $\la(x)=0$ and $\la'(x) =0$ (if the latter exists). 
        The set $Z_g \setminus \on{acc}(Z_g)$ has measure zero.
    \item[Step 3] At this stage, it is not difficult to combine the results of Step 1 and Step 2 to 
        complete the proof of \eqref{eq:radtoshow} and hence of \Cref{thm:radicals}.
\end{description}

\subsection*{Step 0. Uniform convergence }

\begin{lemma} \label{lem:C0rad}
    If $g_n \to g$ in $C^0(\ol I)$
    as $n \to \infty$
    and $\la,\la_n \in C^0(\ol I)$ are such that $\la^d=g$ and $\la_n^d = g_n$, then 
    \begin{equation*} \label{eq:C0rad}
        \| \dd(\bth{\la}, \bth{\la_n} ) \|_{L^\infty(I)} \to 0 \quad \text{ as } n \to \infty. 
    \end{equation*}
\end{lemma}

\begin{proof}
    This follows from \Cref{cor:homeo}. 
    Here is a direct argument: for fixed $x \in I$,
    \[
        \prod_{j=1}^d |\la(x) - \th^j \la_n(x)| = |\la(x)^d - g_n(x)|= |g(x) -g_n(x)|
    \]
    so that $\dd(\bth{\la(x)}, \bth{\la_n(x)}) = \min_{1\le j \le d}|\la(x) - \th^j \la_n(x)| \le |g(x) -g_n(x)|^{1/d}$.
\end{proof}

\subsection*{Step 1. Continuity on $\Om_g$}

\begin{lemma} \label{lem:ptw1}
    Let $x \in \Om_g$. Then $\la'(x)$ and $\la_n'(x)$ exist for sufficiently large $n$.
    Let $1 \le j \le d$. 
    If 
    \[
        |\la(x) - \th^j \la_n(x)|   \to 0 \quad \text{ as } n \to \infty,
    \]
    then also
    \[
        |\la'(x) - \th^j \la_n'(x)|    \to 0 \quad \text{ as } n \to \infty.
    \]
\end{lemma}

\begin{proof}
    Fix $x \in \Om_g$.
    Then $g(x)\ne 0$ and there is $n_0\ge 1$ such that $g_n(x)\ne 0$ for all $n \ge n_0$.
    Fix $n \ge n_0$. 
    So $\la(x)\ne0$, $\la_n(x)\ne0$, and the derivatives $\la'(x)$ and $\la_n'(x)$ exist. 
    Differentiating $\la(x)^d =g(x)$ gives
    \[
        \la'(x) =  \la(x)\cdot \frac{1}{d} \frac{g'(x)}{g(x)},
    \]
    and analogously for $\la_n'(x)$.
    For each $1 \le j \le d$, we have 
    \begin{align*}
        |\la'(x) - \th^j \la_n'(x)| &= \Big| \la(x) \cdot \frac{1}{d} \frac{g'(x)}{g(x)} - \th^j\la_n(x)\cdot \frac{1}{d} \frac{g_n'(x)}{g_n(x)} \Big|
        \\
                                    &\le |\la(x) - \th^j \la_n(x)| \cdot \frac{1}{d} \Big| \frac{g'(x)}{g(x)} \Big| + \frac{|\la_n(x)|}{d}  
                                    \Big| \frac{g'(x)}{g(x)} - \frac{g_n'(x)}{g_n(x)}\Big|.
    \end{align*}
    We have $\big| \frac{g'(x)}{g(x)} - \frac{g_n'(x)}{g_n(x)}\big| \to 0$ as $n \to \infty$ and $|\la_n(x)|$ is bounded, 
    by \eqref{eq:ass}.
    The statement follows.
\end{proof}

\begin{lemma} \label{lem:ptw2}
    For each $x \in \Om_g$, 
    \[
        \mathbf s_{\on{rad},1}(\La,\La_n)(x)  \to 0 \quad \text{ as } n \to 0.
    \]
\end{lemma}

\begin{proof}
    This follows from \Cref{lem:ptw1}, since 
    \[
        |\la(x) - \th^{r(x)}\la_n(x)| = \dd(\La(x) ,\La_n(x)) \to 0 \quad \text{ as } n \to 0,
    \]
    by \Cref{lem:C0rad}.
    (Here $r(x)$ is independent of $n$, for $n$ sufficiently big.  Note that 
    if $x\in \Om_g$ then $r(x)$ is unique and constant on the connected components of $\Om_g$.)
\end{proof}

\begin{proposition} \label{lem:onOmg}
    For all $1 \le q < p$, 
    \[
        \| \mathbf{s}_{\on{rad},1}(\La,\La_n) \|_{L^q(\Om_g)} \to 0 \quad \text{ as } n \to \infty.
    \]
\end{proposition}

\begin{proof}
    Fix $1 \le q < p$.
    By \Cref{prop:m},
    for almost every $x \in I$,
    \[
        |\la'(x)-\th^{r(x)} \la_n'(x)| \le C(d) \, (m(x) + m_n(x)),
    \]
    where $m$ and $m_n$ are the nonnegative functions in $L^p_w$ from \Cref{prop:m} for $g$ and $g_n$, respectively.
    By the monotone convergence theorem, $m + \sup_{n\ge 1} m_n$ is a measurable nonnegative function belonging to $L^q(I)$, for $1 \le q <p$, 
    since $\{g_n : n \ge 1\}$ is a bounded set in  $C^{d-1,1}(\ol I)$.

    By \Cref{lem:ptw2} and the dominated convergence theorem, we may conclude that 
    \begin{align*}
        \int_{\Om_g} \big(\mathbf s_{\on{rad},1}(\La,\La_n)(x)\big)^q \, dx \to 0 \quad \text{ as } n \to \infty,
    \end{align*}
    which implies the assertion.
\end{proof}

\subsection*{Step 2. On accumulation points of $Z_g$}

\begin{lemma} \label{lem:Zg1}
    Let $g \in C^d(\ol I)$.
    If $x_0 \in \on{acc}(Z_g)$,
    then 
    \[
        g(x_0) = g'(x_0) = \cdots = g^{(d)}(x_0) = 0.
    \]
\end{lemma}

\begin{proof}
    By Taylor's formula,
    \[
        g(x) = g(x_0) + g'(x_0)(x-x_0) + \cdots + \frac{g^{(d)}(x_0)}{d!}(x-x_0)^d + o((x-x_0)^d).
    \]
    For contradiction, assume that $k \in \{1,\ldots,d\}$ is minimal with the property that $g^{(k)}(x_0) \ne 0$.
    If $Z_g \ni x_n \to x_0$, then
    \[
        0 = \frac{g^{(k)}(x_0)}{k!} +\frac{g^{(k+1)}(x_0)}{(k+1)!}(x_n -x_0) + \cdots +\frac{g^{(d)}(x_0)}{d!}(x_n-x_0)^{d-k} + o((x_n-x_0)^{d-k})
    \]
    leads to a contradiction.
\end{proof}

In the following, $I(x_0,\de)$ denotes the open $\de$-neighborhood of $x_0$ in $\ol I$ 
and $\ol I(x_0,\de)$ denotes its closure.

\begin{lemma} \label{lem:Zg2}
    Let $x_0 \in \on{acc}(Z_g)$.
    For every $\ep>0$ there exist $\de=\de(x_0,\ep)>0$  
    and $n_0 = n_0(x_0,\ep,\de) \ge 1$ such that
    \begin{equation}\label{eq:Zg2}
        \|\la_n'\|_{p,w,I(x_0,\de)}  \le C(d) \, \de^{1/p} \, \ep, \quad n \ge n_0.
    \end{equation} 
    In particular,
    \begin{equation} \label{eq:Zg3}
        \|\la_n'\|_{L^q(I(x_0,\de))} \le C(d)\, \Big(\frac{p}{p-q}\Big)^{1/q} |I(x_0,\de)|^{1/q}\, \ep, \quad n \ge n_0,
    \end{equation}
    for all $1 \le q <p$.
\end{lemma}

\begin{proof}
    By \Cref{lem:Zg1}, $g(x_0) = g'(x_0) = \cdots = g^{(d)}(x_0) = 0$.
    Fix $\ep>0$. By continuity, 
    there exists $\de>0$ such that 
    \[
        \|g\|_{C^d(\ol I(x_0,\de))} \le \frac{\ep^d}2.
    \]
    Furthermore, by \eqref{eq:ass}, there exists $n_0 \ge 1$ such that, for all $n \ge n_0$, 
    \[
        \|g-g_n\|_{C^d(\ol I(x_0,\de))} \le \frac{\ep^d}2
    \]
    and 
    \[
        |g_n^{(k)}(x_0)| \le \ep^d \de^{d-k}, \quad 0 \le k \le d.
    \]
    Then 
    \[
        \|g_n\|_{C^d(\ol I(x_0,\de))} \le \ep^d, \quad n \ge n_0.
    \]
    By Taylor's formula, for $x \in I(x_0,\de)$ and $n \ge n_0$,
    \begin{align*}
        |g_n'(x)| &= \Big|g_n'(x_0) + g_n''(x_0)(x-x_0) + \cdots + \int_{x_0}^x g_n^{(d)}(t) \frac{(x-t)^{d-2}}{(d-2)!}\, dt \Big|
        \le d\, \ep^d \de^{d-1}.
    \end{align*}  
    Hence,
    \[
        \|g_n'\|_{L^\infty(I(x_0,\de))} \le d\, \ep^d \de^{d-1}, \quad n \ge n_0.
    \]
    By \Cref{prop:m}, 
    we may conclude that 
    \begin{align*}
        \|\la_n'\|_{p,w,I(x_0,\de)} &\le C(d)\, \max \Big\{ |g_n^{(d)}|_{L^\infty(I(x_0,\de))}^{1/d}\, (2\de)^{1/p}, \|g_n'\|_{L^\infty(I(x_0,\de))}^{1/d}  \Big\}
        \\
                                    &\le C(d)\, \de^{1/p} \, \ep, \quad n \ge n_0,
    \end{align*}
    that is \eqref{eq:Zg2}. 
    Finally, \eqref{eq:Zg3} follows from \eqref{eq:qp}.
\end{proof}

\begin{proposition} \label{cor:Zg2}
    For every $\ep>0$ there exist a neighborhood $U$ of $\on{acc}(Z_g)$ in $\ol I$  
    and $n_0 \ge 1$ such that 
    \[
        \|\la_n'\|_{L^q(U)}  \le  C(d)\,  
        \Big (\frac {p}{p-q}\Big )^{1/q}{|U|} ^{1/q}\, \ep, \quad n \ge n_0,
    \]
    for all $1 \le q <p$.
\end{proposition}

\begin{proof}
    Let $\ep>0$.
    For each $x_0 \in \on{acc}(Z_g)$ 
    there exist $\de = \de(x_0,\ep)>0$ and $n_0 =n_0(x_0,\ep,\de) \ge 1$ such that 
    \[
        \|\la_n'\|_{L^q(I(x_0,\de))}  \le C(d) \,   \Big (\frac {p}{p-q}\Big )^{1/q} |I(x_0,\de)|^{1/q} \,\ep, 
        \quad n \ge n_0,
    \]
    for all $1 \le q <p$, by \Cref{lem:Zg2}. 
    Since $\on{acc}(Z_g)$ is compact, it is covered by finitely many $I_1,\ldots,I_s$ among the intervals $I(x_0,\de)$.
    Let $U = I_1 \cup \cdots \cup I_s$. By removing some of the intervals (see \Cref{lem:topline}), we may assume that each point of $U$ 
    belongs to exactly one or two of the intervals $I_\ell$.
    Then $U$ and the maximum of the corresponding $n_0$ are as required:
    \[
        \|\la_n'\|_{L^q(U)}^q 
        \le \sum_{\ell =1}^s \|\la_n'\|_{L^q(I_\ell)}^q 
        \le C(d)^q \,   \Big (\frac {p}{p-q}\Big ) \, \ep^q \sum_{i=1}^s |I_\ell|
        \le C(d)^q \,   \Big (\frac {p}{p-q}\Big ) \, 2 |U| \, \ep^q,
    \]
    and the statement follows. 
\end{proof}

\begin{lemma} \label{lem:topline}
    Let $\mathcal I = \{I_1,\ldots,I_s\}$ be a finite collection of bounded open intervals in $\R$.
    There exists a subset $\mathcal J \subseteq \cI$ such that 
    \[
        U= \bigcup_{I \in \cI} I = \bigcup_{I \in \cJ} I
    \] 
    and each point of $U$ belongs to exactly one or two intervals in $\cJ$.
\end{lemma}

\begin{proof}
    We may assume that $U = \bigcup_{I \in \cI} I$ is connected. 
    Let us write $I_\ell = (a_\ell,b_\ell)$. 
    By relabeling the intervals, we may assume that $a_1 \le a_2 \le \cdots \le a_s$.
    Let $\ell_1$ be defined by $b_{\ell_1} = \max\{b_\ell: a_\ell = a_1\}$.
    If $U = (a_{\ell_1},b_{\ell_1})$ we are done.
    Otherwise, $b_{\ell_1} \in U$, since $U$ is connected.
    Let $\ell_2$ be defined by $b_{\ell_2} = \max\{b_\ell: a_\ell < b_{\ell_1}\}$. 
    Then $b_{\ell_1}< b_{\ell_2}$.
    If $U = (a_{\ell_1},b_{\ell_1}) \cup (a_{\ell_2},b_{\ell_2})$ we are done.
    Otherwise, we repeat the procedure.
    It terminates with the right endpoint $b_{\ell_k}$ of $U$
    and 
    \[
        a_{\ell_i} < b_{\ell_{i-1}} < a_{\ell_{i+1}} < b_{\ell_{i}}, \quad 2 \le i \le k-1.
    \]
    This implies the statement.
\end{proof}

\subsection*{Step 3. End of proof of \Cref{thm:radicals}}

Fix $1 \le q <p$. 
If $x \in \on{acc}(Z_g)$, then $\la(x) = 0$ and $\la'(x)=0$ (if the derivative exists) so that 
\[
    \mathbf s_{\on{rad},1}(\La,\La_n)(x) =     |\la_n'(x)|.
\]
As $Z_g \setminus \on{acc}(Z_g)$ has measure zero, 
we have 
\begin{align*}
    \int_I  \big(\bs_{\on{rad},1}(\La,\La_n)(x) \big)^q \, dx 
    =
    \int_{\Om_g}\big(\bs_{\on{rad},1}(\La,\La_n)(x) \big)^q  \, dx 
    + 
    \int_{\on{acc}(Z_g)}   |\la_n'(x)|^q \, dx. 
\end{align*}
By \Cref{lem:onOmg} and \Cref{cor:Zg2},
both integrals on the right-hand side tend to $0$ as $n \to \infty$.
This shows \eqref{eq:radtoshow} and hence the proof of \Cref{thm:radicals} is complete.

\subsection{Variants of \Cref{thm:radicals}}

\begin{remark}
    In the setting of \Cref{thm:radicals}, 
    let $\la,\la_n : I \to \C$ be continuous functions satisfying $\la^d=g$ and $\la_n^d=g_n$, 
    fix a $d$-th root of unity $\th$,
    and, for $x \in I$ and $n\ge 1$,
    define
    \begin{equation*}
        r_n(x) = \min\big\{r \in \{0,1,\ldots,d-1\} : |\la(x) - \th^r \la_n(x)| < \dd(\La(x) , \La_n(x)) + \tfrac{1}{n} \big\}.
    \end{equation*}
    As in \Cref{lem:Borel}, one sees that $r_n : I \to \{0,1,\ldots,d-1\}$ is Borel measurable. 
    Thus we can replace $r(x)$ by $r_n(x)$ in the definition of 
    $\dd^{1,q}_{\on{rad},I}(\La,\La_n)$ and get a slightly stronger version of \Cref{thm:radicals}.
    In fact, we have 
    \[
        |\la(x) - \th^{r_n(x)} \la_n(x)| < \dd(\La(x), \La_n(x) ) + \tfrac{1}{n} \to 0 \quad \text{ as } n \to \infty
    \]
    so that \Cref{lem:ptw2} remains true.
\end{remark}

\begin{corollary} \label{cor:lifts}
    Let $d\ge 2$ be an integer and $I \subseteq \R$ a bounded open interval.
    Let $g_n \to g$ in $C^d(\ol I)$ as $n\to \infty$.
    Assume that $\la_n : I \to \C$ is a continuous function satisfying $\la_n^d = g_n$, for all $n \ge 1$,
    and that there is a continuous function $\la : I \to \C$ 
    such that, for all $x \in I$, 
    \begin{equation} \label{eq:assradptw}
        \la_n(x) \to \la(x) \quad \text{ as } n \to \infty.
    \end{equation}
    Then $\la^d = g$ and 
    \begin{equation*}
        \|\la - \la_n\|_{L^\infty(I)} + \|\la' - \la_n'\|_{L^q(I)} \to 0 \quad \text{ as } n \to \infty,
    \end{equation*}
    for all $1 \le q <d/(d-1)$.
\end{corollary}

\begin{proof}
    It is clear that $\la^d = g$. So $\la_n'$ and $\la'$ exist almost everywhere in $I$ 
    and belong to $L^q(I)$ for all $1 \le q <d/(d-1)$, by \Cref{cor:m}.
    By \Cref{lem:ptw1}, we may conclude that 
    \[
        \la_n'(x) \to \la'(x)  \quad \text{ as } n \to \infty,
    \] 
    for each $x \in \Om_g$. Thus the dominated convergence theorem implies that, for $1 \le q <p=d/(d-1)$, 
    \[
        \|\la' - \la_n'\|_{L^q(\Om_g)} \to 0  \quad \text{ as } n \to \infty;
    \]
    the domination follows from \Cref{prop:m} as in the proof of \Cref{lem:onOmg}.
    Using \Cref{cor:Zg2}, it is easy to conclude (as in Step 3) that, for $1 \le q <p$, 
    \[
        \|\la' - \la_n'\|_{L^q(I)} \to 0  \quad \text{ as } n \to \infty.
    \]
    Now fix $x_0 \in I$. Since $\la$ and $\la_n$ are absolutely continuous, we have, for any $x \in I$, 
    \begin{align*}
        |\la(x) - \la_n(x)| &= \Big| \la(x_0) - \la_n(x_0) + \int_{x_0}^x \la'(t) - \la_n'(t) \, dt\Big|  
        \\
                            &\le |\la(x_0) - \la_n(x_0)| + \|\la' - \la_n'\|_{L^1(I)}. 
    \end{align*}
    Consequently, 
    \begin{equation*}
        \|\la - \la_n\|_{L^\infty(I)}  \to 0 \quad \text{ as } n \to \infty,
    \end{equation*}
    and the proof is complete.
\end{proof}

This raises the question as to whether the assumption \eqref{eq:assradptw} on $\la_n$ and $\la$ in \Cref{cor:lifts} 
can always be fulfilled.
Not for every 
continuous function $\la$ satisfying $\la^d = g$ on $I$
there exist continuous functions $\la_n$ satisfying $\la_n^d = g_n$ on $I$, for $n\ge 1$, 
such that \eqref{eq:assradptw} holds for almost every $x \in \Om_g$. See \Cref{ex:lifts}.

\begin{example} \label{ex:lifts}
    (1) Consider $g_n (x) = x^2+ \frac 1 n \to g(x) = x^2$. The continuous solutions of $Z^2 = g_n$ 
    converge to either $|x|$ or $-|x|$, but not to $x$ or $-x$.

    (2) Let $g_n (x) = x + i \frac 1 n \to g(x) = x$.  For $n$ fixed, 
    $g_n$ never vanishes, so there are exactly two continuous square roots of $g_n$. Since $\Im (g_n (x) ) > 0$ for all $x$,
    one solution stays in the first quadrant and approaches $\sqrt x$ for $x>0$ and $i \sqrt {|x|}$ 
    for $x<0$ as $n\to \infty$.  The other one approaches $- \sqrt x$ for $x>0$ and $- i \sqrt {|x|}$ 
    for $x<0$.  Now consider another sequence $h_n (x) = x - i \frac 1 n \to g(x) = x$. 
    Since $\Im (h_n (x) ) < 0$ for all $x$,
    one solution stays in the forth quadrant and approaches $\sqrt x$ for $x>0$ and $- i \sqrt {|x|}$ 
    for $x<0$ as $n\to \infty$. The other one approaches $- \sqrt x$ for $x>0$ and $i \sqrt {|x|}$ 
    for $x<0$.
\end{example}

We end this section with a version of \Cref{thm:radicals} in the setting of \cite{GhisiGobbino13} for radicals with real exponents.

\begin{proposition}
    Let $k \in \N$ and $\ga \in (0,1]$.
    Let $I \subseteq \R$ be a bounded open interval.
    Let $g_n \to g$ in $C^{k+1}(\ol I,\R)$ as $n \to \infty$.
    Let $f,f_n : I \to \R$ be continuous functions satisfying
    \[
        |f|^{k+\ga} = |g| \quad \text{ and } \quad |f_n|^{k+\ga} = |g_n|. 
    \]
    For each $x \in I$ and $n\ge 1$, let 
    \[
        r(x) = \min\Big\{r \in \{0,1\} : |f(x) - (-1)^r f_n(x)| = \min_{j\in \{0,1\}} |f(x) - (-1)^j f_n(x)|\Big\}.
    \]
    Then $f$ and $f_n$ are absolutely continuous and satisfy 
    \[
        \|f - (-1)^r f_n\|_{L^\infty(I)} \to 0 \quad \text{ and } \quad  \|f' - (-1)^r f_n'\|_{L^q(I)} \to 0 \quad \text{ as } n \to \infty,
    \]
    for all $1 \le q< \frac{k+\ga}{k+ \ga-1}$.

    In particular,
    for $\la := |g|^{1/(k+\ga)}$ and $\la_n := |g_n|^{1/(k+\ga)}$,
    we have
    \[
        \|\la - \la_n\|_{L^\infty(I)} \to 0 \quad \text{ and } \quad  \|\la' - \la_n'\|_{L^q(I)} \to 0 \quad \text{ as } n \to \infty,
    \]
    for all $1 \le q< \frac{k+\ga}{k+ \ga-1}$.
\end{proposition}

\begin{proof}
    By \cite[Theorem 2.2]{GhisiGobbino13}, each continuous solution $f : I \to \R$ of
    \begin{equation*}
        |f|^{k+\ga} = |g|
    \end{equation*}
    is absolutely continuous and $f' \in L^p_w(I)$ with $p := \frac{k+\ga}{k+ \ga-1}$ and 
    \[
        \|f'\|_{p,w,I} \le C(k)\, \max \Big\{ |g^{(k)}|_{C^{0,\ga}(\ol I)}^{1/(k+\ga)}\, |I|^{1/p}, \|g'\|_{L^\infty(I)}^{1/(k+\ga)}  \Big\}.
    \]

    Let $f_n : I \to \R$ be a continuous function satisfying $|f_n|^{k+\ga}= |g_n|$.
    In analogy to \Cref{cor:Zg2}, we see that for each $\ep>0$ there exist a neighborhood $U$ of $\on{acc}(Z_g)$ in $\ol I$ 
    and $n_0 \ge 1$ such that 
    \[
        \|f_n'\|_{L^q(U)} \le C(k,p,q) \, |U|^{1/q} \ep, \quad n \ge n_0.
    \]
    Now fix $x \in \Om_g$. Then $f'(x)$ and $|f|'(x)$ exist and satisfy $f'(x) = \on{sgn} f(x) \cdot |f|'(x)$ and 
    \[
        |f|'(x) = |f(x)| \cdot \frac{1}{k+\ga} \frac{|g|'(x)}{|g(x)|}.
    \] 
    For large enough $n$, $f_n(x) \ne 0$ and $f_n'(x)$ exists.
    As in \Cref{lem:ptw1}, we conclude that, for $j \in \{0,1\}$, 
    \[
        |f'(x) - (-1)^j f'_n(x)| \to 0 \quad \text{ as } n \to \infty,
    \]
    provided that $|f(x) - (-1)^j f_n(x)| \to 0$. A simple modification of \Cref{lem:C0rad} gives $\|f - (-1)^r f_n\|_{L^\infty(I)} \to 0$ and hence the proposition follows
    by an application of the dominated convergence theorem, as in \Cref{lem:onOmg}.

    For $f := |g|^{1/(k+\ga)}$ and $f_n := |g_n|^{1/(k+\ga)}$, for $n \ge 1$, we have $r \equiv 0$. 
\end{proof}

\subsection{Optimality of the result} \label{ssec:optimality}

By \Cref{prop:m}, in the setting of \Cref{cor:radicals} the set $\{\la'\} \cup \{\la_n'  : n \ge 1\}$ is bounded in 
$L^p_w(I)$, where $p := d/(d-1)$. But, in general, 
\[
    \big\| |\la'| - |\la_n'| \big\|_{p,w,I} \not\to 0 \quad \text{ as } n \to \infty,
\]
as seen in \Cref{ex:Lpw}.

\begin{example} \label{ex:Lpw}
    Let $d \in \R_{>1}$ and set $p:= d/(d-1)$.
    Let $g,g_n : \R \to \R$, $n \ge 1$, be given by $g(x):=x$ and $g_n(x) := x+1/n^{p}$.  
    For $x \in (0,1)$, consider
    $\la(x):= x^{1/d}$ and $\la_n(x) :=  (x+1/n^{p})^{1/d}$.     
    Then, for $x \in (0,1)$,
    \begin{align*}
        |\la'(x)|- |\la_n'(x)| = \frac{1}{d} (x^{-1/p} - (x+1/n^p)^{-1/p}) > 0. 
    \end{align*}
    For $r>0$,  we have  
    \begin{align*}
        \{x \in (0,1) : |\la'(x)|-|\la_n'(x)|>r\} &\supseteq  \{x \in (0,1): x^{-1/p} - n > dr\} 
        \\
                                                  &= (0,(dr+n)^{-p}).
    \end{align*}
    Thus 
    \begin{align*}
        \big\||\la'| - |\la_n'|\big\|_{p,w,(0,1)} 
                                      &\ge \sup_{r>0}  \frac{r}{dr+n} = \frac{1}{d}. 
    \end{align*}

    On the other hand,\footnote{We do not have an example with $\|\la_n'\|_{p,w,I} \not\to \|\la'\|_{p,w,I}$ as $n \to \infty$.} 
    \[
        \|\la_n'\|_{p,w,(0,1)} \to \|\la'\|_{p,w,(0,1)} \quad \text{ as } n \to \infty.
    \]
    Indeed, 
    \[
        \{x \in(0,1) : x^{-1/p} > dr\} =  (0, \min\{1,(dr)^{-p}\}) 
    \]
    so that 
    \[
        \|\la'\|_{p,w,(0,1)} = \max \Big\{ \sup_{0<r\le 1/d} r, \sup_{r>1/d} \frac{r}{dr} \Big\} = \frac{1}d. 
    \]
    Moreover,
    \[
        \{x \in(0,1) : (x+\tfrac{1}{n^p})^{-1/p} > dr\} =
        \begin{cases}
            (0,\min\{1, (dr)^{-p} - n^{-p}\}) & \text{ if } n > dr, 
            \\
            \emptyset & \text{ if } n\le dr.
        \end{cases}
    \]
    Hence, as $(dr)^{-p} - n^{-p} < 1$ if and only if $r> \frac{n}{d(n^p+1)^{1/p}}$, 
    \begin{align*}
        \|\la_n'\|_{p,w,(0,1)} &= \max\Big\{ \sup_{0< r \le\frac{n}{d(n^p+1)^{1/p}} } r, \sup_{r>\frac{n}{d(n^p+1)^{1/p}}} r ((dr)^{-p} - n^{-p})^{1/p}   \Big\}
        \\
                               &= \max\Big\{ \frac{n}{d(n^p+1)^{1/p}} , \sup_{r>\frac{n}{d(n^p+1)^{1/p}}} \frac{(n^p-(dr)^p)^{1/p}}{dn}   \Big\}
                               \\
                               &= \max\Big\{ \frac{n}{d(n^p+1)^{1/p}} , \frac{n}{d(n^p+1)^{1/p}}   \Big\}=\frac{n}{d(n^p+1)^{1/p}}
    \end{align*}
    which tends to $1/d$ as $n \to \infty$.
\end{example}

\section{Monic polynomials} \label{sec:poly}

Let us gather basic facts on monic complex polynomials of degree $d$,
\[
    P_{a}(Z) = Z^d + \sum_{j=1}^d a_jZ^{d-j} \in \C[Z].
\]
We often identify $P_a$ with its coefficient vector $a=(a_1,a_2,\ldots,a_d) \in \C^d$ 
so that $\C^d$ is the space of all monic complex polynomials of degree $d$.

\subsection{Cauchy bound} \label{ssec:Cauchybound} 
If $\la \in \C$ is a root of $P_a(Z) \in \C[Z]$, 
then 
\begin{equation} \label{eq:Cauchybound}
    |\la| \le 2 \, \max_{1 \le  j \le d}|a_j|^{1/j}.
\end{equation}
See e.g.\ \cite[IV Lemma 2.3]{Malgrange66}.

\subsection{Uniform H\"older continuity of the roots} 

\begin{lemma}[{\cite[IV Lemma 2.5]{Malgrange66}}] \label{lem:Malgrange.IV.2.5}
    Let $I = \ol I\subseteq \R$ be a bounded closed interval.
    Let $P_a$ be a monic polynomial of degree $d$ with coefficient vector $a \in C^{0,\ga}(I,\C^d)$, where
    $\ga \in (0,1]$.
    Let $\la \in C^0(I)$ be a continuous root of $P_a$ on $I$.
    Then $\la \in C^{0,\ga/d}(I)$ and
    \begin{equation*}
        |\la(x)- \la(y)| \le H \, |x-y|^{\ga/d}, \quad x,y \in I,
    \end{equation*}
    where
    \begin{equation} \label{e:H}
        H:= 4d\, \max_{1 \le j \le d} \|a_j\|_{C^{0,\ga}(I)}^{1/j}.
    \end{equation}
\end{lemma}

\begin{corollary} \label{cor:Hoelder}
    Let $I = \ol I\subseteq \R$ be a bounded closed interval.
    Let $P_a$ be a monic polynomial of degree $d$ with coefficient vector $a \in C^{0,\ga}(I,\C^d)$, where
    $\ga \in (0,1]$.
    Let $\La : I \to \cA_d(\C)$ be the curve of unordered roots of $P_a$.
    Then 
    \[
        \dd(\La(x),\La(y)) \le H\, |x-y|^{\ga/d}, \quad x,y \in I,
    \]
    for $H$ in \eqref{e:H}.
\end{corollary}

\begin{proof}
    Let $\la : I \to \C^d$ be a continuous parameterization of the roots of $P_a$ so that $\La = [\la]$.
    In view of
    \begin{align*}
        \dd([\la(x)],[\la(y)]) &= \min_{\si \in \on{S}_d} \frac{1}{\sqrt d}\, \| \la(x) - \si \la(y)\|_2 
        \\
                               &\le   \frac{1}{\sqrt d}\, \| \la(x) -  \la(y)\|_2 \le \max_{1\le j \le d} |\la_j(x) - \la_j(y)|,
    \end{align*}
    the statement follows from \Cref{lem:Malgrange.IV.2.5}.
\end{proof}

\begin{lemma} \label{lem:Hoelder2}
    \Cref{lem:Malgrange.IV.2.5} and \Cref{cor:Hoelder} also hold with $H$ replaced by 
    \begin{equation} \label{e:H1}
        H_1:= 2d\,A^{1/d} \big(1+B + B^2 + \cdots + B^{d-1} \big)^{1/d},
    \end{equation}
    where
    \[
        A := |a|_{C^{0,\ga}(I)} \quad \text{ and } \quad B:=2 \max_{1 \le j \le d} \|a_j\|_{L^\infty(I)}^{1/j}.
    \]
\end{lemma}

\begin{proof}
    We modify the proof of \cite[IV Lemma 2.5]{Malgrange66}.
    
    First we show the following claim.
    Let $\la_1,\ldots,\la_d \in \C$ and $\mu_1,\ldots,\mu_d \in \C$ be the roots of $P_a$ and $P_b$, respectively.
    Assume that, for $\al,\be>0$, 
    \[
        \max_{1\le j \le d}|a_{j} - b_{j}| \le \al \quad \text{ and } \quad 2\max_{1\le j \le d}|a_{j}|^{1/j} \le \be. 
    \]
    Then, for each $i$ there exists $j$ such that 
    \[
        |\la_{i} - \mu_{j}| \le  \al^{1/d} \big(1+ \be + \be^2 + \cdots + \be^{d-1} \big)^{1/d}.  
    \]
    To see this, fix $i$. Then 
   \begin{align*}
       \prod_{j=1}^d |\la_{i}-\mu_{j}| = |P_{b}(\la_{i})| = |P_{b}(\la_{i})- P_{a}(\la_{i})| 
       = \Big| \sum_{k=1}^d (b_{k} - a_{k}) \la_{i}^{d-k} \Big| \le \al \sum_{k=1}^d \be^{d-k},  
   \end{align*}
   using \eqref{eq:Cauchybound}, and the claim follows.

   Now suppose we are in the setting of \Cref{lem:Malgrange.IV.2.5}.
    Fix $x< y \in I$ and let $\la_1 = \la(x),\la_2,\ldots,\la_d$ be the roots of $P_{a(x)}$.
    Let $K$ be the union of the closed disks with radius $\frac{H_1}{2d} |x-y|^{\ga/d}$ and centers $\la_j$.
    Then $\la([x,y]) \subseteq K$, by the claim, and, since $\la$ is continuous, 
    $\la([x,y])$ is contained in the connected component of $K$ containing $\la_1$. 
    This implies \Cref{lem:Malgrange.IV.2.5}, and in turn \Cref{cor:Hoelder}, with $H_1$ instead of $H$. 
\end{proof}

\subsection{The solution map} \label{ssec:solutionmap}

The elementary symmetric polynomials induce a bijective map $a=(a_1,\ldots,a_d) : \cA_d(\C) \to \C^d$,
\[
    a_j([z_1,\ldots,z_d]) := (-1)^j \sum_{i_1< \ldots < i_j} z_{i_1} \cdots z_{i_j}, \quad 1 \le j \le d.
\]
Let $\La  : \C^d \to \cA_d(\C)$ be the inverse of $a$. Then $\La(a)$ is the unordered $d$-tuple consisting 
of the $d$ roots of $P_a$ (with multiplicities).
The map $a : \cA_d(\C) \to \C^d$ is a homeomorphism as seen in the following lemma.

\begin{lemma} \label{lem:homeo} 
    For $K \ge 1$  we have:
    \begin{enumerate}
        \item The map $a : \cA_d(\C) \to \C^d$ is locally Lipschitz:
            if $[z_0], [z_1], [z_2] \in \cA_d(\C)$ and $\dd([z_0],[z_i]) \le K$, for $i=1,2$, then 
            \[
                \|a([z_1]) - a([z_2])\|_2 \le C(d,K)\, \dd([z_1],[z_2]).
            \]
        \item The map $\La : \C^d \to \cA_d(\C)$ is locally $1/d$-H\"older: 
            if $a_1,a_2 \in \C^d$ and $\|a_i\|_2 \le K$, for $i=1,2$, then 
            \[
                \dd(\La(a_1),\La(a_2)) \le C(d,K)\, \|a_1-a_2\|_2^{1/d}. 
            \]
    \end{enumerate}
\end{lemma}

\begin{proof}
    (1) The polynomial map $a : \C^d \to \C^d$ clearly is locally Lipschitz: for $z_1, z_2 \in \C^d$ with $\|z_i\|_2 \le K$, for $i =1,2$,  
    we have 
    \[
        \|a(z_1) - a(z_2)\|_2 \le C(d,K)\, \|z_1 - z_2\|_2.
    \]
    But the left-hand side equals $\|a([z_1]) - a([z_2])\|_2$ and on the right-hand side we may replace $z_2$ by 
    $\si z_2$ for any $\si \in \on{S}_d$. This implies (1).

    (2) This follows from  \Cref{cor:Hoelder} and \Cref{lem:Hoelder2} applied to the family $a(t) := ta_1 + (1-t)a_2$, $\ga=1$, 
    $x=0$ and $y=1$.  Then $A= \|a_1-a_2\|_2$ and  $B \le C(d,K)$ so that $H_1 \le C(d,K) \, \|a_1-a_2\|_2^{1/d}$.
\end{proof}

\begin{corollary} \label{cor:homeo}
    Let $K \subseteq \R^m$ be a compact set.
    Then the map $\La_* : C^0(K,\C^d) \to C^0(K,\cA_d(\C))$, $\La_*(a) := \La \o a$, is locally $1/d$-H\"older:
    if $a_1,a_2 \in C^0(K,\C^d)$ and $\|a_i\|_{C^{0}(K,\C^d)} \le L$, for $i=1,2$, then
    \[
        \| \dd(\La_*(a_1) , \La_*(a_2)) \|_{C^0(K)} \le C(d,L)\, \|a_1 -a_2\|_{C^0(K,\C^d)}^{1/d}. 
    \]
\end{corollary}

\begin{proof}
    This is immediate from \Cref{lem:homeo}(2).
\end{proof}

\subsection{Tschirnhausen form}

We say that a monic polynomial 
\[
    P_a(Z) = Z^d + \sum_{j=1}^d a_j Z^{d-j}
\]
is in \emph{Tschirnhausen form} if $a_1=0$.
Every $P_a$ can be put in Tschirnhausen form by the substitution, called \emph{Tschirnhausen transformation},
\[
    P_{\tilde a}(Z) = P_a(Z - \tfrac{a_1}d) = Z^d + \sum_{j=2}^d \tilde a_j Z^{d-j}.
\]
Note that 
\begin{equation} \label{eq:Tschirn}
    \tilde a_j  = \sum_{i=0}^j C_i a_i a_1^{j-i}, \quad 2 \le j \le d,
\end{equation}
where the $C_i$ are universal constants and $a_0=1$.

For clarity of notation, we consistently equip the coefficients of polynomials in Tschirnhausen form with a ``tilde''.

\subsection{Splitting} \label{ssec:split}

The following well-known lemma (see e.g.\ \cite{AKLM98} or \cite{BM90}) is a consequence of the 
inverse function theorem. 

\begin{lemma} \label{split}
    Let $P_a = P_b P_c$, where $P_b$ and $P_c$ are monic complex polynomials without common root.
    Then for $P$ near $P_a$ we have $P = P_{b(P)} P_{c(P)}$
    for analytic mappings of monic polynomials $P \mapsto b(P)$ and $P \mapsto c(P)$,
    defined for $P$ near $P_a$, with the given initial values.
\end{lemma}

\begin{proof}
    The splitting $P_a = P_b P_c$ defines on the coefficients a polynomial mapping $\vh$ such that $a = \vh(b,c)$, 
    where $a=(a_i)$, $b=(b_i)$, and $c=(c_i)$. The Jacobian determinant  
    $\det d\vh(b,c)$ equals the resultant of $P_b$ and $P_c$ which is non-zero by assumption. 
    Thus $\vh$ can be inverted locally. 
\end{proof}

If $P_{\tilde a}$ is in Tschirnhausen form and if $\tilde a \ne 0$, then $P_{\tilde a}$ \emph{splits}, i.e., 
$P_{\tilde a} = P_b P_c$ for monic polynomials $P_b$ and $P_c$ with positive degree and without common zero. 
For, if $\la_1,\ldots,\la_d$ denote the roots of $P_{\tilde a}$ and they all coincide, then since 
\[
    \la_1+\cdots+\la_d = \tilde a_1 = 0
\] 
they all must vanish, contradicting $\tilde a \ne 0$.

Let $\tilde a_1,\ldots,\tilde a_d$ denote the coordinates in $\C^d$. 
Fix $k \in \{2,\ldots,d\}$ and let $\tilde p \in \C^{d} \cap \{\tilde a_1 =0, \,\tilde a_k \ne 0\}$; 
$\tilde p$ corresponds to the polynomial $P_{\tilde a}$ in Tschirnhausen form. 
We associate the polynomial 
\begin{gather*}
    Q_{\underline a}(Z) := \tilde a_k^{- d/k} P_{\tilde a} (\tilde a_k^{1/k} Z)  
    = Z^d + \sum_{j=2}^d \tilde a_k^{- j/k} \tilde a_j Z^{d-j}= Z^d + \sum_{j=2}^d \ul a_j Z^{d-j},\\
    \underline a_j := \tilde a_k^{- j/k} \tilde a_j, \quad j = 1,\ldots, d, 
\end{gather*}
where some branch of the radical is fixed.  
Then $Q_{\underline a}$ is in Tschirnhausen form and $\underline a_k =1$;
it corresponds to a point $\underline p \in \C^{d} \cap \{\underline a_k =1\}$. 
By \Cref{split}, we have a splitting $Q_{\underline a} = Q_{\underline b} Q_{\underline c}$ on some open ball 
$B(\underline p,\rh)$ 
centered at $\underline p$ with radius $\rh>0$. 
In particular, there exist analytic functions $\ps_i$ on $B(\underline p,\rh)$ such that
\begin{equation*} 
    \underline b_i = 
    \ps_i \big(\tilde a_k^{-2/k} \tilde a_2, \tilde a_k^{-3/k} \tilde a_3, \ldots, \tilde a_k^{-d/k} \tilde a_d\big), 
    \quad i = 1,\ldots,\deg Q_{\underline b}.
\end{equation*}
The splitting $Q_{\underline a} = Q_{\underline b} Q_{\underline c}$ induces a splitting 
$P_{\tilde a} = P_b P_c$, where 
\begin{equation} \label{eq:bj}
    b_i =  \tilde a_k^{i/k} 
    \ps_i \big(\tilde a_k^{-2/k} \tilde a_2, \tilde a_k^{-3/k} \tilde a_3, \ldots, \tilde a_k^{-d/k} \tilde a_d\big), 
    \quad i = 1,\ldots,d_b := \deg P_b;
\end{equation}
likewise for $c_j$. 
Shrinking $\rh$ slightly, we may assume that $\ps_i$ and all its partial derivatives are bounded on $B(\underline p, \rh)$. 
Let $\tilde b_j$ denote the coefficients of the polynomial $P_{\tilde b}$ resulting from $P_b$ by the  
Tschirnhausen transformation. Then, by \eqref{eq:Tschirn}, 
\begin{equation} \label{eq:tildebj}
    \tilde b_i = \tilde a_k^{i/k} 
    \tilde \ps_i \big(\tilde a_k^{-2/k} \tilde a_2, \tilde a_k^{-3/k} \tilde a_3, \ldots, \tilde a_k^{-d/k} \tilde a_d\big), 
    \quad i = 2,\ldots,d_b,
\end{equation}
for analytic functions $\tilde \ps_i$ which, together with all their partial derivatives, are bounded on $B(\underline p,\rh)$.

\subsection{Universal splitting of polynomials in Tschirnhausen form} \label{universal}

The set 
\begin{equation*} 
    K:= \bigcup_{k=2}^d\{(0,\ul a_2, \ldots,\ul a_d) \in \C^{d} : \ul a_1=0,~ \ul a_k=1,~ |\ul a_j| \le 1 \text{ for }j \ne k\}
\end{equation*}
is compact. For each point $\underline p \in K$
there exists $\rh(\underline p)>0$ such that we have a splitting  
$P_{\tilde a} = P_b P_{c}$ on the open ball $B(\underline p,\rh(\underline p))$, and we fix this splitting; 
cf.\ \Cref{ssec:split}.
Choose a finite subcover of $K$ by open balls $B(\underline p_\de,\rh_\de)$, $\de \in \De$. 
Then there exists $\rh>0$ such that for every $\underline p \in K$ there is a $\de \in \De$ such that 
$B(\underline p,\rh) \subseteq B(\underline p_\de,\rh_\de)$.

To summarize, for each integer $d \ge 2$ we have fixed 
\begin{enumerate}
    \item a finite cover $\cB$ of $K$ by open balls $B$,
    \item a splitting $P_{\tilde a} = P_b P_{c}$ on each $B \in \cB$ together with analytic functions $\ps_i$ and 
        $\tilde \ps_i$ which are bounded on $B$ along with all their partial derivatives,
    \item a positive number $\rh$ such that for each $\underline p \in K$ there is a $B \in \cB$ such that 
        $B(\underline p,\rh) \subseteq B$ (note that $\rh$ is a Lebesgue number of the cover $\cB$).  
\end{enumerate}

\begin{remark} \label{rem:universal}
    Additionally,
    there exists $\ch>0$ such that for all pairs $P_{\tilde a_1}=P_{b_1} P_{c_1}$ and $P_{\tilde a_2}=P_{b_2} P_{c_2}$ 
    in some fixed $B \in \cB$, where $\ul a_{1,k}=\ul a_{2,k} =1$,
    \begin{equation}\label{eq:remu1}
        \min_{P_{b_1}(\mu_1)=0,\, P_{c_2}(\nu_2)=0}|\tilde a_{2,k}^{1/k} \mu_1 - \tilde a_{1,k}^{1/k}\nu_2 | 
        > \ch \cdot |\tilde a_{1,k}|^{1/k}|\tilde a_{2,k}|^{1/k}.
    \end{equation} 
    Indeed, by shrinking the balls $B(\ul p,\rh(\ul p))$, we may assume that  
    \begin{equation} \label{eq:remu2}
        \min_{Q_{\ul b_1}(\ul \mu_1)=0,\, Q_{\ul c_2}(\ul \nu_2)=0}|\ul \mu_1 - \ul \nu_2 | 
        > \ch_{\ul p}>0, 
    \end{equation}
    for all pairs $Q_{\ul a_1}=Q_{\ul b_1} Q_{\ul c_1}$ and $Q_{\ul a_2}=Q_{\ul b_2} Q_{\ul c_2}$ in $B(\ul p,\rh(\ul p))$. 
    Then take a finite subcover $\cB$ of $K$ and let $\ch$ be the minimum of the respective $\ch_{\ul p}$. 
    Multiplying \eqref{eq:remu2} by $|\tilde a_{1,k}|^{1/k}|\tilde a_{2,k}|^{1/k}$ 
    and observing that the roots of $P_{b_i}$, $P_{c_i}$ are the roots of  $Q_{\ul b_i}$, $Q_{\ul c_i}$ 
    times $\tilde a_{i,k}^{1/k}$, gives \eqref{eq:remu1}.
\end{remark}

\begin{definition} \label{def:universal}
    We will refer to the data fixed in (1), (2), and (3) including $\ch>0$ such that \eqref{eq:remu1} holds 
    as a \emph{universal splitting of polynomials of degree $d$ in Tschirnhausen form}
    and to $\rh$ as the \emph{radius of the splitting}.
\end{definition}

\section{Optimal Sobolev regularity of the roots} \label{sec:opt}

In this section, we recall the main result of \cite{ParusinskiRainer15} and 
prove a related bound that will be essential for the proof of \Cref{thm:main1}.

\subsection{Boundedness} \label{ssec:bd} 

Let us recall the main result of \cite{ParusinskiRainer15}.

\begin{theorem}[{\cite[Theorem 1]{ParusinskiRainer15}}] \label{thm:optimal}
    Let $(\al,\be) \subseteq \R$ be a bounded open interval.\footnote{In this section and the next two, the main parameter interval is denoted by $(\al,\be)$ 
    so that the notation is close to the one in \cite{ParusinskiRainer15} because we will frequently refer to \cite{ParusinskiRainer15}.}
    Let $P_a$ be a monic polynomial of degree $d$ with coefficient vector $a \in C^{d-1,1}([\al,\be],\C^d)$.
    Let $\la \in C^0((\al,\be))$ be a continuous root of $P_a$ on $(\al,\be)$.
    Then $\la \in W^{1,q}((\al,\be))$ for every $1 \le q < d/(d-1)$ 
    and 
    \begin{equation} \label{eq:optimal}
        \|\la'\|_{L^q((\al,\be))} \le C(d,q)\, \max\{1,(\be -\al)^{1/q}\} \max_{1 \le j \le d} \|a_j\|_{C^{d-1,1}([\al,\be])}^{1/j}.
    \end{equation}
\end{theorem}

Let $\La : \C^d \to \Iq$ be the solution map defined in \Cref{ssec:solutionmap}.

\begin{corollary}[{\cite[Section 10.3]{ParusinskiRainer15}}] \label{cor:aLabd}
The map 
\begin{equation*}
    \La_* : C^{d-1,1}([\al,\be],\C^d) \to W^{1,q}((\al,\be),\cA_d(\C)), \quad a \mapsto \La \o a, 
\end{equation*}
is well-defined and bounded, for every $1 \le q < d/(d-1)$, where $W^{1,q}((\al,\be),\cA_d(\C))$ carries the metric structure given in \eqref{eq:Almgren}.
\end{corollary}

\begin{remark} \label{rem:aLabd}
    \Cref{cor:aLabd} remains true if the interval $(\al,\be)$ is replaced by a bounded open box $U = I_1 \times \cdots \times I_m \subseteq \R^m$ 
    (and $[\al,\be]$ by $\ol U$). This follows from the proof of \cite[Theorem 6]{ParusinskiRainer15}.
\end{remark}

\subsection{A different bound}
We shall need a different bound for $\|\la'\|_{L^q((\al,\be))}$; see \Cref{rem:optimalmodreason} for the reason. We formulate and use it for 
polynomials in Tschirnhausen form. 
Some details of the proof of \Cref{thm:optimal} in \cite{ParusinskiRainer15} must be recalled, before the bound can be given.

Let $(\al,\be) \subseteq \R$ be a bounded open interval.
Let $P_{\tilde a}$ be a monic polynomial of degree $d$ in Tschirnhausen form with 
coefficient vector $\tilde a \in C^{d-1,1}([\al,\be],\C^d)$, where $\tilde a$ is not identically zero.
Let $\rh>0$ be the radius of the fixed universal splitting of polynomials of degree $d$ in Tschirnhausen form (see \Cref{def:universal}).
Fix a positive constant $B$ satisfying
\begin{equation} \label{eq:B}
    B < \min\Big\{\frac{1}{3},\frac{\rh}{3 d^2 2^d}\Big\}.
\end{equation}
Let $x_0 \in (\al,\be)$ be such that $\tilde a(x_0) \ne 0$ and let $k=k(x_0) \in \{2,\ldots,d\}$ be such that 
\begin{equation} \label{eq:k}
    |\tilde a_k(x_0)|^{1/k} = \max_{2 \le j \le d} |\tilde a_j(x_0)|^{1/j}.
\end{equation}
Let $M=M(x_0)$ be defined by 
\begin{equation} \label{eq:Mbd}
    M:= \max_{2 \le j \le d} \Big( |\tilde a_j^{(d-1)}|_{C^{0,1}([\al,\be])}^{1/d} |\tilde a_k(x_0)|^{(d-j)/(kd)} \Big).
\end{equation}
Choose a maximal open interval $I = I(x_0) \subseteq (\al,\be)$ containing $x_0$ such that 
\begin{equation} \label{eq:MBle}
    M |I| + \sum_{j=2}^d \|(\tilde a_j^{1/j})'\|_{L^1(I)} \le B\, |\tilde a_k(x_0)|^{1/k}.
\end{equation}

\begin{convention} \label{convention}
    Abusing notation, $\tilde a_j^{1/j}$ denotes one fixed continuous selection of the multi-valued function  $\tilde a_j^{1/j}$;
    the value of $\|(\tilde a_j^{1/j})'\|_{L^1(I)}$ is independent of the choice of the selection (by \cite[Lemma 1]{ParusinskiRainer15}).
\end{convention}

Let us consider the following two cases:
\begin{description}
    \item[Case (i)]
        For each $x_0 \in (\al,\be)$ with $\tilde a(x_0) \ne 0$,
        we have equality in \eqref{eq:MBle}, 
        \begin{equation} \label{eq:MBeq}
            M |I| + \sum_{j=2}^d \|(\tilde a_j^{1/j})'\|_{L^1(I)} = B\, |\tilde a_k(x_0)|^{1/k}.
        \end{equation}
    \item[Case (ii)] 
        There exists $x_0 \in (\al,\be)$ with $\tilde a(x_0) \ne 0$ such that 
        the inequality in \eqref{eq:MBle} is strict,
        \begin{equation} \label{eq:MBst}
            M |I| + \sum_{j=2}^d \|(\tilde a_j^{1/j})'\|_{L^1(I)} < B\, |\tilde a_k(x_0)|^{1/k}.
        \end{equation}
\end{description}

The condition \eqref{eq:MBle} guarantees that         
we have a splitting $P_{\tilde a} = P_bP_{b^*}$ on the interval $I$; 
see \cite[Section 8, Step 1]{ParusinskiRainer15}.
In particular, as $I$ is assumed to be maximal, in Case (ii), we have a splitting on the whole interval $I=(\al,\be)$.

\begin{theorem} \label{thm:optimalmod} 
    Let $(\al,\be) \subseteq \R$ be a bounded open interval.
    Let $P_{\tilde a}$ be a monic polynomial of degree $d$ in Tschirnhausen form with 
    coefficient vector $\tilde a \in C^{d-1,1}([\al,\be],\C^d)$, where $\tilde a$ is not identically zero.
    Let $\la \in C^0((\al,\be))$ be a continuous root of $P_{\tilde a}$ on $(\al,\be)$.
    Then, for every $1 \le q<  d/(d-1)$, $\la \in W^{1,q}((\al,\be))$ and $\|\la'\|_{L^q((\al,\be))}$ admits the following bound. 

    In Case \thetag{i}, 
    \begin{equation} \label{eq:casei}
        \|\la'\|_{L^q((\al,\be))} \le C(d,q)\, \Big((\be-\al)^{1/q}\max_{2 \le j \le d} \|\tilde a_j\|_{C^{d-1,1}([\al,\be])}^{1/j}+ \sum_{j=2}^d \|(\tilde a_j^{1/j})'\|_{L^q((\al,\be))}\Big).  
    \end{equation}

    In Case \thetag{ii}, for each $x \in (\al,\be)$, 
    \begin{equation} \label{eq:caseii}
        \|\la'\|_{L^q((\al,\be))} \le C(d,q)\, (\be-\al)^{-1+1/q} \,  |\tilde a_k(x)|^{1/k}.
    \end{equation}
\end{theorem}

\begin{proof}
    We refer to the proof of \Cref{thm:optimal} in \cite{ParusinskiRainer15}. 
    Assume that $B$, $x_0$, $k$, $M$, and $I$ are chosen as above.
    By Proposition 3 and Lemma 15 in \cite{ParusinskiRainer15}, 
    $P_{\tilde a} = P_b P_{b^*}$ splits on $I$ and every continuous root $\mu \in C^0(I)$ 
    of $P_{\tilde b}$ (which results from $P_b$ by the Tschirnhausen tranformation) on $I$ 
    is absolutely continuous and satisfies, for every $1 \le q < d/(d-1)$,
    \begin{equation} \label{eq:resprop3}
        \|\mu'\|_{L^q(I)} \le C(d,q)\, \Big( \| |I|^{-1} |\tilde a_k(x_0)|^{1/k}\|_{L^q(I)} + \sum_{i=2}^{d_b} \| (\tilde b_i^{1/i})'\|_{L^q(I)} \Big).
    \end{equation}
    By Lemmas 8 and 9 in \cite{ParusinskiRainer15}, 
    we may bound the right-hand side of \eqref{eq:resprop3} by 
    \begin{equation} \label{eq:reslem9}
        \| |I|^{-1} |\tilde a_k(x_0)|^{1/k}\|_{L^q(I)} + \sum_{i=2}^{d_b} \| (\tilde b_i^{1/i})'\|_{L^q(I)} 
        \le C(d,q)\, |I|^{-1+1/q}\, |\tilde a_k(x_0)|^{1/k}. 
    \end{equation}
    We may assume that, on $I$,  $\la$ is a root of $P_{b}$, so that 
    \begin{equation} \label{eq:lamu}
        \la(x) = - \frac{b_1(x)}{d_b} + \mu(x), \quad x \in I,
    \end{equation}
    where $d_b = \deg P_b$.
    By Remark 3 in \cite{ParusinskiRainer15}, 
    \begin{equation} \label{eq:b1}
        \|b_1'\|_{L^\infty(I)} \le C(d)\, |I|^{-1} \, |\tilde a_k(x_0)|^{1/k}.
    \end{equation}

    \subsubsection*{Case \thetag{i}}
    For each $x_0 \in (\al,\be)$ with $\tilde a(x_0) \ne 0$ we have \eqref{eq:MBeq}.
    We may combine \eqref{eq:resprop3} and \eqref{eq:reslem9} with \eqref{eq:MBeq} to get 
    \begin{align*}
        \|\mu'\|_{L^q(I)} 
        &\le C(d,q)\, |I|^{-1+1/q}\, B^{-1}\, \Big( M |I| + \sum_{j=2}^d \|(\tilde a_j^{1/j})'\|_{L^1(I)}\Big).
        \\
        &= C(d,q)\, |I|^{1/q}\, B^{-1}\, \Big( M  + \sum_{j=2}^d |I|^{-1} \|(\tilde a_j^{1/j})'\|_{L^1(I)}\Big).
        \\
        &\le C(d,q)\, |I|^{1/q}\, B^{-1}\, \Big( M  + \sum_{j=2}^d |I|^{-1/q} \|(\tilde a_j^{1/j})'\|_{L^q(I)}\Big).
        \\
        &= C(d,q)\,  B^{-1}\, \Big( M |I|^{1/q}  + \sum_{j=2}^d  \|(\tilde a_j^{1/j})'\|_{L^q(I)}\Big);
    \end{align*}
    the second inequality follows from H\"older's inequality.
    By \eqref{eq:k},
    \begin{equation*}
        |\tilde a_k(x_0)|^{1/k} = \max_{2 \le j \le d} |\tilde a_j(x_0)|^{1/j} \le \max_{2 \le j \le d} \|\tilde a_j\|_{C^{d-1,1}([\al,\be])}^{1/j} =: A
    \end{equation*}
    and hence, by \eqref{eq:Mbd},
    \begin{align*}
        M &= \max_{2 \le j \le d} \Big( |\tilde a_j^{(d-1)}|_{C^{0,1}([\al,\be])}^{1/d} |\tilde a_k(x_0)|^{(d-j)/(kd)} \Big)
        \\
          &\le\max_{2 \le j \le d} \Big( A^{j/d} A^{(d-j)/d} \Big)=A. 
    \end{align*}
    Since $B$ was universal (see \eqref{eq:B}), we obtain
    \[
        \|\mu'\|_{L^q(I)} \le C(d,q)\,   \Big(  |I|^{1/q} \max_{2 \le j \le d}  \|\tilde a_j\|_{C^{d-1,1}([\al,\be])}^{1/j} + \sum_{j=2}^d  \|(\tilde a_j^{1/j})'\|_{L^q(I)}\Big).    
    \]
    By \eqref{eq:b1},
    \[
        \|b_1'\|_{L^q(I)} \le C(d)\, |I|^{-1+1/q} \, |\tilde a_k(x_0)|^{1/k}
    \]
    which is estimated the same way.
    So, in view of \eqref{eq:lamu}, we conclude 
    \begin{equation} \label{eq:LqI}
        \|\la'\|_{L^q(I)} \le C(d,q)\,   \Big(  |I|^{1/q} \max_{2 \le j \le d}  \|\tilde a_j\|_{C^{d-1,1}([\al,\be])}^{1/j} + \sum_{j=2}^d  \|(\tilde a_j^{1/j})'\|_{L^q(I)}\Big).    
    \end{equation}
    To summarize, 
    for each $x_0 \in (\al,\be)$ with $\tilde a(x_0) \ne 0$ we have \eqref{eq:MBeq} and \eqref{eq:LqI}, 
    where the interval $I$ contains $x_0$ and is contained in $\Om_{\tilde a} := (\al,\be) \cap \{\tilde a \ne 0\}$.
    By Proposition 2 in \cite{ParusinskiRainer15} (applied to $\tilde a$ instead of $\tilde b$), 
    there is a cover of $\Om_{\tilde a}$ 
    by a countable family $\cI$ of open intervals $I$ on which \eqref{eq:LqI} holds 
    and such that 
    every point of $\Om_{\tilde a}$ belongs to precisely one or two intervals in $\cI$.
    Thus, it follows from \eqref{eq:LqI} that 
    \begin{align*}
        \|\la'\|_{L^q(\Om_{\tilde a})} 
        \le C(d,q)\,   \Big(  (\be-\al)^{1/q} \max_{2 \le j \le d}  \|\tilde a_j\|_{C^{d-1,1}([\al,\be])}^{1/j} 
        + \sum_{j=2}^d  \|(\tilde a_j^{1/j})'\|_{L^q((\al,\be))}\Big).    
    \end{align*}
    Now it suffices to apply Lemma 1 in \cite{ParusinskiRainer15} to have the same bound for $\|\la'\|_{L^q((\al,\be))}$, 
    i.e., \eqref{eq:casei}.

    \subsubsection*{Case \thetag{ii}} In this case, 
    there exists $x_0 \in (\al,\be)$ such that \eqref{eq:MBst} holds. Then 
    $I=(\al,\be)$ so that \eqref{eq:resprop3} and \eqref{eq:reslem9} give
    \begin{equation*}
        \|\mu'\|_{L^q(I)} \le C(d,q)\, (\be-\al)^{-1+1/q} \,  |\tilde a_k(x_0)|^{1/k}.
    \end{equation*}
    Furthermore, by Lemma 5 in \cite{ParusinskiRainer15},
    \[
        \frac{2}{3} \le \Big| \frac{\tilde a_k(x)}{\tilde a_k(x_0)} \Big|^{1/k} \le \frac{4}{3}, \quad x \in (\al,\be).
    \]
    Together with \eqref{eq:lamu} and \eqref{eq:b1}, we conclude 
    \eqref{eq:caseii}.
\end{proof}

\section{Proof of \Cref{thm:main1}} \label{sec:proofs1}

Let $d\ge 2$ be an integer.
Let $(\al,\be) \subseteq \R$ be a bounded open interval.
Let $a_n \to a$ in $C^d([\al,\be],\C^d)$ as $n\to \infty$. 
Let $\La,\La_n : (\al,\be) \to \cA_d(\C)$ be the curves of unordered roots of $P_a,P_{a_n}$, respectively.
We have to show that 
\begin{equation*}
    \mathbf d_{(\al,\be)}^{1,q}(\La,\La_n)   \to 0 \quad \text{ as } n \to \infty,
\end{equation*}
for all $1 \le q < d/(d-1)$.

By \Cref{cor:aLabd}, $\La,\La_n \in W^{1,q}((\al,\be),\Iq)$, for all $1 \le q< d/(d-1)$.
Let $\la,\la_n : (\al,\be) \to \C^d$ be continuous parameterizations of the roots of $P_a$, $P_{a_n}$, i.e., 
\[
    \La = [\la] \quad \text{ and } \quad \La_n = [\la_n].
\]
Then $\la,\la_n \in W^{1,q}((\al,\be), \C^d)$, for all $1 \le q< d/(d-1)$ (by \Cref{thm:optimal}).
For $1 \le i \le d$, $n\ge 1$, and
almost every $x \in (\al,\be)$,
\[
    D\la_i(x) = \la_i'(x) \quad \text{ and } \quad  D\la_{n,i}(x) = \la_{n,i}'(x);
\]
see \Cref{d:diff}.

\subsection{Uniform convergence of $\dd(\La,\La_n)$}

By \Cref{cor:homeo}, 
\begin{equation} \label{eq:C0}
    \|\mathbf{s}_{0}(\La,\La_n)\|_{L^\infty((\al,\be))}  
    = \| \dd( \La , \La_n ) \|_{L^\infty((\al,\be))} \to 0 \quad \text{ as } n \to \infty. 
\end{equation}
Thus, it suffices to show that, for all $1 \le q < d/(d-1)$,
\[
    \ddd^{1,q}_{(\al,\be)}(\La,\La_n)  \to 0 \quad \text{ as } n \to \infty,
\]
where, for any measurable set $E \subseteq (\al,\be)$, we define
\begin{align*} 
    \ddd^{1,q}_{E}(\La,\La_n) &:= \|\bs_1(\La,\La_n)\|_{L^q(E)}. 
\end{align*}

\subsection{Invariance under the Tschirnhausen transformation}

\begin{lemma} \label{lem:redTschirn}
    Let $I \subseteq \R$ be a bounded open interval.
    Let $P_a$, $P_{a_n}$, for $n\ge 1$,  be monic polynomials with $a , a_n \in C^{d}(\ol I,\C^d)$.
    Let $P_{\tilde a}$, $P_{\tilde a_n}$ result from $P_a$, $P_{a_n}$ by the Tschirnhausen transformation.
    Let $\La, \La_n, \tilde \La, \tilde \La_n : I \to \Iq$ be the curves of unordered roots of $P_a,P_{a_n}, P_{\tilde a}, P_{\tilde a_n}$,
    respectively. Then, as $n \to \infty$,
    \begin{enumerate}
        \item $a_n \to a$ in $C^d(\ol I,\C^d)$ if and only if $\tilde a_n \to \tilde a$ in $C^d(\ol I,\C^d)$;
        \item if the equivalent conditions of \thetag{1} hold, then 
            $\ddd^{1,q}_I(\La,\La_n) \to 0$ if and only if $\ddd^{1,q}_I(\tilde \La, \tilde \La_n) \to 0$, for all $1 \le q < d/(d-1)$.
    \end{enumerate}
\end{lemma}

\begin{proof}
    (1) This follows easily from \eqref{eq:Tschirn} and \Cref{prop:lefttr}.

    (2) Fix $1 \le q < d/(d-1)$.
    By \Cref{cor:homeo}, $\|\dd(\La,\La_n)\|_{L^\infty(I)} \to 0$ 
    as well as $\|\dd(\tilde \La,\tilde \La_n)\|_{L^\infty(I)} \to 0$ as $n \to \infty$.
    Assume that  $\ddd^{1,q}_I(\La,\La_n) \to 0$ as $n \to \infty$.
    By \Cref{thm:Almgren and convergence},
    we have 
    \[
        \|\De \o \La - \De \o \La_n\|_{W^{1,q}(I,\R^N)} \to 0 \quad \text{ as } n \to \infty,
    \]
for any Almgren embedding $\De : \Iq \to \R^N$ (see \eqref{eq:AlmgrenDE}).
    Let 
    $H : \Iq \to \R^d$ be an Almgren map with associated real linear form $\et$ (see \Cref{def:Amap}).
    The Tschirnhausen transformation 
    shifts $H \o \La$ and $H \o \La_n$ by $\frac{1}{d} (\et(a_1),\ldots,\et(a_1))$ and $\frac{1}{d} (\et(a_{n,1}),\ldots,\et(a_{n,1}))$, respectively.
    Thus $\|(H \o \tilde \La)' - (H\o \tilde \La_n)'\|_{L^q(I,\R^d)} \to 0$ and, consequently, 
    \[
        \|(\De \o \tilde \La)' - (\De \o \tilde \La_n)'\|_{L^{q}(I,\R^N)} \to 0 \quad \text{ as } n \to \infty,
    \]
    which implies $\ddd^{1,q}_I(\tilde \La,\tilde \La_n) \to 0$, again by \Cref{thm:Almgren and convergence}. 
    The opposite direction follows from the same arguments.
\end{proof}

Thanks to \Cref{lem:redTschirn}, we may assume that all polynomials are in Tschirnhausen form.

\subsection{Accumulation points of $Z_{\tilde a}$}

Let $P_{\tilde a}$ be a monic complex polynomial of degree $d$ in Tschirnhausen form 
with coefficient vector $\tilde a \in C^d([\al,\be],\C^d)$. 
We denote by $Z_{\tilde a}$ the zero set of $\tilde a$ in $[\al,\be]$,
\[
    Z_{\tilde a} := \{x \in [\al,\be] : \tilde a(x) = 0 \}.
\]
Let $\on{acc}(Z_{\tilde a})$ be the set of accumulation points of $Z_{\tilde a}$.

\begin{proposition}   \label{prop:acc} 
    Let $d \ge 2$ be an integer.
    Let $(\al,\be) \subseteq \R$ be a bounded open interval. 
    Let $(P_{\tilde a_n})_{n\ge 1}$ be a sequence of monic complex polynomials of degree $d$ 
    in Tschirnhausen form such that $\tilde a_n \to \tilde a$ in $C^d([\al,\be],\C^d)$ as $n \to \infty$.
    For each $n\ge 1$, let $\la_n$ be a continuous root of $P_{\tilde a_n}$ on $[\al,\be]$.
    For every $\ep >0$ there exist a neighborhood $U$ of $\on{acc}(Z_{\tilde a})$ in $[\al,\be]$ and 
    $n_0\ge 1$ such that 
    \[
        \|\la_n'\|_{L^q(U)} \le C(d,q)\, |U|^{1/q}\, \ep, \quad n \ge n_0,
    \]
    for all $1 \le q < d/(d-1)$.
\end{proposition}

\begin{proof}
    Fix $1 \le q <d/(d-1)$. Let $x_0 \in \on{acc}(Z_{\tilde a})$. 
    By \Cref{lem:Zg1},
    \[
        \tilde a_j^{(s)}(x_0) =0,\quad 2 \le j \le d,~ 0 \le s \le d.
    \]
    Fix $\ep>0$.
    By continuity, there exists $\de >0$ such that 
    \[
        \|\tilde a_j\|_{C^d(\ol I(x_0,\de))} \le \frac{\ep^j}{2}, \quad 2 \le j \le d,
    \]
    where $I(x_0,\de)$ is the open $\de$-neighborhood of $x_0$ in $[\al,\be]$ and $\ol I(x_0,\de)$ 
    is the closure of $I(x_0,\de)$. 
    As $\tilde a_n \to \tilde a$ in $C^d([\al,\be],\C^d)$,
    there exists $n_0 \ge 1$ such that, for all $n \ge n_0$,
    \begin{equation*}
        \|\tilde a_j - \tilde a_{n,j}\|_{C^d(\ol I(x_0,\de))} \le \frac{\ep^j}{2}, \quad 2 \le j \le d,
    \end{equation*}
    and 
    \begin{equation*}
        |\tilde a_{n,j}^{(s)}(x_0)| \le \de^{j-s} \ep^j, \quad 2 \le j \le d, ~ 0 \le s \le d.
    \end{equation*}
    Then, by (the proof of) \Cref{lem:Zg2}, 
    \[
        \|(\tilde a_{n,j}^{1/j})'\|_{L^q(I(x_0,\de))} \le C(d,q)\, |I(x_0,\de)|^{1/q}\, \ep, \quad n \ge n_0.
    \]
    Consequently, by \Cref{thm:optimalmod},
    \[
        \|\la_n'\|_{L^q(I(x_0,\de))} \le C(d,q)\, |I(x_0,\de)|^{1/q}\, \ep, \quad n \ge n_0.
    \]

    Since $\on{acc}(Z_{\tilde a})$ is compact, we may proceed as in the proof of \Cref{cor:Zg2} 
    and the assertion follows.
\end{proof}

\begin{remark}\label{rem:optimalmodreason}
    For the gluing of the bounds on a cover by intervals (see  the proof of \Cref{cor:Zg2}), 
    we need a bound for $\|\la_n'\|_{L^q(I(x_0,\de))}^q$ that is proportional to the length of the 
    interval $|I(x_0,\de)|$. For this purpose, we proved \Cref{thm:optimalmod}.
\end{remark}

\subsection{Some background from \cite{ParusinskiRainer15}}

We recall and slightly adapt several lemmas from \cite{ParusinskiRainer15}.

\begin{lemma}[{\cite[Lemma 5]{ParusinskiRainer15}}] \label{lem:lem5}
    Let $I \subseteq \R$ be a bounded open interval. 
    Let $P_{\tilde a}$ be a monic complex polynomial of degree $d$ in Tschirnhausen form 
    with coefficient vector $\tilde a \in C^{d-1,1}(\ol I,\C^d)$, where $\tilde a$ is not identically zero.  
    Let $x_0 \in I$ be such that $\tilde a(x_0) \ne 0$ and $k \in \{2,\ldots,d\}$ such that 
    \begin{equation} \label{eq:kdom}
        |\tilde a_k(x_0)|^{1/k} \ge |\tilde a_j(x_0)|^{1/j}, \quad 2 \le j \le d.
    \end{equation}
    Assume that, for some positive constant $B<1/3$, 
    \begin{equation} \label{eq:L1B}
        \sum_{j=2}^d \|(\tilde a_j^{1/j})'\|_{L^1(I)} \le B\, |\tilde a_k(x_0)|^{1/k}.
    \end{equation}
    Then, for all $x \in I$ and $2 \le j \le d$,
    \begin{align}
   &|\tilde a_{j}^{1/j}(x) -\tilde a_{j}^{1/j}(x_0) |  \le B\, |\tilde a_k(x_0)|^{1/k}, \label{eq:lem51}
   \\
   &\frac{2}{3}< 1-B \le \Big|\frac{\tilde a_{k}(x)}{\tilde a_{k}(x_0)}\Big|^{1/k} \le 1+B < \frac{4}{3}, \label{eq:lem52}
   \\
   &|\tilde a_{j}(x)|^{1/j} \le \frac{4}{3}\, |\tilde a_{k}(x_0)|^{1/k} \le 2\, |\tilde a_{k}(x)|^{1/k}. \label{eq:lem53}
    \end{align}
\end{lemma}

\begin{remark} \label{rem:lem5}
    If we replace \eqref{eq:kdom} by 
    \begin{equation} \label{eq:kdom'}
        |\tilde a_k(x_0)|^{1/k} \ge \frac{2}{3}\, |\tilde a_j(x_0)|^{1/j}, \quad 2 \le j \le d,
    \end{equation}
    then the conclusions \eqref{eq:lem51} and \eqref{eq:lem52} remain valid (cf.\ the proof of Lemma 5 in \cite{ParusinskiRainer15}) 
    and instead of \eqref{eq:lem53} we have 
    \begin{equation}
        |\tilde a_{j}(x)|^{1/j} \le 2\, |\tilde a_{k}(x_0)|^{1/k} \le 3\, |\tilde a_{k}(x)|^{1/k}. \label{eq:lem53'}
    \end{equation}
    Indeed, by \eqref{eq:lem51} and \eqref{eq:kdom'},
    \[
        |\tilde a_{j}(x)|^{1/j} \le |\tilde a_{j}(x_0)|^{1/j} + B\, |\tilde a_k(x_0)|^{1/k} \le \Big(\frac{3}2 +B\Big)\,|\tilde a_k(x_0)|^{1/k}
    \]
    which yields the first inequality in \eqref{eq:lem53'}; the second one follows from \eqref{eq:lem52}.
\end{remark}

\begin{lemma}[{\cite[Lemma 6]{ParusinskiRainer15}}] \label{lem:lem6}
    In the setting of \Cref{lem:lem5}, we may consider the $C^d$ curve $\ul a = (\ul a_1,\ldots,\ul a_d) : I \to \C^d$, 
    where $\ul a_1:=0$ and $\ul a_j := \tilde a_k^{-j/k} \tilde a_j$, for $2 \le j \le d$.
    Then the length of $\ul a$ is bounded by $3d^2 2^d B$.

    If we replace \eqref{eq:kdom} by \eqref{eq:kdom'}, then the length of $\ul a$ is bounded by $2d^2 3^dB$.
\end{lemma}

\begin{proof}
    Let us assume that \eqref{eq:kdom'} holds. 
    We estimate $|\ul a_j'| \le |(\tilde a_k^{-j/k})' \tilde a_j| + |\tilde a_k^{-j/k} \tilde a_j'|$, using 
    \eqref{eq:lem52} and \eqref{eq:lem53'}:
    \begin{align*}
        |(\tilde a_k^{-j/k})' \tilde a_j| &\le j\,3^j\, |(\tilde a_k^{1/k})'| |\tilde a_k|^{(-j-1)/k} |\tilde a_k|^{j/k} 
        \le  \frac{j\,3^{j+1}}{2}\, |(\tilde a_k^{1/k})'| |\tilde a_k(x_0)|^{-1/k},
        \\
        |\tilde a_k^{-j/k} \tilde a_j'|&\le 3^{j-1} |\tilde a_j|^{-(j-1)/j} |\tilde a_j'| |\tilde a_k|^{-1/k} 
        \le \frac{j\,3^j}{2}\, |(\tilde a_j^{1/j})'| |\tilde a_k(x_0)|^{-1/k},
        \intertext{whence}
        |\ul a_j'| &\le 2d\, 3^d\,  |\tilde a_k(x_0)|^{-1/k} \big( |(\tilde a_k^{1/k})'| + |(\tilde a_j^{1/j})'|\big).
    \end{align*}
    Thus, by \eqref{eq:L1B},
    \begin{align*}
        \int_I \|\ul a'(x)\|_2 \, dx &\le  \int_I \sum_{j=2}^d |\ul a_j'(x)|  \, dx 
        \\
                                     &\le 2d\, 3^d\, |\tilde a_k(x_0)|^{-1/k} \Big((d-1) \|(\tilde a_k^{1/k})'\|_{L^1(I)} + \sum_{j=2}^d \|(\tilde a_j^{1/j})'\|_{L^1(I)}\Big)
                                     \\
                                     &\le 2d^2 3^dB 
    \end{align*}
    as required.
\end{proof}

\begin{lemma}[{\cite[Lemma 7]{ParusinskiRainer15}}] \label{lem:lem7}
    In the setting of \Cref{lem:lem5}, replace \eqref{eq:L1B} by the stronger condition
    \begin{equation}
        M\, |I|+ \sum_{j=2}^d \|(\tilde a_j^{1/j})'\|_{L^1(I)} \le B\, |\tilde a_k(x_0)|^{1/k},
    \end{equation}
    where 
    \begin{equation} \label{eq:M}
        M:= \max_{2 \le j \le d} \Big( |\tilde a_{j}^{(d-1)}|_{C^{0,1}(\ol I)}^{1/d} |\tilde a_{k}(x_0)|^{(d-j)/(kd)} \Big).
    \end{equation}
    Then, for all $2 \le j \le d$ and $1 \le s \le d-1$,
    \begin{align} \label{eq:lem7}
        \begin{split}
            \|\tilde a_j^{(s)}\|_{L^\infty(I)} &\le C(d)\, |I|^{-s} |\tilde a_{k}(x_0)|^{j/k},
            \\
            |\tilde a_j^{(d-1)}|_{C^{0,1}(\ol I)} &\le C(d)\, |I|^{-d} |\tilde a_{k}(x_0)|^{j/k}.
        \end{split}
    \end{align}
    This remains true (with a different constant $C(d)$) if additionally \eqref{eq:kdom} is 
    replaced by \eqref{eq:kdom'}.
\end{lemma}

\begin{lemma}[{\cite[Lemma 8]{ParusinskiRainer15}}] \label{lem:lem8}
    In the setting of \Cref{lem:lem5}, 
    assume that \eqref{eq:lem7} holds. 
    Additionally, suppose that we have a splitting on $I$,
    \begin{equation*}
        P_{\tilde a}= P_b P_{b^*},
    \end{equation*}
    with $d_b := \deg P_b <d$ and coefficients given by \eqref{eq:bj}.
    Then (after Tschirnhausen transformation $b_i \leadsto \tilde b_i$), for all $2 \le i \le d_b$ and $1 \le s \le d-1$,
    \begin{align} \label{eq:lem8}
        \begin{split}
            \|\tilde b_i^{(s)}\|_{L^\infty(I)} &\le C(d)\, |I|^{-s} |\tilde a_{k}(x_0)|^{i/k},
            \\
            |\tilde b_i^{(d-1)}|_{C^{0,1}(\ol I)} &\le C(d)\, |I|^{-d} |\tilde a_{k}(x_0)|^{i/k}.
        \end{split}
    \end{align}
    This remains true (with a different constant $C(d)$) if additionally \eqref{eq:kdom} is 
    replaced by \eqref{eq:kdom'}.
\end{lemma}

\begin{lemma}[{\cite[Lemma 10]{ParusinskiRainer15}}] \label{lem:lem10}
    In the setting of \Cref{lem:lem8}, suppose that  
    $x_1 \in I$ is such that $\tilde b(x_1) \ne 0$ and $\ell \in \{2,\ldots,d_b\}$ is such that
    \begin{equation*} 
        |\tilde b_\ell(x_1)|^{1/\ell} \ge \frac{2}{3}\,|\tilde b_i(x_1)|^{1/i}, \quad 2 \le i \le d_b.
    \end{equation*}
    Assume that, for  some positive constant $D < 1/3$ and some 
    open interval $J \subseteq I$ with $x_1 \in J$,  
    \begin{equation}\label{eq:JD}
        |J||I|^{-1} |\tilde a_k(x_0)|^{1/k} + \sum_{i=2}^{d_b} \|(\tilde b_i^{1/i})'\|_{L^1(J)} \le D\, |\tilde b_\ell(x_1)|^{1/\ell}.
    \end{equation}
    Then the functions $\tilde b_i$ on $J$ satisfy the following.
    \begin{enumerate}
        \item For all $x \in J$ and $2 \le j \le d_b$,
            \begin{align}
               &|\tilde b_{i}^{1/i}(x) -\tilde b_{i}^{1/i}(x_1) |  \le D\, |\tilde b_\ell(x_1)|^{1/\ell}, \label{eq:blem51}
               \\
               &\frac{2}{3}< 1-D \le \Big|\frac{\tilde b_{\ell}(x)}{\tilde b_{\ell}(x_1)}\Big|^{1/\ell} \le 1+D < \frac{4}{3}, \label{eq:blem52}
               \\
               &|\tilde b_{i}(x)|^{1/i} \le 2\, |\tilde b_{\ell}(x_1)|^{1/\ell} \le 3\, |\tilde b_{\ell}(x)|^{1/\ell}. \label{eq:blem53}
            \end{align}
        \item The length of the curve $\ul b : J \to \C^{d_b}$, 
            where $\ul b_1:=0$ and $\ul b_i := \tilde b_\ell^{-i/\ell} \tilde b_i$, for $2 \le i \le d_b$,
            is bounded by $2 d_b^2 3^{d_b} D$.
        \item For all $2 \le i \le d_b$ and $1 \le s \le d-1$,
            \begin{align} \label{eq:lem10}
                \begin{split}
                    \|\tilde b_i^{(s)}\|_{L^\infty(J)} &\le C(d)\, |J|^{-s} |\tilde b_{\ell}(x_1)|^{i/\ell},
                    \\
                    |\tilde b_i^{(d-1)}|_{C^{0,1}(\ol J)} &\le C(d)\, |J|^{-d} |\tilde b_{\ell}(x_1)|^{i/\ell}.
                \end{split}
            \end{align}
    \end{enumerate}
\end{lemma}

\begin{proof}
    We give a short argument, because in \cite[Lemma 10]{ParusinskiRainer15} 
    equality in \eqref{eq:JD} was assumed, but this was for other reasons.

    (1) follows from \Cref{lem:lem5} and \Cref{rem:lem5}.

    (2) is a consequence of \Cref{lem:lem6}. 

    (3) By \eqref{eq:lem8}, for $x \in I$ and $2 \le i \le d_b$,
    \begin{equation} \label{eq:lem10b}
        |\tilde b_i^{(i)}(x)| \le C(d)\, |I|^{-i} |\tilde a_k(x_0)|^{i/k}.
    \end{equation}
    The interpolation lemma \cite[Lemma 4]{ParusinskiRainer15} yields, for $x \in J$ and $1 \le s \le i$,
    \[
        |\tilde b_i^{(s)}(x)| \le C(i)\, |J|^{-s} \big( V_J(\tilde b_i) + V_J(\tilde b_i)^{(i-s)/i} \|\tilde b_i^{(i)}\|_{L^\infty(J)}^{s/i} |J|^s \big),
    \]
    where $V_J(\tilde b_i) := \sup_{x,y \in J} |\tilde b_i(x) - \tilde b_i(y)|$. 
    Thus, by \eqref{eq:blem53}, \eqref{eq:lem10b}, and \eqref{eq:JD},
    \begin{align*}
        |\tilde b_i^{(s)}(x)| &\le C_1(d)\, |J|^{-s} \big( |\tilde b_{\ell}(x_1)|^{i/\ell} 
        + |\tilde b_{\ell}(x_1)|^{(i-s)/\ell}   |J|^s|I|^{-s} |\tilde a_k(x_0)|^{s/k} \big)
        \\
                              &\le C_2(d)\, |J|^{-s}|\tilde b_{\ell}(x_1)|^{i/\ell}.
    \end{align*}
        So \eqref{eq:lem10} holds for $1 \le s \le i$.
        For $s>i$, $(|J||I|^{-1})^s  \le (|J||I|^{-1})^i$ and thus
        \[
            |I|^{-s} |\tilde a_k(x_0)|^{i/k} 
            \le  |J|^{-s} \big(|J||I|^{-1} |\tilde a_k(x_0)|^{1/k} \big)^i \le |J|^{-s}|\tilde b_\ell(x_1)|^{i/\ell},
        \]
        by \eqref{eq:JD}. Then \eqref{eq:lem10} for $s > i$ follows from \eqref{eq:lem8}.
\end{proof}

\begin{remark}
    Below we will assume that the coefficients of the polynomials are of class $C^d$ (instead of $C^{d-1,1}$).
    Then we can replace the bounds for $|\tilde a_j^{(d-1)}|_{C^{0,1}(\ol I)}$ (e.g.\ in \Cref{lem:lem7}) by the same bounds for $\|\tilde a_j^{(d)}\|_{L^\infty(I)}$.
    We will do this without further mention.
\end{remark}

\subsection{Towards a simultaneous splitting}

Let $d \ge 2$ be an integer.
Let $(\al,\be) \subseteq \R$ be a bounded open interval. 
Let $(P_{\tilde a_n})_{n\ge 1}$ be a sequence of monic complex polynomials of degree $d$ 
in Tschirnhausen form such that $\tilde a_n \to \tilde a$ in $C^d([\al,\be],\C^d)$ as $n \to \infty$.

Assume that $\tilde a \not\equiv 0$.
Let $x_0 \in (\al,\be)$ be such that $\tilde a(x_0) \ne 0$ 
and $k  \in \{2,\ldots,d\}$ such that  
\begin{equation} \label{eq:kdom2}
    |\tilde a_k(x_0)|^{1/k} \ge |\tilde a_j(x_0)|^{1/j}, \quad 2 \le j \le d.
\end{equation}
In particular, $|\tilde a_k(x_0)|^{1/k} > 0$.
As $\tilde a_n \to \tilde a$ in $C^d([\al,\be],\C^d)$, 
there is $n_0\ge 1$ such that, for all $n \ge n_0$ and $2 \le j \le d$,
\begin{equation*}
    \Big| |\tilde a_j(x_0)|^{1/j}-|\tilde a_{n,j}(x_0)|^{1/j}\Big| \le \frac{1}{5}\,|\tilde a_k(x_0)|^{1/k}.
\end{equation*}
In particular, for $n \ge n_0$,
\begin{equation} \label{eq:nkk}
    |\tilde a_{n,k}(x_0)|^{1/k} \ge \frac{4}{5}\, |\tilde a_{k}(x_0)|^{1/k}
\end{equation}
and, for $2 \le j \le d$, 
\begin{equation*}
    |\tilde a_{n,j}(x_0)|^{1/j} \le |\tilde a_j(x_0)|^{1/j} + \frac{1}{5}\,|\tilde a_k(x_0)|^{1/k} \le \frac{6}{5}\,|\tilde a_k(x_0)|^{1/k}
\end{equation*}
so that
\begin{equation} \label{eq:kdom3}
    |\tilde a_{n,k}(x_0)|^{1/k} \ge \frac{2}{3}\, |\tilde a_{n,j}(x_0)|^{1/j}. 
\end{equation}

For $n\ge 1$, let $M_n$ be the quantity defined in \eqref{eq:M} for $\tilde a_n$ and $(\al,\be)$, i.e.,
\begin{equation*} 
    M_n:= \max_{2 \le j \le d} \Big( |\tilde a_{n,j}^{(d-1)}|_{C^{0,1}([\al,\be])}^{1/d} |\tilde a_{n,k}(x_0)|^{(d-j)/(kd)} \Big),
\end{equation*}
and let $M_0$ be the same quantity for $\tilde a$.
Define 
\begin{equation} \label{eq:supM}
    M := \sup_{n\ge 0} M_n.
\end{equation}

Let $\rh>0$ be the radius of the fixed universal splitting of polynomials of degree $d$ in Tschirnhausen form (see \Cref{def:universal}) 
and let $B$ be a fixed constant satisfying 
\begin{equation} \label{eq:BB}
    B < \min\Big\{\frac{1}{3},\frac{\rh}{2^3 d^2 3^d}\Big\}.
\end{equation}

Choose an open interval $I\subseteq (\al,\be)$ containing $x_0$ (independent of $n$) such that 
\begin{equation} \label{eq:MB}
    M |I| + \sum_{j=2}^d \|(\tilde a_{j}^{1/j})'\|_{L^1(I)} \le \frac{B}{3}\, |\tilde a_{k}(x_0)|^{1/k}.
\end{equation}
By \Cref{cor:radicals}, there is $n_1 \ge n_0$ such that, for all $n \ge n_1$,
\[
    \Big|\sum_{j=2}^d  \|(\tilde a_{j}^{1/j})'\|_{L^1(I)} -  \sum_{j=2}^d\|(\tilde a_{n,j}^{1/j})'\|_{L^1(I)} \Big| \le \frac{B}{3}  \, |\tilde a_{k}(x_0)|^{1/k}. 
\]
Consequently, for all $n \ge n_1$,
\begin{align} \label{eq:MBn}
    M |I| + \sum_{j=2}^d \|(\tilde a_{n,j}^{1/j})'\|_{L^1(I)} 
     &\le \frac{2B}{3}  \, |\tilde a_{k}(x_0)|^{1/k} \le B\, |\tilde a_{n,k}(x_0)|^{1/k},
\end{align}
using \eqref{eq:nkk} and \eqref{eq:MB}.

Observe that the assumptions of \Cref{lem:lem5}, respectively \Cref{rem:lem5}, are satisfied for $\tilde a$ and $\tilde a_n$, 
where $n\ge n_1$, on $I$. 
Indeed, \eqref{eq:kdom2} and \eqref{eq:kdom3} amount to \eqref{eq:kdom} and \eqref{eq:kdom'}, respectively, 
and \eqref{eq:MB} and \eqref{eq:MBn} imply \eqref{eq:L1B}.

Furthermore, \Cref{lem:lem7} yields that, for all $2 \le j \le d$, $1 \le s \le d$, and $n \ge n_1$,
\begin{equation*}
    \|\tilde a_{n,j}^{(s)}\|_{L^\infty(I)} \le C(d)\, |I|^{-s} |\tilde a_{n,k}(x_0)|^{j/k},
\end{equation*}
as well as
\begin{equation*}
    \|\tilde a_{j}^{(s)}\|_{L^\infty(I)} \le C(d)\, |I|^{-s} |\tilde a_{k}(x_0)|^{j/k}.
\end{equation*}

In particular, by \eqref{eq:lem52}, the multivalued functions $\tilde a_k^{1/k}$ and $\tilde a_{n,k}^{1/k}$, where $n \ge n_1$, 
are bounded away from zero on $I$. 
So the continuous selections of $\tilde a_k^{1/k}$ and $\tilde a_{n,k}^{1/k}$ on $I$, respectively, 
just differ by a multiplicative factor $\th^r$ for some $1\le r \le k$,
where $\th$ is a $k$-th root of unity.
Thus, since $\tilde a_n \to a$ in $C^d([\al,\be],\C^d)$, we may assume that the continuous selections of $\tilde a_k^{1/k}$ and $\tilde a_{n,k}^{1/k}$ 
are chosen such that they belong to $C^d(\ol I)$ and 
satisfy
\begin{equation} \label{eq:choice}
    \|\tilde a_{k}^{1/k} -\tilde a_{n,k}^{1/k} \|_{C^d(\ol I)} 
    \to 0
    \quad \text{ as } n \to \infty,
\end{equation}
where abusing notation we denote the continuous selection by the same symbol as the multivalued function.
See \Cref{prop:lefttr} and \Cref{convention}.

By \Cref{lem:lem6}, the length of the $C^d$ curves $\ul a : I \to \C^d$ and $\ul a_n : I \to \C^d$, 
where $n \ge n_1$, is bounded by
\[
    2 d^2 3^d B \le \frac{\rh}{4},
\]
thanks to \eqref{eq:BB}.
By \eqref{eq:choice},
there is $n_2 \ge n_1$ such that, for $n \ge n_2$, 
\begin{equation*} 
    \|\ul a(x_0) - \ul a_n(x_0)\|_2 < \frac{\rh}{4}.
\end{equation*}
Then, for $n \ge n_2$, 
the ball $B(\ul a_n(x_0),\rh/4)$ is contained in the ball $B(\ul a(x_0),\rh/2)$
which in turn is contained in some ball of the finite cover $\cB$ of $K$ (see \Cref{def:universal}). 
It follows that we have  splittings on $I$,
\begin{equation} \label{eq:simsplit}
    P_{\tilde a} = P_b P_{b^*} \quad \text{ and } \quad P_{\tilde a_n} = P_{b_n} P_{b_n^*}, \quad n \ge n_2,
\end{equation}
with the following properties:
\begin{enumerate}
    \item $d_b := \deg P_b = \deg P_{b_n}$, for all $n \ge n_2$, and $d_b < d$.
    \item There exist bounded analytic functions $\ps_1, \ldots,\ps_{d_b}$ with bounded partial 
        derivatives of all orders such that, for all $x \in I$ and $1 \le i \le d_b$,
        \begin{align*}
            b_i(x) &= \tilde a_k(x)^{i/k} \ps_i (\ul a(x)),  
            \\
            b_{n,i}(x) &= \tilde a_{n,k}(x)^{i/k} \ps_i (\ul a_n(x)),  \quad n \ge n_2.
        \end{align*}
\end{enumerate}
The same is true for the second factors $P_{b^*}$ and $P_{b_n^*}$.

\begin{definition}
    We say that the family $\{P_{\tilde a}\} \cup \{P_{\tilde a_n}\}_{n\ge n_2}$ has a 
    \emph{simultaneous splitting on $I$} if \eqref{eq:simsplit} and the properties (1) and (2)
    are satisfied.
\end{definition}

We remark that, applying the Tschirnhausen transformation to $P_b$ and $P_{b_n}$, we find 
bounded analytic functions $\tilde \ps_1, \ldots,\tilde \ps_{d_b}$ with bounded partial 
derivatives of all orders such that, for all $x \in I$ and $2 \le i \le d_b$,
\begin{align*}
    \tilde b_i(x) &= \tilde a_k(x)^{i/k} \tilde \ps_i (\ul a(x)),  
    \\
    \tilde b_{n,i}(x) &= \tilde a_{n,k}(x)^{i/k} \tilde \ps_i (\ul a_n(x)),  \quad n \ge n_2.
\end{align*}
That follows from \eqref{eq:Tschirn}.

\begin{lemma} \label{lem:bconv}
    We have $b_n \to b$ and $\tilde b_n \to \tilde b$ in $C^d(\ol I,\C^{d_b})$ as $n \to \infty$.
\end{lemma}

\begin{proof}
    By \eqref{eq:choice},  
    $\tilde a_k^{1/k}, \tilde a_{n,k}^{1/k} \in C^d(\ol I)$ and $\ul a, \ul a_n \in C^d(\ol I,\C^d)$, 
    for $n \ge n_1$,
    and the assertion follows from \Cref{prop:lefttr}.
\end{proof}

We have proved the following proposition.

\begin{proposition} \label{prop:toindass}
    Let $d \ge 2$ be an integer.
    Let $(\al,\be) \subseteq \R$ be a bounded open interval. 
    Let $(P_{\tilde a_n})_{n \ge 1}$ be a sequence of monic complex polynomials of degree $d$ 
    in Tschirnhausen form
    such that $\tilde a_n \to \tilde a$ in $C^d([\al,\be],\C^d)$ as $n \to \infty$.

    Assume that $\tilde a \not\equiv 0$.
    Let $x_0 \in (\al,\be)$ be such that $\tilde a(x_0) \ne 0$ 
    and $k  \in \{2,\ldots,d\}$ such that \eqref{eq:kdom2} holds. 
    Choose an open interval $I\subseteq (\al,\be)$ containing $x_0$  such that \eqref{eq:MB} 
    holds, where $M$ and $B$ are given by \eqref{eq:supM} and \eqref{eq:BB}, respectively.

    Then there exists $n_0 \ge 1$ such that the following holds:
    \begin{enumerate}
        \item For all $2 \le j \le d$, $1 \le s \le d$, and $n \ge n_0$,
            \begin{equation*}
                \|\tilde a_{n,j}^{(s)}\|_{L^\infty(I)} \le C(d)\, |I|^{-s} |\tilde a_{n,k}(x_0)|^{j/k},
            \end{equation*}
            as well as
            \begin{equation*}
                \|\tilde a_{j}^{(s)}\|_{L^\infty(I)} \le C(d)\, |I|^{-s} |\tilde a_{k}(x_0)|^{j/k}.
            \end{equation*}
        \item The family $\{P_{\tilde a}\} \cup \{P_{\tilde a_n}\}_{n\ge n_0}$ 
            has a simultaneous splitting on $I$, 
            \begin{equation*} 
                P_{\tilde a} = P_b P_{b^*} \quad \text{ and } \quad P_{\tilde a_n} = P_{b_n} P_{b_n^*}, \quad n \ge n_0.
            \end{equation*}
        \item We have $b_n \to b$ and $\tilde b_n \to \tilde b$ in $C^d(\ol I,\C^{d_b})$ as $n \to \infty$ 
            and analogously $b_{n}^* \to b^*$ and $\tilde b_n^* \to \tilde b^*$ in $C^d(\ol I,\C^{d-d_b})$.
    \end{enumerate}
\end{proposition}

\subsection{Minimizing permutations respect the splitting}

Our next goal is
\Cref{prop:S'} which will be needed (at the end of the proof of \Cref{prop:induction} and in the proof of \Cref{lem:subsequence}) 
to relate the quantity $\ddd^{1,q}_{E}$ for the roots of the two separate factors of a simultaneous splitting
with the quantity $\ddd^{1,q}_{E}$ for the roots of the product of the two factors.

\begin{definition}
    Let $P_a = P_b P_c$, where $P_b$ and $P_c$ are monic complex polynomials with distinct sets of roots $\{\la_1,\ldots,\la_{d_b}\}$ and  $\{\la_{d_b+1},\ldots,\la_{d}\}$ 
    (with multiplicities),
    where
    $d = \deg P_a$,
    $1\le d_b = \deg P_b<d$, and $\deg P_c = d-d_b$.
    A permutation $\ta \in \on{S}_d$ is said to \emph{respect the splitting $P_a = P_b P_c$} if
    $\ta(\{1,\ldots,d_b\}) = \{1,\ldots,d_b\}$ and 
    $\ta(\{d_b+1,\ldots,d\}) = \{d_b+1,\ldots,d\}$.
\end{definition}

\begin{proposition} \label{prop:S'}
    Let $d \ge 2$ be an integer.
    Let $(\al,\be) \subseteq \R$ be a bounded open interval. 
    Let $(P_{\tilde a_n})_{n \ge 1}$ be a sequence of monic complex polynomials of degree $d$ 
    in Tschirnhausen form
    such that $\tilde a_n \to \tilde a$ in $C^d([\al,\be],\C^d)$ as $n \to \infty$. 
    Let $\la,\la_n : (\al,\be) \to \C^d$ be  continuous parameterizations of the roots of $P_{\tilde a}$, $P_{\tilde a_n}$, respectively.

    Assume that $\tilde a \not\equiv 0$.
    Let $x_0 \in (\al,\be)$ be such that $\tilde a(x_0) \ne 0$ 
    and $k  \in \{2,\ldots,d\}$ such that \eqref{eq:kdom2} holds
    and let $I \subseteq (\al,\be)$ be an interval containing $x_0$ such that \eqref{eq:L1B} holds. 
    Suppose that the family $\{P_{\tilde a}\} \cup \{P_{\tilde a_n}\}_{n\ge n_0}$ 
    has a simultaneous splitting on $I$, 
    \begin{equation*} 
        P_{\tilde a} = P_b P_{b^*} \quad \text{ and } \quad P_{\tilde a_n} = P_{b_n} P_{b_n^*}, \quad n \ge n_0.
    \end{equation*}
    Then, after possibly shrinking $I$ and increasing $n_0$,
    for all $x \in I$ and $n \ge n_0$,
    \begin{align} \label{eq:sym}
        \MoveEqLeft \{\ta \in \on{S}_d: \de(\la(x),\ta \la_n(x)) = \dd([\la(x)],[\la_n(x)])\}
        \\ \notag
                &=\{\ta \in \on{S}_d': \de(\la(x),\ta \la_n(x)) = \dd([\la(x)],[\la_n(x)])\},
    \end{align}
    where $\on{S}_d'$ denotes the subset of permutations in $\on{S}_d$ that respect the splitting $P_{\tilde a_n(x)} = P_{b_n(x)}P_{b_n^*(x)}$.
\end{proposition}

\begin{proof}
    Let $H:= \sup_{n \ge 0} H_n$, where
    \[
        H_0 := 4d\, \max_{1\le j \le d} \|\tilde a_j\|_{C^{0,1}([\al,\be])}^{1/j} \quad \text{ and } \quad 
        H_n := 4d\, \max_{1\le j \le d} \|\tilde a_{n,j}\|_{C^{0,1}([\al,\be])}^{1/j}, \quad n \ge 1.
    \]
    Let $\ch>0$ be as in \Cref{rem:universal}.
    By shrinking $I$, we may assume that 
    \[
        |I|^{1/d} \le \frac{\ch}{15\sqrt d H}\cdot |\tilde a_k(x_0)|^{1/k}.
    \]
    Moreover (by \eqref{eq:C0}), there is $n_1\ge n_0$ such that, for all $n \ge n_1$,
    \[
        \dd([\la(x_0)],[\la_n(x_0)]) \le   \frac{\ch}{15\sqrt d}\cdot |\tilde a_k(x_0)|^{1/k}. 
    \]
    Then, for $x \in I$ and $n \ge n_1$,
    \begin{align} \notag \label{eq:cont1}
        \dd([\la(x)],[\la_n(x)]) &\le \dd([\la(x)],[\la(x_0)]) + \dd([\la(x_0)],[\la_n(x_0)])+\dd([\la_n(x_0)],[\la_n(x)])
        \\ \notag
                                 &\le H \, |x-x_0|^{1/d} + \frac{\ch}{15\sqrt d}\cdot |\tilde a_k(x_0)|^{1/k}+H \, |x-x_0|^{1/d}
                                 \\
                                 &\le \frac{\ch}{5 \sqrt d} \cdot |\tilde a_k(x_0)|^{1/k},
    \end{align}
    by \Cref{cor:Hoelder}.

    After possibly reordering them, we may assume that $\la_1,\ldots,\la_{d_b}$ are the roots of $P_b$ and $\la_{d_b+1},\ldots,\la_d$ 
    are the roots of $P_{b^*}$ 
    and, analogously, that $\la_{n,1},\ldots,\la_{n,d_b}$ are the roots of $P_{b_n}$ and $\la_{n,d_b+1},\ldots,\la_{n,d}$ 
    are the roots of $P_{b_n^*}$.

    By \eqref{eq:Cauchybound}, for all $x \in I$,
    \[
        \max_{1 \le i \le d} |\la_{i}(x)| \le A := 2 \, \max_{1\le j \le d} \|\tilde a_j\|_{L^\infty(I)}^{1/j}.
    \]

    By \Cref{rem:universal}, for all $x \in I$, $n \ge n_1$, $1 \le i \le d_b$, and $d_b+1 \le j \le d$,
    \begin{equation*}
        |\tilde a_{n,k}^{1/k}(x)\la_i(x) - \tilde a_k^{1/k}(x)\la_{n,j}(x)| 
        > \ch\cdot |\tilde a_k(x)|^{1/k}|\tilde a_{n,k}(x)|^{1/k}.
    \end{equation*}
    By \eqref{eq:lem52} and \eqref{eq:nkk},     
    there is $n_2\ge n_1$ such that, for $n \ge n_2$ and all $x \in I$, 
    \[
        |\tilde a_k(x)|^{1/k}|\tilde a_{n,k}(x)|^{1/k} \ge \frac{16}{45}\, |\tilde a_k(x_0)|^{2/k},
    \]
    and, in view of \eqref{eq:choice}, 
    \[
        |\tilde a_k^{1/k}(x) - \tilde a_{n,k}^{1/k}(x)| \le \frac{\ch}{45A}\, |\tilde a_k(x_0)|^{2/k}.
    \]
    We conclude that, for all $x \in I$, $n \ge n_2$, $1 \le i \le d_b$, and $d_b+1 \le j \le d$,
    \begin{align*}
        \frac{16 \ch}{45}\, |\tilde a_k(x_0)|^{2/k} &<  |\tilde a_{n,k}^{1/k}(x)\la_i(x) - \tilde a_k^{1/k}(x)\la_{n,j}(x)|
        \\
                                                    &\le|\tilde a_{n,k}^{1/k}(x) - \tilde a_k^{1/k}(x)||\la_i(x)|
                                                    +|\tilde a_k(x)|^{1/k} |\la_i(x) - \la_{n,j}(x)|
                                                    \\
                                                    &\le\frac{\ch}{45A}\, |\tilde a_k(x_0)|^{2/k} \cdot A
                                                    + \frac{4}{3}\, |\tilde a_k(x_0)|^{1/k} |\la_i(x) - \la_{n,j}(x)|,
    \end{align*}   
    using \eqref{eq:lem52}. 
    Thus, for all $x \in I$, $n \ge n_2$, $1 \le i \le d_b$, and $d_b+1 \le j \le d$,
    \begin{equation} \label{eq:cont2}
        |\la_i(x) - \la_{n,j}(x)| > \frac{\ch}{4}\, |\tilde a_k(x_0)|^{1/k}.
    \end{equation}

    Now we show that \eqref{eq:sym} holds, for all $x \in I$ and $n\ge n_2$.
    If not,
    there exist $x \in I$, $n\ge n_2$, and 
    a permutation $\ta \in \on{S}_d$ that does not respect the splitting $P_{\tilde a_n(x)} = P_{b_n(x)}P_{b_n^*(x)}$
    such that
    \[
        \de(\la(x),\ta \la_n(x)) = \dd([\la(x)],[\la_n(x)]).
    \]
    So there exist $i \in \{1,\ldots,d_b\}$ and  $j \in \{d_b+1,\ldots,d\}$ with $\ta(i) = j$.
    Thus, by \eqref{eq:cont2},
    \begin{align*}
        \de(\la(x),\ta \la_n(x)) &= \frac{1}{\sqrt d} \|\la(x) -\ta\la_n(x)\|_2 
        \\
                                 &\ge \frac{1}{\sqrt d} |\la_i(x) -\la_{n,j}(x)|
                                 \\
                                 &> \frac{\ch}{4\sqrt d}\, |\tilde a_k(x_0)|^{1/k}
    \end{align*}    
    which contradicts \eqref{eq:cont1}.
\end{proof}

\subsection{The induction argument}

\begin{proposition} \label{prop:induction}
    Let $d\ge 2$ be an integer.
    Let $I \subseteq \R$ be a bounded open interval. 
    Let $(P_{\tilde a_n})_{n\ge 1}$ be a sequence of monic complex polynomials of degree $d$
    in Tschirnhausen form 
    such that $\tilde a_n \to \tilde a$ in $C^d(\ol I,\C^d)$ as $n \to \infty$.
    Assume that $\tilde a \not\equiv 0$.
    Let $x_0 \in I$ and $k \in \{2,\ldots,d\}$ be such that the following conditions are 
    satisfied:
    \begin{enumerate}
        \item $\tilde a(x_0) \ne 0$.
        \item $|\tilde a_{k}(x_0)|^{1/k} \ge |\tilde a_{j}(x_0)|^{1/j}$ for all $2 \le j \le d$.
        \item $\sum_{j=2}^d \|(\tilde a_{j}^{1/j})'\|_{L^1(I)} \le B\, |\tilde a_{k}(x_0)|^{1/k}$ 
            for some constant $B < 1/3^2$.
        \item There exists $n_0 \ge 1$, such that for all $2 \le j \le d$, $1 \le s \le d$, and $n \ge n_0$,
            \begin{align*}
                \|\tilde a_{n,j}^{(s)}\|_{L^\infty(I)} \le C(d) \, |I|^{-s} \,|\tilde a_{n,k}(x_0)|^{j/k}. 
            \end{align*}
        \item
            The family $\{P_{\tilde a}\} \cup \{P_{\tilde a_n}\}_{n\ge n_0}$ 
            has a simultaneous splitting on $I$, 
            \begin{equation*} 
                P_{\tilde a} = P_b P_{b^*} \quad \text{ and } \quad P_{\tilde a_n} = P_{b_n} P_{b_n^*}, \quad n \ge n_0.
            \end{equation*}
    \end{enumerate}
    Let $\mu,\mu_n : I \to \C^{d_b}$ be continuous parameterizations of the roots of $P_{\tilde b}$, $P_{\tilde b_n}$ 
    (which result from $P_{b}$, $P_{b_n}$ by means of the Tschirnhausen transformation),
    respectively.

    Then there exist a set $E_0 \subseteq I$ of measure zero and a countable cover $\cE$ of 
    $I \setminus E_0$ 
    by measurable sets with the property that,  
    for each $E \in \cE$, 
    \begin{equation*}
        \ddd^{1,q}_{E}([\mu],[\mu_n]) \to 0 \quad \text{ as } n \to \infty,
    \end{equation*}
    for all $1 \le q < d/(d-1)$.
\end{proposition}

Clearly, $\ddd^{1,q}_{E}([\mu],[\mu_n])$ is here understood in dimension $d_b$. 
By symmetry, the conclusion of \Cref{prop:induction} also holds for continuous parameterizations of the roots of the second factors $P_{\tilde b^*}$ and $P_{\tilde b_n^*}$ in (5).

\begin{proof}
    We proceed by induction on $d$.

\subsubsection*{Base case} 
If $d =2$ then $d_b=1$ and $P_{\tilde b}(Z) = P_{\tilde b_n}(Z) = Z$. 
Hence $\mu = \mu_n =0$ and thus $\ddd^{1,q}_{I}([\mu],[\mu_n])=0$.
So the assertion is trivially true.

 \subsubsection*{Induction step}
 Let $d>2$ and assume that the statement holds if the degree of the polynomials is smaller than $d$.

 By \Cref{lem:bconv},    
 \begin{align} \label{eq:bconv}
     \|\tilde b - \tilde b_n\|_{C^d(\ol I,\C^{d_b})} &\to 0 \quad \text{ as } n \to \infty. 
 \end{align}
 Here \Cref{lem:bconv} is valid, because
 only the existence of a simultaneous splitting and the assumptions \eqref{eq:kdom2} and
 \[
     \sum_{j=2}^d \|(\tilde a_{j}^{1/j})'\|_{L^1(I)} \le B\, |\tilde a_{k}(x_0)|^{1/k},
 \]
 for some constant $B < 1/3^2$,
 are needed to conclude \eqref{eq:choice} which is used in the proof of \Cref{lem:bconv}.

 By \Cref{prop:acc}, 
 for every $\ep>0$
 there exist a neighborhood $U$ of $\on{acc}(Z_{\tilde b})$ in $\ol I$ and $n_1\ge n_0$ 
 such that 
 \[
     \|\mu_n'\|_{L^q(U, \C^{d_b})} \le C(d,q) \, |U|^{1/q}\,\ep, \quad n\ge n_1,
 \]
 for all $1 \le q < d/(d-1)$ (since $d_b < d$). 
 It follows that
 \begin{equation} \label{eq:indacc}
     \ddd^{1,q}_{\on{acc}(Z_{\tilde b})} ([\mu],[\mu_n])  \to 0 \quad \text{ as } n \to \infty.
 \end{equation}
 Indeed, $\mu|_{\on{acc}(Z_{\tilde b})} = 0$ and hence, for all $x \in \on{acc}(Z_{\tilde b})$ where $\mu'(x)$ and $\mu_n'(x)$ exist, 
 we have  $\mu'(x)=0$ and thus
 \begin{align*}
     \bs_1([\mu],[\mu_n])(x) =  \frac{1}{\sqrt d_b} \,\|\mu_n'(x)\|_2
 \end{align*}
 which implies \eqref{eq:indacc}.

 Assume that $x_1 \in I$ is such that $\tilde b(x_1) \ne 0$, in particular, $d_b\ge 2$.
 Let $\ell \in \{2,\ldots,d_b\}$ be such that 
 \begin{equation*}
     |\tilde b_\ell(x_1)|^{1/\ell} \ge  |\tilde b_i(x_1)|^{1/i}, \quad 2 \le i \le d_b.
 \end{equation*}
 As in the derivation of \eqref{eq:nkk} and \eqref{eq:kdom3} from \eqref{eq:kdom2},
 we find that
 there is $n_1\ge n_0$ such that, for all $n \ge n_1$,
 \begin{equation} \label{eq:nll}
     |\tilde b_{n,\ell}(x_1)|^{1/\ell} \ge \frac{4}{5}\, |\tilde b_\ell(x_1)|^{1/\ell}
 \end{equation}
 and, for  $2 \le i \le d_b$, 
 \begin{equation*}
     |\tilde b_{n,\ell}(x_1)|^{1/\ell} \ge \frac{2}{3}\, |\tilde b_{n,i}(x_1)|^{1/i}. 
 \end{equation*}

 Choose an open interval $J \subseteq I$ containing $x_1$ such that
 \begin{equation*}
     |J||I|^{-1} |\tilde a_k(x_0)|^{1/k}  + \sum_{i=2}^{d_e} \|(\tilde b_i^{1/i})'\|_{L^1(J)} \le \frac{D}{3}\,|\tilde b_\ell(x_1)|^{1/\ell},
 \end{equation*}
 where $D$ is a positive constant satisfying
 \begin{equation} \label{eq:D}
     D < \min \Big\{ \frac{1}{3}, \frac{\si}{2^3 d_b^2 3^{d_b}}\Big\}
 \end{equation}
 and $\si$ is the radius of the universal splitting of polynomials of degree $d_b$ in Tschirnhausen form (see \Cref{def:universal}).

 By \eqref{eq:bconv} and \Cref{cor:radicals},
 there is  
 $n_2 \ge n_1$ such that, for all $n \ge n_2$,
 \[
     \Big| \sum_{i=2}^{d_b} \|(\tilde b_{i}^{1/i})'\|_{L^1(J)} - \sum_{i=2}^{d_b} \|(\tilde b_{n,i}^{1/i})'\|_{L^1(J)} \Big|
     \le \frac{D}{6}  \, |\tilde b_{\ell}(x_1)|^{1/\ell} 
 \]
 and (cf.\ \eqref{eq:choice}) 
 \[
     |J||I|^{-1} \big||\tilde a_k(x_0)|^{1/k} - \tilde a_{n,k}(x_0)|^{1/k} \big|\le \frac{D}{6}  \, |\tilde b_{\ell}(x_1)|^{1/\ell}.
 \]
 Consequently, for all $n \ge n_2$,
 \begin{align*}
     |J||I|^{-1} |\tilde a_{n,k}(x_0)|^{1/k}  + \sum_{j=2}^{d_b} \|(\tilde b_{n,i}^{1/i})'\|_{L^1(J)} 
    &\le \frac{2D}{3}\,|\tilde b_\ell(x_1)|^{1/\ell}\le D\,|\tilde b_{n,\ell}(x_1)|^{1/\ell},
 \end{align*}
 using \eqref{eq:nll}. 

 We see that 
 \Cref{lem:lem10} applies to $\tilde b$ and to $\tilde b_n$, for $n\ge n_2$. 
 So,  for all $2 \le i \le d_b$, $1 \le s \le d$, and $n \ge n_2$,
 \begin{align*}
     \|\tilde b_{n,i}^{(s)}\|_{L^\infty(J)} \le C(d) \, |J|^{-s} \,|\tilde b_{n,\ell}(x_1)|^{j/\ell}. 
 \end{align*}
 Moreover, the length of the curves $\ul b : J \to \C^{d_b}$ and $\ul b_n : J \to \C^{d_b}$, for $n \ge n_2$,
 is bounded by
 \[
     2 d_b^2 3^{d_b} D \le \frac{\si}{4},
 \]
 using \eqref{eq:D}. 
 By \eqref{eq:blem52}, we conclude (as in the derivation of \eqref{eq:choice}) 
 that there are continuous selections $\tilde b_\ell^{1/\ell}$ and $\tilde b_{n,\ell}^{1/\ell}$, for $n \ge n_2$, 
 in $C^d(\ol J)$ of the corresponding multivalued functions and 
 \begin{equation*} 
     \|\tilde b_{\ell}^{1/\ell} -\tilde b_{n,\ell}^{1/\ell} \|_{C^d(\ol J)} 
     \to 0
     \quad \text{ as } n \to \infty.
 \end{equation*} 
 It follows (as in the derivation of \eqref{eq:simsplit})
 that there is $n_3 \ge n_2$ such that 
 the family $\{P_{\tilde b}\} \cup \{P_{\tilde b_n}\}_{n\ge n_3}$ 
 has a simultaneous splitting on $J$, 
 \begin{equation*} 
     P_{\tilde b} = P_c P_{c^*} \quad \text{ and } \quad P_{\tilde b_n} = P_{c_n} P_{c_n^*}, \quad n \ge n_3.
 \end{equation*}

 We may assume that
 \begin{equation} \label{eq:rootssplit}
     \mu|_J = (\nu,\nu^*) \quad \text{ and }\quad \mu_n|_J = (\nu_n,\nu_n^*), \quad n \ge n_3,
 \end{equation}
 where $\nu, \nu_n: J \to \C^{d_c}$, $\nu^*, \nu^*_n : J \to \C^{d_{c^*}}$ are continuous parameterizations of the roots of $P_{c}$, $P_{c_n}$,
 $P_{c^*}$, $P_{c_n^*}$,
 respectively. Hereby $1\le d_c := \deg P_c = \deg P_{c_n}< d_b$ and $d_{c^*} := \deg P_{c^*} = \deg P_{c_n^*} = d_b -d_c$.
 Set 
 \begin{align*} 
     \begin{split}
         \tilde \nu &:= \nu + \tfrac{1}{d_c}(c_1,\ldots,c_1),
         \quad
         \tilde \nu^* := \nu^* + \tfrac{1}{d_{c^*}}(c_1^*,\ldots,c_1^*),
         \\
         \tilde \nu_n &:= \nu_n + \tfrac{1}{d_c}(c_{n,1},\ldots,c_{n,1}),
         \quad
         \tilde \nu_n^* := \nu_n^* + \tfrac{1}{d_{c^*}}(c_{n,1}^*,\ldots,c_{n,1}^*).
     \end{split}
 \end{align*}
 Then $\tilde \nu, \tilde \nu_n: J \to \C^{d_c}$, $\tilde \nu^*, \tilde \nu^*_n : J \to \C^{d_{c^*}}$ are continuous parameterizations of the roots of $P_{\tilde c}$, $P_{\tilde c_n}$,
 $P_{\tilde c^*}$, $P_{\tilde c_n^*}$,
 respectively.

 By the induction hypothesis, 
 there is a set $E_{J,0} \subseteq J$ of measure zero and a countable 
 cover $\cE_J$ of $J \setminus E_{J,0}$
 by measurable sets with the property that, 
 for each $E_J \in \cE_J$, 
 \[
     \ddd^{1,q}_{E_J}([\tilde \nu], [\tilde \nu_n])
     \to 0 \quad \text{ and } \quad 
     \ddd^{1,q}_{E_J}([\tilde \nu^*],[\tilde \nu_n^*])
     \to 0 \quad \text{ as } n \to \infty,
 \]
 for all $1 \le q < d/(d-1)$ (since $d_b <d$), 
 where $\ddd^{1,q}_{E_J}([\tilde \nu],[\tilde \nu_n])$  
 is understood in dimension $d_c$ and $\ddd^{1,q}_{E_J}([\tilde \nu^*],[\tilde \nu_n^*])$  in dimension $d_{c^*}$.
 By \Cref{lem:redTschirn},
 this implies 
 \begin{equation} \label{eq:checknu}
     \ddd^{1,q}_{E_J}( [\nu], [\nu_n])\to 0 \quad \text{ and } \quad \ddd^{1,q}_{E_J}( [\nu^*], [\nu_n^*])\to 0 \quad \text{ as } n \to \infty,
 \end{equation}
 for all $1 \le q < d/(d-1)$, 
 since $c_n \to c$ in $C^d(\ol J,\C^{d_c})$ 
 and $c_n^* \to c^*$ in $C^d(\ol J,\C^{d_{c^*}})$ as $n \to \infty$,
 by \Cref{lem:bconv}.

 We conclude that, 
 for each $E_J \in \cE_J$,
 \begin{equation} \label{eq:indrest}
     \ddd^{1,q}_{E_J}([\mu],[\mu_n])  \to 0  \quad \text{ as } n \to \infty,
 \end{equation}
 for all $1 \le q < d/(d-1)$, 
 where $\ddd^{1,q}_{E_J}([\mu],[\mu_n])$ is now understood in dimension $d_b$.
 This follows from \eqref{eq:rootssplit}, \eqref{eq:checknu} and \Cref{prop:S'}.

 Then the induction step is a consequence 
 of \eqref{eq:indacc}, \eqref{eq:indrest}, and the fact that $\{x \in I : \tilde b (x) \ne 0\}$ 
 can be covered by countably many intervals $J$. 
 The proposition is proved.
\end{proof}

\subsection{Proof of \Cref{thm:main1}}

By \Cref{lem:redTschirn}, we may assume that all polynomials are in Tschirnhausen form.
So let $d\ge 2$ be an integer, $(\al,\be) \subseteq \R$ a bounded open interval, and
$\tilde a_n \to \tilde a$ in $C^d([\al,\be],\C^d)$ as $n \to \infty$.
Let $\La,\La_n : (\al,\be) \to \cA_d(\C)$ be the curves of unordered roots of $P_{\tilde a}$, $P_{\tilde a_n}$, respectively. 

In view of \eqref{eq:C0}, it remains to show that, for all $1 \le q < d/(d-1)$, 
\begin{equation} \label{eq:toshow}
    \ddd^{1,q}_{(\al,\be)}(\La,\La_n) \to 0 \quad \text{ as } n \to \infty.
\end{equation}
To this end, we prove the following proposition.

\begin{proposition} \label{lem:subsequence}
    There is a subsequence $(n_k)$ such that, for all $1 \le q < d/(d-1)$, 
    \begin{equation*}
        \ddd^{1,q}_{(\al,\be)}(\La,\La_{n_k}) \to 0 \quad \text{ as } k \to \infty.
    \end{equation*}
\end{proposition}

\Cref{lem:subsequence} implies \eqref{eq:toshow}, 
in view of \Cref{lem:reals}.

In the proof of \Cref{lem:subsequence}, we will use Vitali's convergence theorem, i.e., \Cref{thm:Vitali}.
As a preparation, we first show the following lemma.

\begin{lemma} \label{lem:uniformint}
    The set $\{\mathbf s_1(\La,\La_n)^q : n \ge 1\}$ is uniformly integrable, for each $1 \le q < d/(d-1)$.
\end{lemma}

\begin{proof}
    Let $\la,\la_n : (\al,\be) \to \C^d$ be continuous parameterizations of the roots so that $\La=[\la]$ and $\La_n = [\la_n]$.
    
    Let $p := \frac{d}{d-1}$,  
    fix $1 \le q < p$, and let $r:= \frac{p+q}{2q}$.
    Then $r >1$ and $qr< p$.
    The function $G(t) := t^r$ for $t \ge 0$ is nonnegative, increasing, and $G(t)/t \to \infty$ as $t \to \infty$.
    We have 
    \[
        \sup_{n \ge 1} \int_\al^\be G(|\la_n'(x)|^q) \,dx
        = \sup_{n \ge 1} \int_\al^\be |\la_n'(x)|^{qr} \,dx < \infty,
    \]
    by \eqref{eq:optimal}. Thus the assertion follows from the fact that,  
    for almost every $x \in (\al,\be)$, 
    \[
        \mathbf{s}_1( \La,\La_n )(x)  \le \frac{1}{\sqrt d}\big(  \|\la'(x)\|_2 + \|\la_n'(x)\|_2\big),
    \]
    and from de la Vall\'ee Poussin's criterion, i.e., \Cref{thm:VP}.
\end{proof}

\begin{proof}[Proof of \Cref{lem:subsequence}]
    We claim that there
    is a set $F_0 \subseteq (\al,\be)$ of measure zero and a countable cover $\{F_1,F_2,\ldots\}$ of 
    $(\al,\be) \setminus F_0$ 
    by measurable sets with the property that, 
    for each $F_i$ with $i \ge 1$,
    \begin{equation}\label{eq:onFi}
        \ddd^{1,q}_{F_i}(\La, \La_{n})\to 0 \quad \text{ as } n \to \infty,
    \end{equation}
    for all $1 \le q < d/(d-1)$.
    By \Cref{prop:acc}, we may take $F_1 = \on{acc}(Z_{\tilde a})$. 
    For $x_0 \in (\al,\be)$ with $\tilde a(x_0) \ne 0$, let $k \in \{2,\ldots,d\}$ be such that 
    \eqref{eq:kdom2} holds and let $I\ni x_0$ be an interval such that \eqref{eq:MB} holds with $M$ and $B$ defined in \eqref{eq:supM} and \eqref{eq:BB}, respectively.
    Then, by \Cref{prop:toindass}, the assumptions of \Cref{prop:induction} are satisfied.
    Thus, 
    there is a set $E_0 \subseteq I$ of measure zero and a countable cover $\cE$ of 
    $I \setminus E_0$ 
    by measurable sets with the property that, 
    for each $E \in \cE$,
    \[
        \ddd^{1,q}_{E}(\La,\La_{n})\to 0 \quad \text{ as } n \to \infty,
    \]
    for all $1 \le q < d/(d-1)$ (arguing as at the end of the proof of \Cref{prop:induction} with \Cref{prop:S'}). 
    The claim follows, since the set $Z_{\tilde a} \setminus \on{acc}(Z_{\tilde a})$ has measure zero 
    and $(\al,\be) \setminus Z_{\tilde a}$ can be covered by countably many intervals $I$.

    Next we assert that there is a subsequence $(n_k)$ such that, as $k \to \infty$, 
    \begin{equation} \label{eq:s1aept}
        \bs_1(\La,\La_{n_k})  \to 0
        \quad   \text{ almost everywhere in } (\al,\be).
    \end{equation}
    Indeed, setting $u_n:=\bs_1(\La,\La_{n})$, we infer from \eqref{eq:onFi} 
    that
    there is a subsequence $n^1_1 < n^1_2 < \cdots$ such that 
    \[
        u_{n^1_k} \to 0
        \quad  \text{ almost everywhere in } F_1.
    \]
    Again by \eqref{eq:onFi}, there is a subsequence 
    $n^2_1 < n^2_2 < \cdots$ of $(n^1_k)$ such that 
    \[
        u_{n^2_k} \to 0
        \quad  \text{ almost everywhere in } F_2,
    \]
    and, in general, 
    let $n^{i+1}_1 < n^{i+1}_2 < \cdots$ be a subsequence of $(n^i_k)$ such that 
    \[
        u_{n^{i+1}_k} \to 0
        \quad  \text{ almost everywhere in } F_{i+1}.
    \]
    Then \eqref{eq:s1aept} holds for the subsequence $n_k := n^k_k$.

    In view of \eqref{eq:s1aept} and \Cref{lem:uniformint}, we now use Vitali's convergence theorem, i.e., \Cref{thm:Vitali}, 
    to conclude the assertion of the proposition.\footnote{Note that on a finite measure space almost everywhere convergence 
    implies convergence in measure, by Egorov's theorem.}
\end{proof}

This completes the proof of \Cref{thm:main1}.

\subsection{Proof of \Cref{cor:main1}} \label{ssec:proofcormain1}

Fix $1 \le q < d/(d-1)$.
Fix an ordering of $\on{S}_d$.
For $x \in I$, 
let $\ta(x) \in \on{S}_d$ be as defined in \Cref{def:dd}. 
Then 
\begin{align*}
    \big\| \|\la'\|_2 - \|\la_n'\|_2 \big\|_{L^q(I)} 
         &= \big\| \|\la'\|_2 - \|\ta\la_n'\|_2 \big\|_{L^q(I)}
         \le \big\| \|\la' - \ta\la_n'\|_2 \big\|_{L^q(I)} 
         \\
         &\le \sqrt d \, \big\| \bs_1([\la], [\la_n]) \big\|_{L^q(I)}
         \le \sqrt d\cdot  \dd^{1,q}_{I}([\la],[\la_n]).
\end{align*}
Thus, \Cref{thm:main1} implies that
\[
    \big\| \|\la'\|_2 - \|\la_n'\|_2 \big\|_{L^q(I)}\to 0 \quad \text{ as } n \to \infty.
\]
Since
\begin{align*}
    \big| \|\la'\|_{L^q(I,\C^d)} - \|\la_n'\|_{L^q(I,\C^d)} \big| 
        &=  
        \big| \big\| \|\la'\|_2 \big\|_{L^q(I)} - \big\| \|\la_n'\|_2 \big\|_{L^q(I)} \big| 
        \\
        &\le \big\| \|\la'\|_2 - \|\la_n'\|_2 \big\|_{L^q(I)},
\end{align*}
we also have 
\[
    \|\la_n'\|_{L^q(I,\C^d)} \to \|\la'\|_{L^q(,\C^d)} \quad \text{ as } n \to \infty.
\]
\Cref{cor:main1} is proved.

\section{Proof of \Cref{thm:main1var}} \label{sec:main1var}

\Cref{thm:main1var} follows from an adaptation of the proof of \Cref{thm:main1}; actually, the proof simplifies.

Let $d\ge 2$ be an integer.
Let $(\al,\be) \subseteq \R$ be a bounded open interval.
Let $a_n \to a$ in $C^d([\al,\be],\C^d)$ as $n \to \infty$.
Assume that $\la_n : (\al,\be) \to \C^d$ is a continuous parameterization of the roots of $P_{a_n}$  
and that $\la_n$ converges in $C^0([\al,\be],\C^d)$ to a continuous parameterization $\la$ of the roots of $P_a$.

Our goal is to show that
\begin{equation} \label{eq:goalvar}
    \|\la' - \la_n'\|_{L^q((\al,\be),\C^d)} \to 0 \quad \text{ as } n \to \infty,
\end{equation}
for all $1 \le q <d/(d-1)$. 

Without loss of generality we may assume that all polynomials are in Tschirnhausen form.

Instead of \Cref{prop:induction} we use:

\begin{proposition} \label{prop:inductionvar}
    In the setting of   \Cref{prop:induction},  
    let $\mu_n : I \to \C^{d_b}$ be continuous parameterizations of the roots of $P_{\tilde b_n}$ 
    that converge in $C^0(\ol I,\C^d)$ to a continuous parameterization of the roots of $P_{\tilde b}$.

    Then there exist a set $E_0 \subseteq I$ of measure zero and a countable cover $\cE$ of 
    $I \setminus E_0$ 
    by measurable sets with the property that,  
    for each $E \in \cE$, 
    \begin{equation*}
        \|\mu' - \mu_n'\|_{L^q(E)} \to 0 \quad \text{ as } n \to \infty,
    \end{equation*}
    for all $1 \le q < d/(d-1)$.
\end{proposition}

\begin{proof}
    The proof is analogous to the one of \Cref{prop:induction} but simpler.
    We only indicate the necessary modifications.
    The base case of the induction on $d$ is trivial. 
    In the induction step,
    we may take $E_1 = \on{acc}(Z_{\tilde b})$, by \Cref{prop:acc}.
    By the proof of \Cref{prop:induction},
    there is $n_3 \ge 1$ such that 
    the family $\{P_{\tilde b}\} \cup \{P_{\tilde b_n}\}_{n\ge n_3}$ 
    has a simultaneous splitting on $J$
    and $\mu$, $\mu_n$ satisfy \eqref{eq:rootssplit}.
    We observe that $\nu_n \to \nu$ and $\nu_n^* \to \nu^*$ uniformly on $J$ as $n \to \infty$, by the uniform convergence of $\mu_n$. 
    The induction hypothesis (after Tschirnhausen transformation)
    and the fact that $\{x \in I : \tilde b(x)\ne 0\}$ can be covered by countably many intervals $J$ 
    easily imply the assertion.
\end{proof}

We claim that there
is a set $F_0 \subseteq (\al,\be)$ of measure zero and a countable cover $\{F_1,F_2,\ldots\}$ of 
$(\al,\be) \setminus F_0$ 
by measurable sets with the property that, 
for each $F_i$ with $i \ge 1$, 
\[
    \|\la' - \la_n'\|_{L^q(F_i)}  \to 0 \quad \text{ as } n \to \infty,
\]
for all $1 \le q < d/(d-1)$.
By \Cref{prop:acc}, we may take $F_1 = \on{acc}(Z_{\tilde a})$. 
In the complement of $Z_{\tilde a}$, 
we use \Cref{prop:inductionvar} to conclude the claim.

As in \eqref{eq:s1aept}, we infer that
we have pointwise almost everywhere convergence $\la_{n_k}' \to \la'$ in $(\al,\be)$ of a subsequence $(n_k)$
and thus \eqref{eq:goalvar} is true on the subsequence $(n_k)$,
by the dominated convergence theorem. Consequently, \eqref{eq:goalvar} holds (by \Cref{lem:reals}).
The proof of \Cref{thm:main1var} is complete.

\section{Proofs of the multiparameter versions}
\label{sec:proofs2}

Let $\De : \Iq \to \R^N$ be an Almgren embedding.

\subsection{Proof of \Cref{thm:mainAlmgrenmult}}

Assume that $a_n \to a$ in $C^d(\ol U, \C^d)$ as $n \to \infty$, where $U = I_1 \times \cdots \times I_m$.
Let $\La,\La_n : U \to \cA_d(\C)$ be the maps of unordered roots of $P_a$, $P_{a_n}$, respectively.

First observe that $\dd(\La,\La_n) \to 0$ uniformly on $U$, by \Cref{cor:homeo}.
Consequently, $\De \o \La_n \to \De \o \La$ uniformly on $U$, by the fact that $\De$ is Lipschitz.

For brevity, we write $F:= \De \o \La$ and $F_n:= \De \o \La_n$.
We know that $\cF:=\{F_n : n \ge 1\}\cup \{F\}$ is a bounded set in $W^{1,q}(U,\R^N)$ for all $1 \le q< d/(d-1)$ 
(see \Cref{rem:aLabd}).
Fix $1 \le q < d/(d-1)$.
Let $x = (x_1,x')$. For $x' \in  U'=I_2 \times \cdots \times I_m$, 
consider 
\[
    A_n(x') = \int_{I_1} \big\| \p_1 F(x_1,x') -  \p_1 F_n (x_1,x')\big\|_2^q\, dx_1. 
\]
Then $A_n(x') \to 0$ as $n \to \infty$, by \Cref{thm:mainAlmgren}.
By Tonelli's theorem,
\begin{equation} \label{eq:Tonelli}
    \int_U \big\| \p_1 F(x) -  \p_1 F_n (x)\big\|_2^q\, dx = \int_{U'} A_n(x')\, dx'. 
\end{equation}
Let us check that the family $\{A_n : n\ge 1\}$ is uniformly integrable. 
Set
\begin{equation} \label{e:r}
    r:= \frac{1}{2q}\Big(q + \frac{d}{d-1}\Big).
\end{equation}
By Jensen's inequality,
\begin{align} \label{eq:Jensen}
    \sup_{n\ge 1} \int_{U'} A_n(x')^r \,dx' &= \sup_{n \ge 1} \int_{U'}\Big( \int_{I_1} \big\| \p_1 F(x_1,x') -  \p_1 F_n (x_1,x')\big\|_2^q\, dx_1 \Big)^r\, dx'
    \notag \\
                                            &\le |I_1|^{r-1}  \sup_{n \ge 1} \int_{U'} \int_{I_1} \big\| \p_1 F(x_1,x') -  \p_1 F_n (x_1,x')\big\|_2^{qr}\, dx_1 \, dx'
                                            \notag    \\
                                            &= |I_1|^{r-1}  \sup_{n \ge 1} \int_{U} \big\| \p_1 F(x) -  \p_1 F_n (x)\big\|_2^{qr}\, dx
\end{align}
which is finite, by the boundedness of $\cF$, as $qr < d/(d-1)$.
Since $r>1$, we conclude that $\{A_n : n\ge 1\}$ is uniformly integrable,
by de la Vall\'ee Poussin's criterion, i.e., \Cref{thm:VP}. 

By Vitali's convergence theorem, see \Cref{thm:Vitali}, and \eqref{eq:Tonelli},
\[
    \int_U \big\| \p_1 F(x) -  \p_1 F_n (x)\big\|_2^q\, dx \to 0 \quad \text{ as } n \to \infty.
\]
The partial derivatives $\p_j$, for $2 \le j \le m$, are handled in the same way.
Thus \Cref{thm:mainAlmgrenmult} is proved.

\subsection{Independence from the Almgren embedding}

\begin{lemma} \label{lem:Almgrenbornology}
    Let $U \subseteq \R^m$ be open.
    The bornology of $W^{1,q}(U,\Iq)$ (induced by the metric \eqref{eq:Almgren}) is independent of the Almgren embedding $\De$.   
\end{lemma}

\begin{proof}
    Let $\De^i : \Iq \to \R^{N_i}$, for $i =1,2$, be two Almgren embeddings. 
The map
\[
    \De^2 \o (\De^1)|_{\De^1(\Iq)}^{-1} : \De^1(\Iq) \to \R^{N_2}
\]
is Lipschitz and has a Lipschitz extension $\Ga$ to all of $\R^{N_1}$.
It is well-known that superposition with $\Ga$ maps bounded sets in $W^{1,q}(U,\R^{N_1})$ to bounded sets in $W^{1,q}(U,\R^{N_2})$ 
(see e.g.\ \cite{MarcusMizel79}). The lemma follows.
\end{proof}

\begin{theorem} \label{thm:topmult}
    Let $\De^i : \Iq \to \R^{N_i}$, for $i =1,2$, be two Almgren embeddings. 
    Let $U \subseteq \R^m$ be a bounded open box, $U=I_1 \times \cdots \times I_m$. 
    Let $p>1$.
    Let $f,f_n \in W^{1,p}(U,\Iq)$, for $n\ge 1$. 
    Then, as $n \to \infty$, 
    \[
        \| \De^1 \o f - \De^1 \o f_n \|_{W^{1,p}(U,\R^{N_1})} \to 0 \quad  \implies  \quad  \| \De^2\o f - \De^2 \o f_n \|_{W^{1,q}(U,\R^{N_2})} \to 0,
    \]
    for each $1\le q < p$. 
\end{theorem}

\begin{proof}
    Set $F^i := \De^i \o f$ and $F^i_n := \De^i \o f_n$.
   For $x' \in  U'=I_2 \times \cdots \times I_m$, 
   consider 
   \begin{align*}
       A^1_n(x') &= \int_{I_1} \big\| \p_1 F^1(x_1,x') -  \p_1 F^1_n (x_1,x')\big\|_2^p\, dx_1, 
       \\
       B^1_n(x') &= \int_{I_1} \big\|  F^1(x_1,x') -  F^1_n (x_1,x')\big\|_2^p\, dx_1.
   \end{align*}
   Assume that
\[
    \|F^1-F^1_n\|_{W^{1,p}(U,\R^{N_1})} \to 0 \quad \text{ as } n \to \infty.
\]
Then $\{F^1_n : n \ge 1\}$ is a bounded subset of  $W^{1,p}(U,\R^{N_1})$ and thus  
$\{F^2_n : n \ge 1\}$ is a bounded subset of  $W^{1,p}(U,\R^{N_2})$,
by \Cref{lem:Almgrenbornology}. 

By assumption and Tonelli's theorem,
   \[
       \int_{U'} A^1_n(x')\, dx' \to 0 \quad \text{ and }  \quad \int_{U'} B^1_n(x')\, dx' \to 0 \quad \text{ as } n \to \infty.
   \]
   Thus there is a subsequence $(n_k)$ such that $A^1_{n_k}(x') \to 0$ and $B^1_{n_k}(x') \to 0$ for almost every $x' \in U'$ as $k \to \infty$.
   For each such $x'$,
   \Cref{thm:Almgren and convergence} implies that $A^2_{n_k}(x') \to 0$ and $B^2_{n_k}(x') \to 0$ as $k \to \infty$, 
   which are defined in analogy to $A^1_n(x')$ and $B^1_n(x')$ with $p$ replaced by $q$.
   Set $r = \frac{q+p}{2q}$.
   Then, as in \eqref{eq:Jensen}, 
   we see that $\{A^2_n : n \ge 1\}$ and $\{B^2 : n \ge 1\}$ are uniformly integrable, 
   because $\{F^2_n : n \ge 1\}$ is bounded in $W^{1,p}(U,\R^{N_2})$.

   Then \Cref{thm:Vitali} and Tonelli's theorem imply that
\[
    \|  F^2 -  F^2_{n_k} \|_{L^{q}(U,\R^{N_2})} \to 0 \quad \text{ and } \quad  \| \p_1 F^2 - \p_1 F^2_{n_k} \|_{L^{q}(U,\R^{N_2})} \to 0 
\]
as $k \to \infty$. Since the partial derivatives $\p_j$, for $2 \le j \le m$, can be treated in the same way, we 
have showed that there is a subsequence $(n_k)$ such that
\[
    \|  F^2 -  F^2_{n_k} \|_{W^{1,q}(U,\R^{N_2})} \to 0 \quad \text{ as } k \to \infty.
\]
This implies the assertion, by \Cref{lem:reals}.
\end{proof}

\subsection{A multiparameter version of \Cref{thm:main1}}

\begin{definition}
    Let $U=I_1 \times \cdots \times I_m$ be a bounded open box in $\R^m$.
    Let $f, g  : U \to \Iq$ be such that the restriction to each segment in $U$ parallel to the coordinate axes belongs to $W^{1,q}$:
    writing
    \[
        x = \sum_{j=1}^m x_j e_j = x_i e_i +\sum_{j\ne i} x_je_j =: x_i e_i + \ul x_i
    \]
    and $f_{\ul x_i}(x_i) := f(x_i e_i + \ul x_i)$, we have $f_{\ul x_i},g_{\ul x_i} \in W^{1,q}(I_i,\Iq)$ for all $1 \le i \le m$ 
    and all $\ul x_i \in U_i := \prod_{j \ne i} I_j$. 
    Define 
    \[
        \bs_1(f,g)(x)  := \max_{1\le i \le m} \bs_1(f_{\ul x_i},g_{\ul x_i})(x_i).
    \]
\end{definition}

\begin{theorem} \label{thm:multex}
    Let $d\ge 2$ be an integer.
    Let $U \subseteq \R^m$ be a bounded open box, $U=I_1 \times \cdots \times I_m$.
    Let $a_n \to a$ in $C^d(\ol U,\C^d)$.
    Let $\La,\La_n : U \to \cA_d(\C)$ be the maps of unordered roots of $P_a$, $P_{a_n}$, respectively.
    Then
    \begin{equation*}
        \|\bs_1(\La,\La_n)\|_{L^{q}(U)} \to 0 \quad \text{ as } n \to \infty,
    \end{equation*}
    for all $1 \le q < d/(d-1)$.
\end{theorem}

\begin{proof}
    Fix $1 \le q < d/(d-1)$. Without loss of generality let $i=1$ and set $\ul x_1 = x'$ and $U_1 = U'$.
For $x' \in  U'$, 
consider 
\[
    A_n(x') = \int_{I_1} ( \bs_1(\La_{x'},(\La_n)_{x'})(x_1))^q\, dx_1. 
\]
Then $A_n(x') \to 0$ as $n \to \infty$, by \Cref{thm:main1}.
The family $\{A_n : n\ge 1\}$ is uniformly integrable: with $r$ as defined in \eqref{e:r} 
we have (see \eqref{eq:Jensen})
\begin{align*} 
    \sup_{n\ge 1} \int_{U'} A_n(x')^r \,dx' 
                                           &\le |I_1|^{r-1}  \sup_{n \ge 1} \int_{U'} \int_{I_1} (\bs_1(\La_{x'},(\La_n)_{x'})(x_1))^{qr}\, dx_1 \, dx'.
\end{align*}
Let $\la_{x'}$, $(\la_n)_{x'}$ be a continuous parameterization of the roots of $P_{a(\cdot,x')}$, $P_{a_n(\cdot,x')}$, respectively.
Then 
\begin{align*}
   \MoveEqLeft \|\bs_1(\La_{x'},(\La_n)_{x'})\|_{L^{qr}(I_1)} 
   \le \frac{1}{\sqrt d} \Big( \|\la_{x'}'\|_{L^{qr}(I_1,\C^d)} + \|(\la_n)_{x'}'\|_{L^{qr}(I_1,\C^d)} \Big) 
      \\
      &\le C(d,qr,|I_1|) \, \Big(\max_{1 \le j \le d} \|a_j(\cdot,x')\|_{C^{d-1,1}(\ol I_1)}^{1/j} 
      +\max_{1 \le j \le d} \|a_{n,j}(\cdot,x')\|_{C^{d-1,1}(\ol I_1)}^{1/j} \Big),
\end{align*}
by \Cref{thm:optimal}, as $qr < d/(d-1)$. By assumption, the right-hand side is bounded by a constant that is independent of $x'$ and $n$.
Thus, uniform integrability follows from \Cref{thm:VP}.
So Tonelli's theorem and \Cref{thm:Vitali} imply the theorem.
\end{proof}

\section{Interpretation of the results in the Wasserstein space on $\C$}
\label{sec:Wasserstein}

In this section, we interpret our results in the space of probability measures on $\C$.
This point of view allowed Antonini, Cavalletti, and Lerario to study optimal transport between algebraic hypersurfaces in
$\C\mathbb P^n$ in \cite{Antonini:2022aa}. We will also finish the proof of \Cref{thm:mainse}.

\subsection{The roots as a probability measure} \label{ssec:probability}
With a monic polynomial $P_a(Z) = Z^d + \sum_{j=1}^d a_j Z^{d-j}$, 
where $a=(a_1,\ldots, a_d) \in \C^d$,  
we may associate in a natural way a probability measure $\mu(a)$ on $\C$ defined by
\begin{equation*}
    \mu(a) := \frac{1}{d} \sum_{P_a(\la) = 0} m_\la(a)\cdot \a{\la},
\end{equation*}
where $m_\la(a)$ denotes the multiplicity of $\la$ as a root of $P_a$ and $\a{\la}$ is the Dirac mass at $\la$.

Let us endow the set $\sP(\C)$ of probability measures on $\C$ (with its Euclidean structure) 
with the $q$-Wasserstein distance $W_q$, for $q \ge 1$, 
and denote the resulting metric space by $\sP_q(\C)$.

Then we get a map 
\begin{equation*}
    \mu : \C^d \to \sP_q(\C)
\end{equation*}
which, besides $\La : \C^d \to \cA_d(\C)$ from \Cref{ssec:solutionmap}, is another incarnation of the solution map. 
The image of $\mu$ can be indentified with $\cA_d(\C)$ or with the quotient of $\C^d$ by the symmetric group
$\on{S}_d$. The restriction of the $W_2$-metric to $\mu(\C^d) \subseteq \sP_2(\C)$ is given by 
\begin{equation*}
    W_2([z],[w])^2 =  \min_{\si \in \on{S}_d} \frac{1}{d} \sum_{j=1}^d |z_j - w_{\si(j)}|^2,
\end{equation*}
and thus coincides with the metric $\dd$ on $\cA_d(\C)$ from \Cref{ssec:AdC}. (For this reason we chose the factor $1/\sqrt d$ 
in the definition of $\dd$.) We get a similar expression for $q \ne 2$, but all of them are equivalent.

It turns out that $\mu(\C^d)$ is geodesically convex in $\sP_q(\C)$.

\subsection{Some distances on $AC^q(I,X)$} \label{ssec:distAC}

Let $(X,\mathsf d)$ be a complete metric space and $AC^q(I,X)$ the set introduced in \Cref{ssec:ACq}.
The set $AC^q(I,X)$ can be seen as a subset of the metric space $C^0(I,X)$ with the uniform norm. 
Additionally, the metric speed or the $q$-energy can be used in a natural way to measure ``closeness'' 
in $AC^q(I,X)$. Thus, we define, for $\ga_1,\ga_2 \in AC^q(I,X)$, 
\begin{equation*}
    \on{dist}^s_q(\ga_1,\ga_2) := \sup_{x \in I}\mathsf d(\ga_1(x), \ga_2(x)) + \big\| |\dot \ga_1| - |\dot \ga_2| \big\|_{L^q(I)}    
\end{equation*}
and
\begin{equation*}
    \on{dist}^e_q(\ga_1,\ga_2) := \sup_{x \in I}\mathsf d(\ga_1(x), \ga_2(x)) + \big| \cE_q(\ga_1) - \cE_q(\ga_2) \big|.   
\end{equation*}
Both $\on{dist}^s_q$ and $\on{dist}^e_q$ are metrics on $AC^q(I,X)$. 

\subsection{Boundedness and continuity of the map $\mu_*$}

In the following, we identify $\mu(\C^d)$ with $\cA_d(\C)$. 
Recall that the map $[\cdot] : \C^d \to \cA_d(\C)$ which sends an ordered $d$-tuple to the corresponding unordered one is Lipschitz.

\begin{lemma} \label{lem:metricspeed}
    Let $\la : I \to \C^d$ be an absolutely continuous curve and let $\ga : I \to \sP_2(\C)$ be 
    defined by $\ga(x) := [\la(x)]$, for $x \in I$.
    Then the metric speed of $\ga$ is given by
    \[
        |\dot \ga|(x) = \frac{1}{\sqrt d} \|\la'(x)\|_2\quad \text{ for almost every } x \in I.
    \]
\end{lemma}

\begin{proof}
    For $x \in I$ and $|h|$ sufficiently small,
   \begin{align*}
        W_2(\ga(x), \ga(x+h)) = \dd([\la(x)],[\la(x+h)])  \le \frac{1}{\sqrt d} \|\la(x) - \la(x+h)\|_2.
    \end{align*}
    Let $\si_h \in \on{S}_d$ be such that 
    \[
       \dd([\la(x)],[\la(x+h)]) =  \frac{1}{\sqrt d} \|\la(x) - \si_h\la(x+h)\|_2.
    \] 
    We claim that, for sufficiently small $|h|$, $\si_h$ lies in the stabilizer group of $\la(x)$, i.e., $\si_h\la(x)=\la(x)$. 
    Otherwise, there is a sequence $h_n \to 0$ such that $\si_{h_n}\la(x) \ne\la(x)$. 
    Since $\on{S}_d$ is finite, by passing to a subsequence, we may assume that $\si_{h_n} =: \si$ is independent of $n$. 
    But continuity of $\la$
    implies
    \[
        \dd([\la(x)],[\la(x+h_n)]) =  \frac{1}{\sqrt d} \|\la(x) - \si\la(x+h_n)\|_2 \to 0 \quad
        \text{ as } n \to \infty,
    \]
    and hence $\si \la(x) = \la(x)$, a contradition.

    By the claim, for small enough $|h|$, 
    \[
       \dd([\la(x)],[\la(x+h)]) =  \frac{1}{\sqrt d} \|\la(x) - \la(x+h)\|_2.
    \]
    Now the assertion follows easily.
\end{proof}

Now \Cref{thm:optimal} implies the following.

\begin{theorem} \label{thm:bdse} 
Let $I \subseteq \R$ be a bounded open interval.
The map 
\[
    \mu_* : C^{d-1,1}(\ol I,\C^d) \to AC^q(I,\sP_2(\C)), \quad a \mapsto \mu \o a,
\]
is well-defined and bounded, for every $1 \le q < d/(d-1)$, 
where $AC^q(I,\sP_2(\C))$ carries the metric $\on{dist}^s_q$ or $\on{dist}^e_q$ from \Cref{ssec:distAC}.
\end{theorem}

\begin{proof}
    Let $a \in C^{d-1,1}(\ol I,\C^d)$ and  
    let $\la : I \to \C^d$ be a continuous parameterization of the roots of $P_a$. 
    Fix $1 \le q < d/(d-1)$. 
    Then $\la \in AC^{q}(I,\C^d)$, by \Cref{thm:optimal}. 
    Since $\mu(a(x)) = [\la(x)]$, for all $x \in I$, the statement is a consequence of \eqref{eq:Cauchybound}, 
    \eqref{eq:optimal}, and \Cref{lem:metricspeed}.
\end{proof}

\Cref{thm:main1} and \Cref{cor:main1} lead to the following continuity result. 

\begin{theorem}\label{thm:contse}
Let $I \subseteq \R$ be a bounded open interval.
The map 
\[
    \mu_* : C^{d}(\ol I,\C^d) \to AC^q(I,\sP_2(\C)), \quad a \mapsto \mu \o a,
\]
is continuous, for every $1 \le q < d/(d-1)$, 
where $AC^q(I,\sP_2(\C))$ carries the metric $\on{dist}^s_q$ or $\on{dist}^e_q$.
\end{theorem}

\begin{proof}
    Let $a_n \to a$ in $C^d(\ol I,\C^d)$ as $n\to \infty$. 
    We must show that $\mu_*(a_n) \to \mu_*(a)$ with respect to $\on{dist}^s_q$ and $\on{dist}^e_q$.
    There exist continuous parameterizations $\la,\la_n : I \to \C^d$ of the roots of $P_a$, $P_{a_n}$, respectively.
    Then $\mu(a(x)) = [\la(x)]$ and $\mu(a_n(x)) = [\la_n(x)]$, for all $x \in I$ and $n$. 
    So \Cref{thm:main1} (or \Cref{cor:homeo}) 
    shows that 
    \[
        \sup_{x \in I} \dd(\mu(a(x)),\mu(a_n(x))) \to 0 \quad \text{ as } n \to \infty.
    \]
    The rest follows from \Cref{cor:main1} and \Cref{lem:metricspeed}.
\end{proof}

Clearly, \Cref{thm:contse} implies \Cref{thm:mainse}.

\appendix

\section{}
\label{sec:appendix}

\subsection{Vitali's convergence theorem}

Let $(X,\cA,\mu)$ be a measure space with nonnegative measure $\mu$ (finite or with values in $[0,\infty]$).
A set of functions $\cF \subseteq L^1(\mu)$ is called \emph{uniformly integrable} if 
\[
    \lim_{C\to +\infty} \sup_{f \in \cF} \int_{|f|>C} |f|\, d\mu = 0.
\]

\begin{theorem}[De la Vall\'ee Poussin's criterion {\cite[Theorem 4.5.9]{Bogachev:2007aa}}] \label{thm:VP}
    Let $\mu$ be a finite nonnegative measure. 
    A family $\cF$ of $\mu$-integrable functions is uniformly integrable if and only if 
    there exists a nonnegative increasing function $G$ on $[0,\infty)$ such that
    \[
        \lim_{t \to +\infty} \frac{G(t)}{t} = \infty 
        \quad \text{ and }
        \quad
        \sup_{f \in \cF} \int G(|f(x)|)\, \mu(dx) < \infty.
    \]
    In such a case, one can choose a convex increasing function $G$.
\end{theorem}

Recall that a sequence of complex valued measurable functions $f_n$ on $X$ is said to 
\emph{converge in measure} to $f$ 
if, for all $\ep >0$, 
\[
    \mu(\{x \in X : |f(x) - f_n(x)| \ge \ep\}) \to 0 \quad \text{ as } n \to \infty.
\]

\begin{theorem}[Vitali's convergence theorem {\cite[Theorem 4.5.4]{Bogachev:2007aa}}] \label{thm:Vitali}
    Let $\mu$ be a finite measure.
    Suppose that $f$ is a $\mu$-measurable function and $\{f_n\}$ is a sequence of $\mu$-integrable functions.
    Then the following assertions are equivalent:
    \begin{enumerate}
        \item $f_n \to f$ in measure and $\{f_n\}$ is uniformly integrable.
        \item $f$ is integrable and $f_n \to f$ in $L^1(\mu)$.
    \end{enumerate}
\end{theorem}

\subsection{Proof of \Cref{prop:lefttr}}
\label{ssec:A2}

We follow \cite{LlaveObaya99} in which H\"older--Lipschitz spaces are treated; 
the proofs simplify considerably.

   Let $U \subseteq \R^m$ and $V \subseteq \R^\ell$ be open, bounded, and convex.
   For brevity, we will simply write $\| \cdot\|_k$ for the $C^k$ norm from \eqref{eq:Cknorm}.

   \begin{lemma} \label{lem:linear}
    Let $\ps : \R^\ell \to \R^p$ be a linear map. 
    Then $\ps_* : C^k(\ol U,\R^\ell) \to C^k(\ol U,\R^p)$ is linear and continuous
    with operator norm $\|\ps_*\|= \|\ps\|$.
   \end{lemma}
 
\begin{proof}
   For $k=0$ and $\vh \in C^0(\ol U,\R^\ell)$, $\|\ps \o \vh\|_0 \le \|\ps\|\|\vh\|_0$. 
   Let $k\ge 1$ and $\vh \in C^k(\ol U,\R^\ell)$.
   Then $d(\ps \o \vh) = \ps \o d\vh$ and the statement follows by induction.
\end{proof}

\begin{lemma} \label{lem:bilinear}
    Let $\ps : \R^{\ell_1} \times \R^{\ell_2} \to \R^p$ be a bilinear map.
    Then $\ps_* : C^k(\ol U,\R^{\ell_1}) \times C^k(\ol U,\R^{\ell_2}) \to C^k(\ol U,\R^p)$
    is bilinear and continuous with $\|\ps_*\| \le C(k) \|\ps\|$.
\end{lemma}

\begin{proof}
    For $k=0$ and $\vh_i \in C^0(\ol U,\R^{\ell_i})$, $i=1,2$,
    \[
        \|\ps_*(\vh_1,\vh_2)\|_0 \le 2 \|\ps\| \|\vh_1\|_0 \|\vh_2\|_0.
    \]
    For $k \ge 1$ and $\vh_i \in C^1(\ol U,\R^{\ell_i})$, $i=1,2$,
    \[
        d(\ps_*(\vh_1,\vh_2)) = (\ps_1)_*(d\vh_1,\vh_2) + (\ps_2)_*(\vh_1,d\vh_2),
    \]
where $\ps_1$ and $\ps_2$ are the bilinear maps 
\begin{align*}
    \ps_1 &: L(\R^m,\R^{\ell_1}) \times \R^{\ell_2} \to L(\R^m,\R^p), \quad (h_1,y_2) \mapsto (x \mapsto \ps(h_1(x),y_2)),
    \\
    \ps_2 &: \R^{\ell_1} \times L(\R^m,\R^{\ell_2}) \to L(\R^m,\R^p), \quad (y_1,h_2) \mapsto (x \mapsto \ps(y_1,h_2(x))).
\end{align*}
Thus the statement follows by induction.
\end{proof}

\begin{lemma} \label{lem:right}
   Let $\vh \in C^k(\ol U,V)$. 
   Then $\vh^* : C^k(\ol V,\R^p) \to C^k(\ol U,\R^p)$, $\vh^*(\ps) := \ps \o \vh$, is well-defined, linear, and continuous.
   More precisely, for $\ps \in C^k(\ol V,\R^p)$,
\[
    \|\vh^*(\ps)\|_k  \le C(k)\, \|\ps\|_k (1+\|\vh\|_k)^k. 
\]   
\end{lemma}

\begin{proof}
   The statement for $k=0$ is clear. 
   Now let us proceed by induction on $k$ and assume that the statement holds for $k-1$.
   By \Cref{lem:bilinear} and the induction hypothesis, 
   \begin{align*}
       \|d(\ps \o \vh)\|_{k-1} &\le C\, \|d\ps\o \vh\|_{k-1} \|d\vh\|_{k-1}  
       \\
                               &\le C_1\, \|d\ps\|_{k-1}(1 + \|\vh\|_{k-1})^{k-1} \|d\vh\|_{k-1}   
                               \\
                               &\le C_1\, \|\ps\|_k (1 + \|\vh\|_k)^k
   \end{align*}
   and the assertion for $k$ follows easily.
\end{proof}

Now we are ready to prove \Cref{prop:lefttr} which is reformulated in the following lemma.

\begin{lemma} \label{lem:left}
   Let $\ps \in C^{k+1}(\ol V,\R^p)$. Then 
   $\ps_* : C^k(\ol U,V) \to C^k(\ol U,\R^p)$, $\ps_* (\vh) := \ps \o \vh$, is well-defined and continuous.
   More precisely, for $\vh_1,\vh_2$ in a bounded subset $B$ of $C^k(\ol U,V)$,
   \[
       \| \ps_*(\vh_1) - \ps_*(\vh_2)\|_k \le C \, \|\ps\|_{k+1} \|\vh_1-\vh_2\|_k,
   \]
   where $C=C(k,B)$.
\end{lemma}

\begin{proof}
    For $k=0$, we have 
    \[
        \| \ps\o \vh_1 - \ps\o \vh_2\|_0 \le C\, \|d\ps\|_0 \|\vh_1 -\vh_2\|_0 \le C\, \|\ps\|_1  \|\vh_1 -\vh_2\|_0.
    \]
    Let us proceed by induction on $k$ and assume that the statement holds for $k-1$.
    We have
    \begin{align*}
        \MoveEqLeft \| d(\ps \o \vh_1) - d(\ps \o \vh_2)\|_{k-1} 
        \\
        &= \|(d\ps \o \vh_1).(d\vh_1 -d\vh_2) - (d\ps \o \vh_2 -d\ps \o \vh_1).d\vh_2\|_{k-1} 
        \\
        &\le C\, \|d\ps \o \vh_1\|_{k-1} \|d\vh_1 -d\vh_2\|_{k-1} + C\, \| d\ps \o \vh_2 -d\ps \o \vh_1\|_{k-1}\|d\vh_2\|_{k-1},
    \end{align*}
    where we used \Cref{lem:bilinear} in the last step.
    By \Cref{lem:right},
    \begin{align*}
        \|d\ps \o \vh_1\|_{k-1} \|d\vh_1 -d\vh_2\|_{k-1} &\le C_1\, \|d\ps\|_{k-1} (1+ \|\vh_1\|_{k-1})^{k-1}  \|\vh_1 -\vh_2\|_{k}
        \\
&\le C_1\, \|\ps\|_{k} (1+ \|\vh_1\|_{k-1})^{k-1}  \|\vh_1 -\vh_2\|_{k}.
    \end{align*}
    By the induction hypothesis,
    \begin{align*}
        \| d\ps \o \vh_2 -d\ps \o \vh_1\|_{k-1}\|d\vh_2\|_{k-1} &\le  C_2\,  \|d\ps\|_{k} \|\vh_1 - \vh_2 \|_{k-1} \|\vh_2\|_{k}
        \\
&\le  C_2\,  \|\ps\|_{k+1} \|\vh_1 - \vh_2 \|_{k} \|\vh_2\|_{k}.
    \end{align*}
    Now the statement follows easily.
\end{proof}

\subsection*{Acknowledgement}
We are grateful to Antonio Lerario for posing the question about the continuity of the solution map.


\def\cprime{$'$}
\providecommand{\bysame}{\leavevmode\hbox to3em{\hrulefill}\thinspace}
\providecommand{\MR}{\relax\ifhmode\unskip\space\fi MR }
\providecommand{\MRhref}[2]{%
  \href{http://www.ams.org/mathscinet-getitem?mr=#1}{#2}
}
\providecommand{\href}[2]{#2}

\end{document}